\documentclass[a4paper, 11pt, reqno]{amsart}
\usepackage{amsthm,amssymb,amsmath,enumerate,graphicx,psfrag,enumitem,mathrsfs}

\usepackage[margin=2.5cm]{geometry}
\usepackage{hyperref}
\usepackage{array}
\hypersetup{colorlinks=true,
   citecolor=blue,
   filecolor=blue,
   linkcolor=blue,
   urlcolor=blue
}

\usepackage{tikz}
\usepackage[textsize=footnotesize]{todonotes}

\usepackage{algorithm}
\usepackage{tikz}

\allowdisplaybreaks

\usetikzlibrary{backgrounds}

\newtheorem{definition}{Definition}[section]
\newtheorem{claim}{Claim}
\newtheorem{proposition}[definition]{Proposition}
\newtheorem{theorem}[definition]{Theorem}
\newtheorem{corollary}[definition]{Corollary}
\newtheorem{lemma}[definition]{Lemma}

\newtheorem*{claim*}{Claim}

\numberwithin{equation}{section}

\newcommand{\comment}[1]{}

\newcommand{\cA}{\mathcal{A}}

\newcommand{\cK}{\mathcal{K}}

\newcommand{\cQ}{\mathcal{Q}}

\newcommand{\cF}{\mathcal{F}}

\newcommand{\cC}{\mathcal{C}}

\newcommand{\cH}{\mathcal{H}}
\newcommand{\cM}{\mathcal{M}}

\newcommand{\den}{{\rm den}}

\renewcommand{\epsilon}{\varepsilon}

\newcommand{\COMMENT}[1]{}

\floatstyle{plain}
\restylefloat{figure}

\title[A bandwidth theorem for approximate decompositions]{A bandwidth theorem for approximate decompositions}
\author{Padraig Condon, Jaehoon Kim, Daniela K\"uhn and Deryk Osthus}

\thanks{The research leading to these results was partially supported by the EPSRC, grant no. EP/N019504/1 (D.~K\"uhn), 
and by the Royal Society and the Wolfson Foundation (D.~K\"uhn).
The research was  also partially supported by the European Research Council under the European Union's Seventh Framework Programme (FP/2007--2013) / ERC Grant 306349 (J.~Kim and D.~Osthus). }

\date{\today}

\begin{document}

\date{\today}

\restylefloat{figure} 
\maketitle

\begin{abstract}
We provide a degree condition on a regular $n$-vertex graph $G$ which ensures the 
existence of a near optimal packing of any family $\mathcal H$ of bounded
degree $n$-vertex $k$-chromatic separable graphs into $G$.
In general, this degree condition is best possible.

Here a graph is separable if it has a sublinear separator whose removal results in a set of components of sublinear size. Equivalently, the separability condition can be replaced by that of having small bandwidth.
Thus our result can be viewed as a version of the bandwidth theorem of 
B\"ottcher, Schacht and Taraz in the setting of approximate decompositions.

More precisely, let $\delta_k$ be the infimum over all $\delta\ge 1/2$ ensuring an approximate $K_k$-decompo\-si\-tion of any sufficiently large regular $n$-vertex graph $G$ of degree at least $\delta n$. 
Now suppose that $G$ is an $n$-vertex graph which is close to $r$-regular for some $r \ge (\delta_k+o(1))n$ and suppose that $H_1,\dots,H_t$ is a sequence of bounded degree $n$-vertex $k$-chromatic separable graphs with $\sum_i e(H_i) \le (1-o(1))e(G)$. We show that there is an edge-disjoint packing of 
$H_1,\dots,H_t$ into~$G$. 

If the $H_i$ are bipartite, then $r\geq (1/2+o(1))n$ is sufficient.
In particular, this yields an approximate version of the tree packing conjecture 
in the setting of regular host graphs $G$ of high degree.
Similarly, our result implies approximate versions of the Oberwolfach problem,
the Alspach problem and the existence of resolvable designs in the setting of
regular host graphs of high degree.


\end{abstract}

\section{Introduction}\label{sec: intro} 
Starting with Dirac's theorem on Hamilton cycles, a successful research direction in extremal combinatorics has been to find
appropriate minimum degree conditions on a graph $G$
which guarantee the existence of a copy of a (possibly spanning) graph $H$
as a subgraph.
On the other hand, several important questions and results in design theory ask for the existence of 
a decomposition of $K_n$ into edge-disjoint copies of a (possibly spanning) graph $H$,
or more generally into a suitable family of graphs $H_1,\dots,H_t$.

Here, we combine the two directions: rather than finding just a single spanning graph $H$ in a dense graph $G$, we seek (approximate)
decompositions of a dense regular graph $G$ into edge-disjoint copies of spanning sparse graphs $H$.
A specific instance of this is the recent proof of the Hamilton decomposition
conjecture and the $1$-factorization conjecture for large $n$~\cite{CKOLT}: 
the former states that for  $r \ge   \lfloor n/2 \rfloor $, every $r$-regular $n$-vertex graph $G$  has a decomposition
into Hamilton cycles and at most one perfect matching,
the latter provides the corresponding threshold for decompositions into perfect matchings. 
In this paper, we restrict ourselves to approximate decompositions, but
achieve asymptotically best possible results for a much wider class of graphs than matchings and Hamilton cycles.

\subsection{Previous results: degree conditions for spanning subgraphs}
Minimum degree conditions for spanning subgraphs have been obtained mainly for
(Hamilton) cycles, trees, factors and bounded degree graphs.
We now briefly discuss several of these.
Recall that Dirac's theorem states that any $n$-vertex graph $G$ with minimum degree at least $n/2$ contains a Hamilton cycle.
More generally, Abbasi's proof~\cite{A} of the El-Zahar conjecture determines the minimum degree threshold for the existence
of a copy of $H$ in $G$ where $H$ is a spanning union of vertex-disjoint cycles
(the threshold turns out to be $\lfloor (n+odd_H)/2\rfloor$, where $odd_H$ denotes the number of 
odd cycles in $H$).

Koml\'os, S\'ark\"ozy and  Szemer\'edi \cite{KSS} proved a conjecture
of Bollob\'as by showing that a minimum degree degree of $n/2+o(n)$ guarantees
 every bounded degree $n$-vertex tree as a subgraph
 (this was later strengthened in~\cite{KSS2, CLGS, J}).
 
 An \emph{$F$-factor} in a graph $G$ is a set of vertex-disjoint copies of $F$
 covering all vertices of $G$.
The Hajnal-Szemer\'{e}di theorem~\cite{HS} implies that the minimum degree
threshold for the existence of a $K_k$-factor is $(1- 1/k)n$.
This was generalised to $k$th powers of Hamilton cycles by 
Koml\'os, S\'ark\"ozy and  Szemer\'edi~\cite{KSS3}.
 The threshold for arbitrary $F$-factors was determined by K\"uhn and Osthus~\cite{KO09}, 
 and is given by $(1-c(F))n+O(1)$, where $c(F)$  satisfies $1/\chi(F) \le c(F) \le 1/(\chi(F)-1)$
 and can be determined explicitly (e.g.~$c(C_5)=2/5$, in accordance with 
 Abbasi's result).

A far-reaching generalisation of the Hajnal-Szemer\'{e}di theorem~\cite{HS}
would be provided by the  Bollob\'as-Catlin-Eldridge (BEC) conjecture.
This would imply that every $n$-vertex graph $G$ of minimum degree
at least $(1-1/(\Delta+1))n$ contains every $n$-vertex graph $H$ of maximum
degree at most $\Delta$ as a subgraph.
Partial results include the proof for 
$\Delta=3$ and large $n$ by Csaba, Shokoufandeh and Szemer\'edi~\cite{CSS} and bounds for
 large $\Delta$ by Kaul, Kostochka and Yu~\cite{KKY}.

Bollob\'as and Koml\'os conjectured that one can improve 
on the BEC-conjecture for
graphs $H$ with a linear structure: any $n$-vertex graph $G$ with minimum degree at least $(1-1/k+ o(1))n$ contains a copy of every $n$-vertex $k$-chromatic graph $H$ with bounded maximum degree and small bandwidth. 
Here an $n$-vertex graph $H$ has \emph{bandwidth} $b$ if there exists an ordering $v_1,\dots, v_n$ of $V(H)$ such that all edges $v_iv_j\in E(H)$ satisfy $|i-j|\leq b$.
Throughout the paper, by $H$ being $k$-chromatic we mean $\chi(H) \le k.$
This conjecture was resolved by the bandwidth theorem of B\"ottcher, Schacht and Taraz \cite{BST}. 
Note that while this result is essentially best possible when considering the 
class of $k$-chromatic graphs as a whole
(consider e.g.~$K_k$-factors), 
the results in~\cite{A, KO09}
mentioned above
show that there are many graphs $H$ for which the actual threshold is significantly
smaller (e.g.~the $C_5$-factors mentioned above).

The notion of bandwidth is related to the concept of separability:
 An $n$-vertex graph $H$ is said to be \emph{$\eta$-separable} if there exists a set $S$ of at most $\eta n$ vertices such that every component of $H\setminus S$ has size at most $\eta n$. 
 We call such a set an \emph{$\eta$-separator} of $H$.
In general, the notion of having small bandwidth is more restrictive than that of 
being separable. However, for graphs with bounded maximum degree, 
it turns out that these notions are actually equivalent (see \cite{BPTW}).

\subsection{Previous results: (approximate) decompositions into large graphs}

We say that a collection $\mathcal{H}=\{H_1,\dots, H_s\}$ of graphs \emph{packs} into $G$ if there exist pairwise edge-disjoint copies of $H_1, \dots, H_s$ in $G$. 
In cases where $\mathcal{H}$ consists of copies of a single graph $H$ we refer to this packing as an $H$-$\emph{packing}$ in $G$.
If $\cH$ packs into $G$ and $e(\cH)=e(G)$ (where $e(\cH)=\sum_{H\in \cH} e(H)$), then we say that $G$ has a \emph{decomposition} into $\cH$.
Once again, if $\mathcal{H}$ consists of copies of a single graph $H$, we refer to this as an $H$-decomposition of $G$.
Informally, we refer to a packing which covers almost all edges of the host graph $G$ as an approximate decomposition.

As in the previous section, most attention so far has focussed on (Hamilton) cycles, trees, factors,
and graphs of bounded degree.
Indeed, a classical construction of Walecki going back to the 19th century 
guarantees a decomposition of $K_n$ into Hamilton cycles whenever $n$ is odd.
As mentioned earlier, this was extended to Hamilton decompositions of 
regular graphs~$G$ of high degree by Csaba, K\"uhn, Lo, Osthus and Treglown \cite{CKOLT}
(based on the existence of Hamilton decompositions in robustly  expanding graphs
proved in~\cite{KO}).
A different generalisation of Walecki's construction is given by the Alspach problem,
which asks for a decomposition of $K_n$ into cycles of given length.
This was recently resolved by Bryant, Horsley and Petterson~\cite{BHP}.

A further famous open problem in the area  is the tree packing conjecture of Gy\'arf\'as and Lehel, which says that for any collection  $\mathcal{T}=\{T_1,\dots, T_n\}$ of trees with $|V(T_i)|=i$, the complete graph $K_n$ has a decomposition into $\mathcal{T}$.
This was recently proved by Joos, Kim, K\"uhn and Osthus \cite{JKKO} for
the case where $n$ is large and each $T_i$ has bounded degree. 
The crucial tool for this was the blow-up lemma for approximate decompositions
of $\varepsilon$-regular graphs $G$ by Kim, K\"uhn, Osthus and Tyomkyn \cite{KKOT}.
In particular, this lemma implies that if
$\mathcal{H}$ is a family of bounded degree $n$-vertex graphs with $e(\mathcal{H}) \le (1-o(1))\binom{n}{2}$, then $K_n$
has an approximate decomposition into $\mathcal{H}$.
This generalises earlier results of
B\"ottcher, Hladk\`y, Piguet and Taraz~\cite{BHPT} on tree packings,
as well as results of
Messuti, R\"odl and Schacht \cite{MRS} and Ferber, Lee and Mousset \cite{FLM}
on packing separable graphs.
Very recently, Allen, B\"ottcher, Hladk\`y and Piguet~\cite{ABHP} were able to show that one 
can in fact find an approximate decomposition of $K_n$ into $\mathcal{H}$
provided that the graphs in $\mathcal H$ have bounded degeneracy and 
maximum degree $o( n/\log n)$. 
This implies an approximate version of the tree packing conjecture when the 
trees have maximum degree $o( n/\log n)$.
The latter improves a bound of Ferber and Samotij~\cite{FS} which follows from their work on packing 
(spanning) trees in random graphs.

An important type of  decomposition of $K_n$ is given by resolvable designs:
a resolvable $F$-design consists of a decomposition into $F$-factors.
Ray-Chaudhuri and Wilson~\cite{RW} proved the existence of resolvable $K_k$-designs in $K_n$ (subject to the necessary divisibility conditions being satisfied).
This was generalised to arbitrary $F$-designs by Dukes and Ling~\cite{DL}.

\subsection{Main result: packing separable graphs of bounded degree}
Our main result provides a degree condition which ensures that $G$ has an approximate
decomposition into $\cH$ for any collection $\cH$ of $k$-chromatic $\eta$-separable graphs of bounded degree.
As discussed below, our degree condition is best possible in general (unless one has additional information about
the graphs in $\cH$).
By the remark at the end of Section $1.1$ earlier, one can replace the condition of being $\eta$-separable by that of having bandwidth at most $\eta n$ in Theorem~\ref{thm:main}.
Thus our result implies a version of the bandwidth theorem of \cite{BST} in the setting of approximate decompositions.

 To state our result, we first introduce the approximate $K_k$-decomposition 
 threshold $\delta_k^{\text{reg}}$ for regular graphs.
\begin{definition}[Approximate $K_k$-decomposition threshold for regular graphs]\label{def: delta k}
For each $k\in \mathbb{N}\backslash \{1\}$, let $\delta_k^{\text{reg}}$ be the infimum over all $\delta \ge 0$ satisfying the following:  for any $\epsilon>0$, there exists $n_0\in\mathbb{N}$ such that for all $n\geq n_0$ and $r \ge \delta n$ every $n$-vertex $r$-regular graph $G$  has a $K_k$-packing consisting of at least $(1-\epsilon)e(G)/e(K_k)$ copies of $K_k$.
\end{definition}
\COMMENT{The purpose of this parameter is to claim that our result is sharp.
If we take a number $\delta' < \delta_k^{\text{reg}}$, then there exists an almost-$\delta'n$-regular graph $G$ which does not have an approximate decomposition into $K_k$, thus our result is sharp. (Need only to define $\delta_k^{\text{reg}}$ for
regular graphs (instead of close to regular graphs) here since in the proof of Lemma~\ref{lem: factor decomposition} we only apply it to a regular graph.)

However, if we take a number $\delta' < \delta_k^{0+}$, then we don't know whether there exists an almost $\delta'n$-regular graph $G$ without an approximate decomposition into $K_k$. Since $\delta'< \delta_k^{0+}$, we know that there exists a graph $G'$ with $\delta(G') \geq \delta'n$ which does not have an approximate decomposition into $K_k$. However, $G'$ might be very far from regular. So if we just use $\delta^{0+}_k$, then it is not easy to claim that our result is sharp.
}
Roughly speaking, we will pack $k$-chromatic graphs $H$ into regular host graphs $G$ of degree at least $\delta_k^{\text{reg}} n$.
Actually it turns out that it suffices to assume that $H$ is `almost' $k$-chromatic in the sense that $H$ has a ($k+1$)-colouring where one colour
is used only rarely. More precisely, we say that $H$ is \emph{$(k,\eta)$-chromatic} if there exists a proper colouring of the graph $H'$ obtained from $H$ by deleting all its isolated vertices with $k+1$ colours such that one of the colour classes has size at most $\eta |V(H')|$. A similar feature is also present in~\cite{BST}.

\begin{theorem}\label{thm:main}
For all $\Delta, k \in \mathbb{N}\backslash\{1\}$, $0<\nu<1$ and 
$\max\{1/2, \delta_k^{\text{reg}}\} < \delta \leq 1$,  there exist $\xi, \eta>0$ and $n_0 \in \mathbb{N}$ such that for all $n \ge n_0$ the following holds. 
Suppose that $\mathcal{H}$ is a collection of $n$-vertex $(k,\eta)$-chromatic $\eta$-separable graphs and $G$ is an $n$-vertex graph such that
\begin{enumerate}[label=(\roman*)]
\item $(\delta - \xi)n \le \delta(G) \le \Delta(G) \le (\delta + \xi)n,$
\item $\Delta(H) \le \Delta$ for all $H \in \mathcal{H},$
\item $e(\mathcal{H}) \le (1-\nu)e(G).$
\end{enumerate}
Then $\mathcal{H}$ packs into $G$.
\end{theorem}

Note that our result holds for any minor-closed family $\mathcal{H}$ of $k$-chromatic bounded degree graphs by the separator theorem of Alon, Seymour and Thomas \cite{AST}.
Moreover, note that since $\mathcal{H}$ may consist e.g.~of Hamilton cycles, the condition
that $G$ is close to regular is clearly necessary.
 Also, the condition $\max\{1/2, \delta_k^{\text{reg}}\} < \delta$ is necessary.
 To see this, if $\delta_k^{\text{reg}} \leq 1/2$ (which holds if $k=2$), then we consider $K_{n/2-1,n/2+1}$ which does not even contain a single perfect matching, let alone an approximate decomposition into perfect matchings. 
If $\delta_k^{\text{reg}} > 1/2$ (which holds if $k\geq 3$), then for any $\delta < \delta_k^{\text{reg}}$, the definition of $\delta_k^{\text{reg}}$ ensures that there exist arbitrarily large regular graphs $G$ of degree at least $\delta n$ without an approximate decomposition into copies of $K_k$. 
As a disjoint union of a single copy of $K_k$ with $n-k$ isolated vertices satisfies (ii), this shows that the condition of $\max\{1/2, \delta_k^{\text{reg}}\} < \delta$ is sharp when considering the class of all $k$-chromatic separable graphs (though as in the case of embedding a single copy of some $H$ into $G$, it may be possible to improve the degree bound for certain families $\cH$).


To obtain explicit estimates for $\delta_k^{\text{reg}}$, we also introduce the approximate $K_k$-decomposition threshold $\delta_k^{0+}$ for graphs of large minimum degree.

\begin{definition}[Approximate $K_k$-decomposition threshold]
For each $k\in \mathbb{N}\backslash \{1\}$, let $\delta^{0+}_k$ be the infimum over all $\delta \ge 0$ satisfying the following:
for any $\epsilon>0$, there exists $n_0 \in \mathbb{N}$ such that any $n$-vertex graph $G$ with $n\geq n_0$ and $\delta(G)\geq \delta n$ has a $K_k$-packing consisting of at least $(1-\epsilon)e(G)/e(K_k)$ copies of $K_k$.%
\end{definition}

It is easy to see that $\delta^{\text{reg}}_2=\delta^{0+}_2=0$ and $\delta_k^{\text{reg}} \leq \delta^{0+}_k$.
The value of $\delta_k^{0+}$ has been subject to much attention recently:
one reason is that by results of~\cite{BKLO,GKLOM}, for $k \ge 3$ the approximate
decomposition threshold $\delta_k^{0+}$ is equal to the analogous threshold
$\delta_k^{\rm dec}$ which ensures a `full' $K_k$-decomposition
of any $n$-vertex graph $G$ with $\delta(G) \ge (\delta_k^{\rm dec}+o(1))n$
which satisfies the necessary divisibility conditions.
A beautiful conjecture (due to Nash-Williams in the triangle case and 
Gustavsson in the general case) would imply  that  $\delta_k^{\rm dec}=1-1/(k+1)$ for $k \ge 3$.
On the other hand for $k\geq 3$, it is easy to modify a well-known construction (see Proposition \ref{extremal example}) to show that $\delta_k^{\text{reg}}\geq 1- 1/(k+1)$.
Thus the conjecture would imply that $\delta_k^{\text{reg}} = \delta_k^{0+} = \delta_k^{\rm dec} = 1 - 1/(k+1)$ for $k \ge 3.$ A result of Dross~\cite{Dr} implies that $\delta_3^{0+} \leq 9/10$, and a very recent result of 
Montgomery~\cite{M17} implies that
$\delta_k^{0+} \leq 1- 1/(100k)$ (see Lemma~\ref{lem: delta* value}).
With these bounds, the following corollary is immediate.
\begin{corollary}
For all $\Delta, k  \in \mathbb{N}\backslash\{1\}$ and  $0<\nu ,\delta< 1,$   there exist $\xi >0$ and $n_0 \in \mathbb{N}$ such that for $n \ge n_0$ the following holds for every $n$-vertex graph $G$ with $$(\delta - \xi)n \le \delta(G) \le \Delta(G) \le (\delta + \xi)n.$$
\begin{enumerate}[label=(\roman*)]
\item Let $\mathcal{T}$ be a collection of trees such that for all $T \in \mathcal{T}$ we have $|T| \le n$ and $\Delta(T) \le \Delta$. 
Further suppose $\delta>1/2$ and $e(\mathcal{T}) \le (1-\nu)e(G)$. 
Then $\mathcal{T}$ packs into $G.$ 
\item Let $F$ be an $n$-vertex graph consisting 
of a union of vertex-disjoint cycles and let $\mathcal{F}$ be a collection of copies of $F$.
Further suppose $\delta> 9/10$ and $e(\mathcal{F}) \le (1-\nu)e(G)$.
Then $\mathcal{F}$ packs into $G$.
\item Let $\mathcal{C}$ be a collection of cycles, each on at most $n$ vertices.
Further suppose $\delta> 9/10$  and  $e(\mathcal{C}) \le (1-\nu)e(G).$
Then $\mathcal{C}$ packs into $G$.
\item Let $n$ be divisible by $k$ and let $\mathcal{K}$ be a collection of $n$-vertex $K_k$-factors.
Further suppose $\delta>1 - 1/(100k)$ and $e(\mathcal{K}) \le (1-\nu)e(G).$
Then $\mathcal{K}$ packs into $G$.

\end{enumerate}
\end{corollary}
Note that (i) can be viewed as an approximate version of the tree packing conjecture
in the setting of dense (almost) regular graphs.
In a similar sense, (ii) relates to the Oberwolfach conjecture,
(iii) relates to the Alspach problem and (iv) relates to the existence of resolvable designs in graphs.

Moreover, the feature that Theorem~\ref{thm:main} allows us to efficiently pack $(k,\eta)$-chromatic graphs (rather than $k$-chromatic graphs)
gives several additional consequences, for example:
if the cycles of $F$ in (ii) are all sufficiently long, then we can replace the condition `$\delta> 9/10$'
by `$\delta> 1/2$'.

If we drop the assumption of being $G$ close to regular, then one can still ask for the 
size of the largest packing of bounded degree separable graphs.
For example, it was shown in~\cite{CKOLT} that every sufficiently large graph $G$ with 
$\delta(G) \ge n/2$ contains at least $(n-2)/8$ edge-disjoint Hamilton cycles.
The following result gives an approximate answer to the above question in the case when $\mathcal{H}$ consists of (almost) bipartite graphs.

\begin{theorem}\label{thm:2}
For all $\Delta \in \mathbb{N}$, $1/2 <\delta \le 1$ and $\nu>0$, there exist $\eta>0$ and $n_0 \in \mathbb{N}$ such that for all $n \ge n_0$ the following holds. 
Suppose that $\mathcal{H}$ is a collection of $n$-vertex $(2,\eta)$-chromatic $\eta$-separable graphs and $G$ is an $n$-vertex graph such that
\begin{enumerate}[label=(\roman*)]
\item $\delta(G) \geq \delta n,$
\item $\Delta(H) \le \Delta$ for all $H \in \mathcal{H},$
\item $e(\mathcal{H}) \leq \frac{(\delta+ \sqrt{2\delta-1} -\nu)n^2}{4}.$
\end{enumerate}
Then $\mathcal{H}$ packs into $G$.
\end{theorem}

The result in general cannot be improved:
Indeed, for $\delta>1/2$ the number of edges of the densest regular spanning subgraph of $G$
is close to $(\delta+ \sqrt{2\delta-1} )n^2/4$ (see~\cite{CKO}).
So the bound in~(iii) is asymptotically optimal e.g.~if $n$ is even and
$\cH$ consists of Hamilton cycles.
We discuss the very minor modifications to the proof of Theorem~\ref{thm:main}
which give Theorem~\ref{thm:2}  at the end of Section~\ref{sec: main proof}.


We raise the following open questions:
\begin{itemize}

\item
We conjecture that the error term $\nu e(G)$ in condition (iii) of 
Theorem~\ref{thm:main} can be improved.
Note that it cannot be completely removed unless one assumes some 
divisibility conditions on $G$. However, even additional divisibility conditions will not 
always ensure a `full' decomposition under the current degree conditions: indeed, for $C_4$, the minimum degree threshold which guarantees a $C_4$-decomposition of a graph $G$ is close to $2n/3$, and the extremal example
is close to regular (see~\cite{BKLO} for details, more generally, the decomposition threshold
of an arbitrary bipartite graph is determined in~\cite{GKLOM}).

\item It would be interesting to know whether the condition on separability can be 
 omitted. Note however, that if we do not assume separability, then the degree
 condition may need to be strengthened.
 
\item It would be interesting to know whether one can relax the maximum degree condition in assumption (ii) of Theorem~\ref{thm:main},
e.g.~for the class of trees.

\item Given the recent progress on the existence of decompositions and designs in the hypergraph setting
and the corresponding minimum degree thresholds~\cite{K,GKLO1,GKLO2}, it would be interesting
to generalise (some of) the above results to hypergraphs.

\end{itemize}
 
Our main tool in the proof of Theorem~\ref{thm:main} will be the recent blow-up lemma for approximate decompositions by Kim, K\"uhn, Osthus and Tyomkyn~\cite{KKOT}: roughly speaking, given a set $\mathcal{H}$ of $n$-vertex bounded degree graphs and an $n$-vertex graph $G$ with $e(\mathcal{H}) \leq (1-o(1))e(G)$ consisting of super-regular pairs, it guarantees a packing of $\mathcal{H}$ in $G$ (such super-regular pairs arise from applications of Szemer\'edi's regularity lemma). Theorem~\ref{Blowup} gives the precise statement of the special case that we shall apply (note that the original blow-up lemma of  Koml\'os, S\'ark\"ozy and  Szemer\'edi \cite{KSS-blowup} corresponds to the case where $\mathcal{H}$ consists of a single graph).

Subsequently, Theorem~\ref{thm:main} has been used as a key tool in the resolution of the Oberwolfach problem
in~\cite{GJKKO}. This was posed by Ringel in 1967, given an $n$-vertex graph $H$ consisting of vertex-disjoint cycles, it
asks for a decomposition of $K_n$ into copies of $H$ (if $n$ is odd). 
In fact, the results in~\cite{GJKKO} go considerably beyond the setting of the Oberwolfach problem, and imply e.g.~a positive resolution also to the 
Hamilton-Waterloo problem.

\section{Outline of the argument}\label{sec: outline}

Consider a given collection $\mathcal{H}$ of $k$-chromatic $\eta$-separable graphs with
bounded degree and a given almost-regular graph $G$ as in Theorem \ref{thm:main}. 
We wish to pack $\mathcal{H}$ into $G$.
The approach will be to decompose $G$ into a bounded number of highly structured subgraphs $G_t$ and partition $\cH$ into a bounded number of collections $\cH_t$. 
We then aim to pack each $\cH_t$ into $G_t$. 
As described below, for each $H \in \cH_t,$ most of the edges will be embedded via the blow-up lemma for approximate decompositions proved in \cite{KKOT}. 

As a preliminary step, we first apply Szemer\'edi's regularity lemma (Lemma~\ref{Szemeredi}) to $G$ to obtain a reduced multigraph $R$ which is almost regular.
Here each edge $e$ of $R$ corresponds to a bipartite $\epsilon$-regular subgraph of $G$ and the density of these subgraphs does not depend on $e$.
We can then apply a result of Pippenger and Spencer on the chromatic index of regular hypergraphs and the definition of $\delta_k^{\text{reg}}$ to find an approximate decomposition of the reduced multigraph $R$ into almost $K_k$-factors.
 More precisely, we find a set of edge-disjoint copies of almost $K_k$-factors covering almost all edges of $R$, where an almost $K_k$-factor is a set of vertex-disjoint copies of $K_k$ covering almost all vertices of $R$.
 This approximate decomposition translates into the existence of an approximate decomposition of $G$ into `(almost-)$K_k$-factor blow-ups'. 
Here a $K_k$-factor blow-up consists of a bounded number of clusters $V_1,\dots, V_{kr}$ where
 each pair $(V_i, V_j)$ with $\lfloor (i-1)/k\rfloor = \lfloor (j-1)/k\rfloor$ is $\epsilon$-regular of density $d$, and crucially $d$ does not depend on $i, j$. 
We wish to use the blow-up lemma for approximate decompositions (Theorem~\ref{Blowup}) to pack graphs into each $K_k$-factor blow-up.
 Ideally, we would like to split $\mathcal{H}$ into a bounded number of subcollections $\mathcal{H}_{t,s}$ and pack each $\mathcal{H}_{t,s}$ into a separate $K_k$-factor blow-up $G_{t,s}$, where the $G_{t,s}\subseteq G$ are all edge-disjoint.

There are several obstacles to this approach. 
The first obstacle is that (i) the $K_k$-factor blow-ups $G_{t,s}$ are not spanning.
 In particular, they do not contain the vertices in the exceptional set $V_0$ produced by the regularity lemma. 
On the other hand, if we aim to embed an $n$-vertex graph $H \in \cH$ into $G$, we must embed some vertices of $H$ into $V_0$.
 However, Theorem~\ref{Blowup} does not produce an embedding into vertices outside the $K_k$-factor blow-up.
The second obstacle is that (ii)  the $K_k$-factor blow-ups are not  connected, whereas $H$ may certainly be (highly) connected.
This is one significant difference to \cite{BST}, where the existence of a structure similar to a blown-up power of a Hamilton path in $R$ could be utilised for the embedding.
A third issue is that (iii) any resolution of (i) and (ii) needs to result in a `balanced' packing of the $H \in \cH$, i.e. the condition $e(\cH) \le (1-\nu)e(G)$ means that for most $x \in V(G)$ almost all their incident edges need to be covered. 

To overcome the first issue, we use the fact that $H$ is $\eta$-separable to choose a small separating set $S$ for $H$ and  consider the small components of $H-S$.
To be able to embed (most of) $H$ into the $K_k$-factor blow-up, we need to add further edges to each $K_k$-factor blow-up so that the resulting `augmented $K_k$-factor blow-ups' have strong connectivity properties. 
For this, we partition $V(G) \backslash V_0$ into $T$ disjoint `reservoirs' $Res_1,\dots, Res_T$, where $1/T \ll 1$. 
We will later embed some vertices of $H$ into $V_0$ using the edges between $Res_t$ and $V_0$ (see Lemma~\ref{lem: separable partition}).
Here we have to embed a vertex of $H$ onto $v \in V_0$ using only edges between $v$ and $Res_{t}$ because we do not have any control on the edges between $v$ and a regularity cluster $V_i$.
We explain the reason for choosing a partition into many reservoir sets
(rather than choosing a single small reservoir)  below.

We also decompose most of $G$ into graphs $G_{t,s}$ so that each 
$G_{t,s}$ has vertex set $V(G)\backslash (Res_t \cup V_0)$
and is a $K_k$-factor blow-up.
We then find sparse bipartite graphs $F_{t,s} \subseteq G$ connecting $Res_t$ with $G_{t,s}$, bipartite graphs $F'_t\subseteq G$ connecting $Res_t$ with $V_0$
as well as sparse graphs $G_t^* \subseteq G$ which provide
connectivity within $Res_t$ as well as between $Res_t$ and $G_{t,s}$. 
The fact that $G_{t,s}$ and $G_{t,s'}$ share the same reservoir for $s \neq s'$
permits us to choose the reservoir $Res_t$ to be significantly larger than $V_0$. Moreover, 
as $\bigcup Res_{t}$ covers all vertices in $V\setminus V_0$, if the graphs $F'_{t}$ are appropriately chosen, then almost all edges incident to the vertices in $V_0$ are available to be used at some stage of the packing process.
Our aim is to pack each $\cH_{t,s}$ into the `augmented' $K_k$-factor
blow-up $G_{t,s} \cup F_{t,s} \cup F'_t \cup G_t^*$. 
To ensure that the resulting packings can be combined into a packing of all of the graphs in $\cH$, we will
use the fact that the graphs $G_t:=   \bigcup_s (G_{t,s} \cup F_{t,s} ) \cup F'_t \cup G_t^*$ referred to in the first paragraph are edge-disjoint for different~$t$.

We now discuss how to find this packing of $\cH_{t,s}.$
Consider some $H \in \cH_{t,s}$. 
We first use the fact that $H$ is separable to find a partition of $H$ which reflects the structure of (the augmentation of) $G_{t,s}$  (see Section~\ref{sec: H structure}).
Then we construct an appropriate embedding $\phi_*$ of parts of each graph $H\in \cH_{t,s}$ into $Res_t \cup V_0$ which covers all vertices in $Res_t\cup V_0$
(this makes crucial use of the fact that $Res_t$ is much larger than $V_0$).
 Later we aim to use the blow-up lemma for approximate decompositions (Theorem~\ref{Blowup}) to find an embedding $\phi$ of the remaining vertices of $H$ into $V(G)\backslash(Res_t \cup V_0)$.
When we apply Theorem~\ref{Blowup}, we use its additional features: in particular, the ability to prescribe appropriate `target sets' for some of the vertices of $H$, to guarantee the consistency between the two embeddings $\phi_*$ and $\phi$.

An important advantage of the reservoir partition which helps us to overcome obstacle (iii) is the following: the blow-up lemma for approximate decompositions can achieve a near optimal packing, i.e. it uses up almost all available edges. 
This is far from being the case for the part of the embeddings that use $F_{t,s}$, $F'_t$ and $G_t^*$ to embed vertices into $Res_t \cup V_0$, where the edge usage might be comparatively `imbalanced' and `inefficient'. (In fact, we will try to avoid using these edges as much as possible in order to preserve the connectivity properties of these graphs. We will use
probabilistic allocations to avoid over-using any parts of $F_{t,s}$, $F'_t$ and $G_t^*$.)
However, since every vertex in $V(G_0) \backslash V_0$ is a reservoir vertex for only a small proportion of the embeddings, the resulting effect of these imbalances on the overall leftover degree of the vertices in $V(G_0) \backslash V_0$ is negligible.
For $V_0$, we will be able to assign only low degree vertices of each $H$ to ensure that there will always be edges of $F_t'$ available to embed their incident edges
(so the overall leftover degree of the vertices in $V_0$ may be large).

The above discussion motivates why we use many reservoir sets which cover all vertices in $V(G)\backslash V_0$, rather than using only one vertex set $Res_1$ for all $H \in \mathcal{H}$.
Indeed, if some vertices of $G$ only perform the role of reservoir vertices, this might result in an imbalance of the usage of edges incident to these vertices: some vertices in the reservoir might lose incident edges much faster or slower than the vertices in the regularity clusters. Apart from the fact that a fast loss of the edges incident to one vertex can prevent us from embedding any further spanning graphs into $G$, a large loss of the edges incident to the reservoir is also problematic in its own right. Indeed, since we are forced to use the edges incident to the reservoir in order to be able to embed some vertices onto vertices in $V_0$, this would prevent us from packing any further graphs.


Another issue is that the regularity lemma only gives us $\epsilon$-regular $K_k$-factor blow-ups while we need super-regular $K_k$-factor blow-ups in order to use Theorem~\ref{Blowup}. 
To overcome this issue, we will make appropriate adjustments to each $\epsilon$-regular $K_k$-factor blow-up. 
This means that the exceptional set $V_0$ will actually be different for each pair $t,s$ of indices. We can however use probabilistic arguments to ensure that this does not significantly affect the overall
`balance' of the packing.
 In particular, for simplicity, in the above proof sketch we have ignored this issue.

The paper is organised as follows. 
We collect some basic tools in Section~\ref{sec: prelim}, and we prove a lemma which finds a suitable partition of each graph $H \in \cH$  in Section~\ref{sec: H structure} (Lemma~\ref{lem: separable partition}).
We prove our main lemma (Lemma~\ref{embedone}) in Section~\ref{sec: main lemma}.
 This lemma guarantees that we can find a suitable packing of an appropriate collection $\cH_{t,s}$ of $k$-chromatic $\eta$-separable graphs with bounded degree into a graph consisting of a super-regular $K_k$-factor blow-up $G_{t,s}$ and suitable connection graphs $F_{t,s}$, $F'_t$ and $G_t^*$.
In Section~\ref{sec: main proof}, we will partition $G$ and $\cH$ as described above. 
Then we will repeatedly apply Lemma~\ref{embedone} to construct a packing of $\mathcal{H}$ into $G$.

\section{Preliminaries}\label{sec: prelim}
\subsection{Notation}
We write $[t] := \{1, \dots, t\}.$ 
We often treat large numbers as integers whenever this does not affect the argument.
The constants in the hierarchies used to state our results  are chosen from right to left.
That is, if we claim that a result holds for $0<1/n \ll a \ll b \leq 1$, we mean there exist non-decreasing functions $f:(0,1] \rightarrow (0,1] $ and $g:(0,1] \rightarrow (0,1]$ such that the result holds for all $0 \le a, b \le 1$ and all $n \in \mathbb{N}$ with $a \le f(b)$ and $1/n \le g(a)$. 
We will not calculate these functions explicitly. 

We use the word \emph{graphs} to refer to simple undirected finite graphs, and refer to \emph{multi-graphs} as graphs with potentially parallel edges, but without loops. 
\emph{Multi-hypergraphs} refer to (not necessarily uniform) hypergraphs with potentially parallel edges.
A \emph{$k$-graph} is a $k$-uniform hypergraph.
A \emph{multi-$k$-graph} is a $k$-uniform hypergraph with potentially parallel edges. 
For a multi-hypergraph $H$ and a non-empty set $Q\subseteq V(H)$,
we define $\text{mult}_H(Q)$ to be the number of parallel edges of $H$ consisting of exactly the vertices in $Q$.
 We say that a multi-hypergraph has \emph{edge-multiplicity} at most $t$ if $\text{mult}_H(Q) \le t$ for all non-empty $Q \subseteq V(H)$. 
A \emph{matching} in a multi-hypergraph $H$ is a collection of pairwise disjoint edges of $H$.
The \emph{rank} of a multi-hypergraph $H$ is the size of a largest edge.

We write $H\simeq G$ if two graphs $H$ and $G$ are isomorphic.  
For a collection $\cH$ of graphs, we let $v(\cH):= \sum_{H\in \cH} |V(H)|$.
We say a partition $V_1, \dots, V_k$ of a set $V$ is an \textit{equipartition} if $||V_i| - |V_j|| \le 1$ for all $i,j \in [k]$. 
For a multi-hypergraph $H$ and $A,B \subseteq V(H),$ we let $E_H(A,B)$ denote the set of edges in $H$ intersecting both $A$ and $B$. 
We define $e_H(A,B):= |E_H(A,B)|.$
For $v \in V(H)$ and $A \subseteq V(H)$,  we let $d_{H,A}(v): = |\{e \in E(H): v \in e, e\backslash \{v\} \subseteq A\}|$.%
\COMMENT{New: this definition has changed as the old definition allowed for edges that only contained $v$ and nothing else from $A$ whenever $v \in A.$}
 Let $d_{H}(v) := d_{H,V(H)}(v)$. 
For $u,v \in V(H)$, we define $c_H(u,v):= |\{ e \in E(H): \{u,v\} \subseteq e\}|.$ Let $\Delta(H) = \max\{ d_H(v) : v\in V(H)\}$ and $\delta(H):=\min\{ d_H(v): v\in V(H)\}$.

 For a graph $G$ and sets $X, A \subseteq V(G)$, we define
 $$N_{G,A}(X) := \{ w \in A: uw \in E(G) \text{ for all } u\in X \} \enspace \text{and} \enspace N_{G}(X) := N_{G,V(G)}(X).$$
Thus $N_G(X)$ is the common neighbourhood of $X$ in $G$ and $N_{G,A}(\emptyset) = A$.
 For a set $X \subseteq V(G),$ we define $N_G^d(X) \subseteq V(G)$ to be the set of all vertices of distance at most $d$ from a vertex in $X$.
 In particular, $N^d_{G}(X)= \emptyset$ for $d < 0$.
 Note that $N_{G}(X)$ and $N^1_{G}(X)$ are different in general as e.g. vertices with a single edge to $X$ are included in the latter.
 Moreover, note that $N_{G}(X) \subseteq N^1_{G}(X)$.
We say a set $I \subseteq V(G)$ in a graph $G$ is $k$-$\emph{independent}$ if for any two distinct vertices $u,v \in I$, the distance between $u$ and $v$ in $G$ is at least $k$ (thus a 2-independent set $I$ is an independent set).
If $A$, $B \subseteq V(G)$ are disjoint, we write $G[A,B]$ for the bipartite subgraph of $G$ with vertex classes $A$, $B$ and edge set $E_G(A,B).$

 For two functions $\phi:A\rightarrow B$ and $\phi': A'\rightarrow B'$ with $A\cap A'=\emptyset$, we let $\phi\cup \phi'$ be the function from $A\cup A'$ to $B\cup B'$ such that for each $x\in A\cup A'$, $$(\phi\cup\phi')(x) := \left\{ \begin{array}{ll}
\phi(x) & \text{ if } x\in A, \\
\phi'(x) & \text{ if } x\in A'.
\end{array}\right.$$

For graphs $H$ and $R$ with $V(R)\subseteq [r]$ and an ordered partition $ (X_1,\dots,X_r)$ of $V(H)$, we say that $H$ $\emph{admits the vertex partition}$ $(R, X_1,\dots, X_r)$, if  $H[X_i]$ is empty for all $i \in [r]$, and for any $i,j\in [r]$ with $i \ne j$
we have that $e_H(X_i, X_j) > 0$ implies $ij\in E(R).$ 
We say that $H$ is \emph{internally $q$-regular with respect to} $(R,X_1,\dots, X_r)$ if $H$ admits $(R,X_1,\dots, X_r)$ and $H[X_i,X_j]$ is $q$-regular for each $ij\in E(R)$.

We will often use the following Chernoff bound (see e.g. ~Theorem A.1.16 in \cite{AS08}). 

\begin{lemma} \label{lem: chernoff} \cite{AS08}
Suppose $X_1,\dots, X_n$ are independent random variables such that
$0\leq X_i\leq b$ for all $i\in [n]$. Let $X:= X_1+\dots + X_n$. Then for all $t>0$, $\mathbb{P}[|X - \mathbb{E}[X]| \geq t] \leq 2e^{-t^2/(2b^2 n)}$. 
\end{lemma}

\subsection{Tools involving $\epsilon$-regularity}
In this subsection, we introduce the definitions of $(\epsilon, d)$-{regularity} and $(\epsilon, d)$-{super-regularity}. 
We then state a suitable form of the {regularity lemma} for our purpose. 
 We will also state an embedding lemma (Lemma \ref{lem: embedding lemma}) which we will use later to prove our main lemma (Lemma~\ref{embedone}).

We say that a bipartite graph $G$ with vertex partition $(A,B)$ is  $(\epsilon,d)$-\emph{regular} if
for all sets $A'\subseteq A$, $B'\subseteq B$ with  $|A'|\geq \epsilon |A|$, $|B'|\geq \epsilon |B|$,
we have $| \frac{e_G(A',B')}{|A'||B'|}- d| < \epsilon.$
Moreover, we say that $G$ is \emph{$\epsilon$-regular}
if it is $(\epsilon,d)$-regular for some $d$.
If $G$ is $(\epsilon,d)$-regular and $d_{G}(a)= (d\pm \epsilon)|B|$ for $a\in A$ and $d_{G}(b)= (d\pm \epsilon)|A|$ for $b\in B$, then we say $G$ is  $(\epsilon,d)$-\emph{super-regular}.
We say that $G$ is $(\epsilon,d)^+$-\emph{(super)-regular} if it is $(\epsilon,d')$-(super)-regular for some $d'\geq d$.

For a graph $R$ on vertex set $[r]$, and disjoint vertex subsets $V_1,\dots, V_r$ of $V(G)$, we say that $G$ is $(\epsilon,d)^+$-\emph{(super)-regular with respect to the vertex partition} $(R,V_1,\dots, V_r)$ if $G[V_i,V_j]$ is $(\epsilon,d)^+$-(super)-regular for all $ij\in E(R)$. 
Being $(\epsilon,d)$-\emph{(super)-regular with respect to the vertex partition} $(R,V_1,\dots, V_r)$ is defined analogously. 
The following observations follow directly from the definitions.
\begin{proposition}\label{prop: reg smaller}
Let $0<\epsilon\leq \delta \leq d \leq 1$. Suppose $G$ is an $(\epsilon,d)$-regular bipartite graph with vertex partition $(A,B)$ and let $A' \subseteq A$, $B'\subseteq B$ with ${|A'|}/{|A|}$, ${|B'|}/{|B|}\geq \delta$.
Then $G[A',B']$ is $(\epsilon/\delta, d)$-regular.
\end{proposition}

\begin{proposition}\label{prop: reg subgraph}
Let $0<\epsilon \le \delta \le d \leq 1$. Suppose $G$ is an $(\epsilon,d)$-regular bipartite graph with vertex partition $(A,B)$. If $G'$ is a subgraph of $G$ with $V(G') = V(G)$ and $e(G') \geq (1-\delta)e(G)$, then $G'$ is $(\epsilon+ \delta^{1/3}, d)$-regular.%
\COMMENT{ Let $X$ and $Y$ be subsets of $A$ and $B$ respectively with $|X| > (\epsilon+ \delta^{1/3})|A|$ and $|Y| > (\epsilon+ \delta^{1/3})|B|$. We have that $e_{G'[X,Y]}$ can be as small as $d|X||Y| - \epsilon |X||Y| - \delta |A||B|.$  In order for the proposition to hold we require this to be larger in value than $d|X||Y| - (\epsilon + \delta^{1/3})|X||Y|$. This is equivalent to $\delta |A||B| < \delta^{1/3}|X||Y|.$ We have that $\delta^{1/3}|X||Y| >(\delta^{1/3})^3|A||B| = \delta |A||B|.$ The result follows.}
\end{proposition}

\begin{proposition}\label{prop: super}
Let $0<\epsilon \ll d \leq 1$.
Suppose $G$ is an $(\epsilon,d)$-regular bipartite graph with vertex partition $(A,B)$.
Let 
$$A':= \{ a\in A: d_{G}(a) \neq (d\pm \epsilon)|B|\} \text{ and } B':=\{  b\in B: d_{G}(b)\neq (d\pm \epsilon)|B|\}.$$
Then $|A'|\leq 2\epsilon |A|$ and $|B'|\leq 2\epsilon|B|$. 
\end{proposition}

The next lemma is a `degree version' of Szemer\'edi's regularity lemma (see e.g.~\cite{KOsurvey} on how to derive it from the original version).

\begin{lemma} [Szemer\'edi's regularity lemma]
\label{Szemeredi} 
Suppose $M, M', n \in \mathbb{N}$ with $0< 1/n \ll 1/M \ll \epsilon, 1/M' <1$ and $d >0$. 
Then for any $n$-vertex graph $G$, there exist a partition of $V(G)$ into $V_0, V_1, \dots, V_r$ and a spanning subgraph $G' \subseteq G$ satisfying the following. \begin{enumerate}[label=(\roman*)]
\item $M' \le r \le M,$
\item $|V_0| \le \epsilon n,$
\item $|V_i| = |V_j|$ for all $i, j \in [r],$
\item $d_{G'}(v) > d_G(v)  - (d + \epsilon)n$ for all $ v \in V (G),$
\item $e(G'[V_i]) = 0$ for all $ i \in [r],$
\item for all $i, j$ with $1 \le i < j \le r$, the graph $G' [V_i, V_j ]$ is either empty or $(\epsilon, d_{i,j})$-regular for some $d_{i,j} \in [d,1]$.
\end{enumerate}

\end{lemma}

The next lemma allows us to embed a small graph $H$ into a graph $G$ which is $(\epsilon,d)^+$-regular with respect to a suitable vertex partition $(R,V_1,\dots, V_r)$. 
In our proof of Lemma \ref{embedone} later on, properties (B1)$_{\ref{lem: embedding lemma}}$ and (B2)$_{\ref{lem: embedding lemma}}$ will help us to prescribe appropriate `target sets' for some of the vertices when we apply the blow-up lemma for approximate decompositions (Theorem \ref{Blowup}).
There, $H$ will be part of a larger graph that is embedded in several stages. 
(B1)$_{\ref{lem: embedding lemma}}$ ensures that the embedding of $H$ is compatible with constraints arising from earlier stages and (B2)$_{\ref{lem: embedding lemma}}$ will ensure the existence of sufficiently large target sets when embedding vertices $x$ in later stages (each edge of $\cM$ corresponds to the neighbourhood of such a vertex $x$).

\begin{lemma}\label{lem: embedding lemma}
Suppose $n, \Delta \in \mathbb{N}$ with $0<1/n \ll \epsilon \ll \alpha, \beta, d, 1/\Delta \leq 1$. Suppose that $G, H$ are graphs and $\mathcal{M}$ is a multi-hypergraph on $V(H)$ with edge-multiplicity at most $\Delta$.
Suppose $V_1,\dots, V_r$ are pairwise disjoint subsets of $V(G)$ with $\beta n\leq |V_i| \leq n$ for all $i\in [r]$,
and $X_1,\dots, X_r$ is a partition of $V(H)$ with $|X_i| \leq \epsilon n$ for all $i\in [r]$. Let $f: E(\mathcal{M}) \rightarrow [r]$ be a function, and for all $i\in [r]$ and $x\in X_i$, let $A_x \subseteq V_i$.
Let $R$ be a graph on $[r].$
 Suppose that the following hold.
\begin{enumerate}
\item[\rm (A1)$_{\ref{lem: embedding lemma}}$] $G$ is $(\epsilon,d)^+$-regular with respect to $(R,V_1,\dots, V_r)$,
\item[\rm (A2)$_{\ref{lem: embedding lemma}}$] $H$ admits the vertex partition $(R,X_1,\dots, X_r)$,
\item[\rm (A3)$_{\ref{lem: embedding lemma}}$] $\Delta(H)\leq \Delta$, $\Delta(\mathcal{M})\leq \Delta$ and the rank of $\mathcal{M}$ is at most $\Delta$,
\item[\rm (A4)$_{\ref{lem: embedding lemma}}$] for all $i, j \in [r]$, if $f(e)=i$ and $e\cap X_j\neq \emptyset$, then $ij\in E(R)$,\COMMENT{Had $i\ne j$ but without this we're still consistent by our definition of a graph (no loops)}
\item[\rm (A5)$_{\ref{lem: embedding lemma}}$] for all $i\in [r]$ and $x\in X_i$, we have $|A_x| \geq \alpha |V_i| $.
\end{enumerate}
Then there exists an embedding $\phi$ of $H$ into $G$ such that 
\begin{enumerate}
\item[\rm (B1)$_{\ref{lem: embedding lemma}}$] for each $x\in V(H)$, we have $\phi(x)\in A_x$,
\item[\rm (B2)$_{\ref{lem: embedding lemma}}$] for each $e\in \mathcal{M}$, we have
$|N_{G}( \phi(e)) \cap V_{f(e)} | \geq  (d/2)^{\Delta} |V_{f(e)}|$.
\end{enumerate}
\end{lemma}
Note that (A4)$_{\ref{lem: embedding lemma}}$ implies for all $e \in E(\mathcal{M})$ that $e \cap X_{f(e)}= \emptyset.$
\begin{proof}
For each $x \in V(H)$, let $e_x:= N_H(x)$ and $\cM'$ be a multi-hypergraph on vertex set $V(H)$ with $E(\cM') = \{ e_x : x\in V(H) \}$.
Since a vertex $x\in V(H)$ belongs to $e_y$ only when $y\in N_H(x)$, we have $d_{\cM'}(x) = d_{H}(x)$. So $\cM'$ is a multi-hypergraph with rank at most $\Delta$ and $\Delta(\cM')\leq \Delta$.
Let $\cM^* := \cM \cup \cM'$ and for each $e\in E(\cM^*)$,  define
$$B_e := \left\{\begin{array}{ll}
V_{f(e)} & \text{ if } e\in E(\cM), \\
A_x & \text{ if } e=e_x \in E(\cM') \text{ for } x\in V(H).
\end{array}\right. $$\COMMENT{Note that since $\cM^*$ is a multi-hypergraph, there could be a case that $e_x = e_y$ for vertices $x\neq y$. In this case, $B_{e_x}$ and $B_{e_y}$ may be  different.}
Note that by (A3)$_{\ref{lem: embedding lemma}}$, we have
\begin{align}\label{eq: max degree cH*}
\text{$\cM^*$ has rank at most $\Delta$, and }
\Delta(\cM^*) \leq \Delta(\cM)+ \Delta(\cM')\leq 2\Delta.
\end{align}
Let $V(H):= \{ x_1,\dots, x_m\}$, and for each $i\in [m]$, we let $Z_i:=\{x_1,\dots, x_i\}$.
We will iteratively extend partial embeddings $\phi_{0},\dots, \phi_{m}$ of $H$ into $G$ in such a way that the following hold for all $i \le m$.
\begin{enumerate}
\item[\rm ($\Phi$1)$^i_{\ref{lem: embedding lemma}}$] $\phi_{i}$ embeds $H[Z_i]$ into $G$,
\item[\rm ($\Phi$2)$^i_{\ref{lem: embedding lemma}}$] $\phi_i (x_k) \in A_{x_k}$, for all $k \in [i],$
\item[\rm ($\Phi$3)$^i_{\ref{lem: embedding lemma}}$] for all $e \in \cM^*$, we have $|N_{G}( \phi_i(e\cap Z_i)) \cap B_e| \geq (d/2)^{ |e\cap Z_i|} |B_e|.$
\end{enumerate}
Note that ($\Phi$1)$^0_{\ref{lem: embedding lemma}}$--($\Phi$3)$^0_{\ref{lem: embedding lemma}}$ hold for an empty embedding $\phi_0: \emptyset \rightarrow \emptyset$.
Assume that for some $i\in [m]$, we have already defined an embedding $\phi_{i-1}$ satisfying ($\Phi$1)$^{i-1}_{\ref{lem: embedding lemma}}$--($\Phi$3)$^{i-1}_{\ref{lem: embedding lemma}}$. 
We will construct $\phi_{i}$ by choosing an appropriate image for $x_{i}$.
Let $s \in [r]$ be such that  $x_i \in X_{s}$, and let $S:= N_{G}( \phi_{i-1}(Z_i \cap e_{x_{i}}) ) \cap B_{e_{x_i}}.$
Thus $S \subseteq V_s.$
Since $Z_{i-1} \cap e_{x_{i}} = Z_i \cap e_{x_i}$, we have that ($\Phi$3)$^{i-1}_{\ref{lem: embedding lemma}}$ implies  
\begin{align}\label{eq: S size}
|S| \geq (d/2)^{|Z_i \cap e_{x_{i}}|} \alpha \beta n > (d/2)^{\Delta} \alpha \beta n > \epsilon^{1/3}n.
\end{align}
For each $e\in E(\cM^*)$ containing $x_{i}$,\COMMENT{Note that $e_{x_i}$ does not contain $x_i$.} we consider
$$S_{e}:= N_{G}( \phi_{i-1}( Z_{i-1}\cap e ) ) \cap B_e.$$
By ($\Phi$3)$^{i-1}_{\ref{lem: embedding lemma}}$, we have 
\begin{align}\label{eq: Se size}
|S_{e}| \geq (d/2)^{\Delta}\alpha \beta n > \epsilon^{1/3}n.
\end{align}
If $e = N_H(x)$ for some $x\in X_{s'}$ with $s'\in [r]$, then we have $S_e\subseteq B_e \subseteq V_{s'}$, and (A2)$_{\ref{lem: embedding lemma}}$ implies that $ss' \in E(R)$.
 Moreover, note that if $e\in \cM$ with $f(e)=s'$ for some $s'\in [r]$, then $S_e\subseteq B_e = V_{s'}$, and (A4)$_{\ref{lem: embedding lemma}}$ implies that $ss'\in E(R)$.
Thus in any case, (A1)$_{\ref{lem: embedding lemma}}$ implies that $G[V_s, V_{s'}]$ is $(\epsilon,d')$-regular for some $d'\geq d$.
Hence, Proposition~\ref{prop: reg smaller} with \eqref{eq: S size} and \eqref{eq: Se size} implies that $G[S,S_e]$ is $(\epsilon^{1/2},d')$-regular.
Let 
$$S'_e:= \{ v\in S: d_{G,S_e}(v) < (d/2)|S_e|\}.$$
By Proposition~\ref{prop: super}, we have $|S'_e| \leq 2\epsilon^{1/2} n$.
Thus 
\begin{align}
 |S \setminus \bigcup_{e\in E(\cM^*): \space x_i \in e} S'_e| 
 \stackrel{\eqref{eq: max degree cH*}}{\geq} |S| - 2\Delta \cdot 2\epsilon^{1/2} n \stackrel{\eqref{eq: S size}}{\geq } 1. 
 \end{align}
We choose $v \in S\setminus \bigcup_{e\in E(\cM^*): \space x_i \in e} S'_e$, and we extend $\phi_{i-1}$ into $\phi_{i}$ by letting $\phi_{i}(x_i) := v$.
Since 
$$\phi_{i}(x_i) \in S = N_{G}( \phi_{i-1}(Z_i \cap e_{x_{i}}) ) \cap B_{e_{x_i}}= N_{G}( \phi_i(Z_{i} \cap N_H(x_i)) ) \cap A_{x_i},$$
 ($\Phi$1)$^i_{\ref{lem: embedding lemma}}$ and ($\Phi$2)$^i_{\ref{lem: embedding lemma}}$ hold.
Also, for each $e\in E(\cM^*)$,
if $x_{i}\notin e$, then as we have $Z_i\cap e = Z_{i-1} \cap e$,
$$|N_{G}( \phi_i( Z_i\cap e ) ) \cap B_e| = |N_{G}(  \phi_{i-1}( Z_{i-1}\cap e )) \cap B_e| \stackrel{\text{($\Phi$3)$^{i-1}_{\ref{lem: embedding lemma}}$ }}{\geq}  (d/2)^{|Z_i\cap e |} |B_e|.$$
If $x_{i}\in e$, then since $\phi_i(x_{i})\notin S'_e$ and $|Z_i\cap e| = |Z_{i-1} \cap e|+1$, we have 
\begin{eqnarray}
|N_{G}( \phi_i( Z_i\cap e)) \cap B_e| &\geq& |N_{G}(\phi_i(x_i) )\cap S_e|  \geq (d/2) |S_e| 
\stackrel{\text{($\Phi$3)$^{i-1}_{\ref{lem: embedding lemma}}$ }}{\geq} (d/2)^{|Z_i\cap e|} |B_e|.
\end{eqnarray}
Thus ($\Phi$3)$^{i}_{\ref{lem: embedding lemma}}$ holds.
By repeating this until we have embedded all vertices of $H$, we obtain an embedding $\phi_m$ satisfying ($\Phi$1)$^m_{\ref{lem: embedding lemma}}$--($\Phi$3)$^m_{\ref{lem: embedding lemma}}$.
Let $\phi:= \phi_m$. 
Then ($\Phi$2)$^m_{\ref{lem: embedding lemma}}$ implies that (B1)$_{\ref{lem: embedding lemma}}$ holds, and 
($\Phi$3)$^m_{\ref{lem: embedding lemma}}$ together with (A3)$_{\ref{lem: embedding lemma}}$ and the definition of $B_e$ implies that (B2)$_{\ref{lem: embedding lemma}}$ holds.
\end{proof}
\subsection{Decomposition tools}


In this subsection, we first give bounds on $\delta_{k}^{\text{reg}}$.
The following proposition provides a lower bound for $\delta_{k}^{\text{reg}}$. 
The proof is only a slight extension of the extremal construction given by Proposition $1.5$ in \cite{BKLO}, and thus we omit it here. 

\begin{proposition}\label{extremal example}
For all $k \in \mathbb{N}\backslash\{1,2\}$ we have $\delta_{k}^{\text{reg}} \ge 1-1/(k+1).$
\end{proposition}
It will be convenient to use that for $k\ge 2$ this lower bound implies
\begin{align}\label{eq: deltak 1-1/k}
\max\{1/2, \delta_k^{\text{reg}}\} \geq 1 - 1/k.
\end{align}

Given two graphs $F$ and $G$, let $\binom{G}{F}$ denote the set of all copies of $F$ in $G$.
A function $\psi$ from $\binom{G}{F}$ to $[0,1]$ is a \emph{fractional $F$-packing} of $G$ if $\sum_{F'\in \binom{G}{F}: e\in F'} \psi(F')\leq 1$ for each $e\in E(G)$ (if we have equality for each $e \in E(G)$ then this is referred to as a \emph{fractional $F$-decomposition}). 
Let $\nu^*_F(G)$ be the maximum value of $\sum_{F'\in \binom{G}{F}} \psi(F')$ over all fractional $F$-packings $\psi$ of $G$.
Thus $\nu^*_F(G) \le e(G)/e(F)$ and $\nu^*_F(G) = e(G)/e(F)$ if and only if $G$ has a fractional $F$-decomposition.
The following very recent result of Montgomery gives a degree condition which ensures a fractional $K_k$-decomposition in a graph.

\begin{theorem}\cite{M17}\label{thm: Montgomery}
Suppose $k, n \in \mathbb{N}$ and $0< 1/n \ll 1/k<1$. 
Then any $n$-vertex graph $G$ with $\delta(G) \geq (1- 1/(100k))n$ satisfies $\nu^*_{K_k}(G) = e(G)/e(K_k).$ 
\end{theorem}

The next result due to Haxell and R\"odl implies that a fractional $K_k$-decomposition gives rise to the existence of an approximate $K_k$-decomposition.

\begin{theorem}\cite{HR}\label{thm: fractional implies}
Suppose $n \in \mathbb{N}$ with $0 < 1/n \ll \epsilon <1$. 
Then any $n$-vertex graph $G$ has an $F$-packing consisting of at least $\nu^*_F(G) - \epsilon n^2$ copies of $F$.
\end{theorem}

\begin{lemma}\label{lem: delta* value}
For $k\in \mathbb{N} \backslash \{1,2\}$, we have $\delta_k^{\text{reg}} \leq \delta_k^{0+} \leq 1- 1/(100k)$.
Moreover, $\delta^{\text{reg}}_2 = \delta_2^{0+} = 0$ and $\delta^{\text{reg}}_3\leq \delta^{0+}_3 \leq 9/10$.
\end{lemma}
\begin{proof}
It is easy to see that Theorem~\ref{thm: Montgomery} and Theorem~\ref{thm: fractional implies} together imply that $\delta^{0+}_k \le 1- 1/(100k)$. 
Moreover, Theorem~\ref{thm: fractional implies} together with a result of Dross \cite{BKLO} implies that $\delta^{0+}_3 \leq 9/10$.
As any graph can be decomposed into copies of $K_2$, we have $\delta^{0+}_2 =0$.
\end{proof}

In the remainder of this subsection, we prove Lemma~\ref{lem: factor decomposition}.
In the proof of Theorem~\ref{thm:main}, we will apply it to obtain an approximate decomposition of the reduced multi-graph $R$
into almost $K_k$-factors (see Section~\ref{sec: main proof}).
We will use the following consequence of Tutte's $r$-factor theorem.
\begin{theorem}\cite{CKO}
\label{thm: hamilton decomp}
Suppose $n \in \mathbb{N}$ and $0< 1/n \ll \gamma \ll 1$. 
If $G$ is an $n$-vertex graph with $\delta(G) \ge (1/2 + \gamma) n$ and $\Delta(G) \leq \delta(G) + \gamma^2 n$, then $G$ contains a spanning $r$-regular subgraph for every even $r$ with $r \leq \delta(G)-\gamma n$.
\end{theorem}

The following powerful result of Pippenger and Spencer \cite{PS} (based on the R\"odl nibble) shows that every almost regular multi-$k$-graph with small maximum codegree has small chromatic index.

\begin{theorem}\cite{PS}\label{thm: hypergraph chromatic index}
Suppose $n, k \in \mathbb{N}$ and $ 0< 1/n \ll \mu \ll \epsilon, 1/k <1$.
Suppose $H$ is an $n$-vertex multi-$k$-graph satisfying $\delta(H) \geq (1-\mu)\Delta(H)$, and 
$c_H(u,v) \leq \mu \Delta(H)$ for all $u\neq v \in V(H)$.
Then we can partition $E(H)$ into $(1+\epsilon)\Delta(H)$ matchings.
\end{theorem}

We can now combine these tools to approximately decompose an almost regular multi-graph $G$ of sufficient degree into `almost' $K_k$-factors. 
All vertices of $G$ will be used in almost all these factors except the
vertices in a `bad' set $V'$ which are not used in any factor. Moreover, the factors come in $T$ groups of equal size such that parallel edges of $G$ belong to different groups.
As explained in Section~\ref{sec: outline}, we will apply this to the reduced
multi-graph obtained from Szemer\'edi's regularity lemma.

\begin{lemma}\label{lem: factor decomposition}
Suppose $n, k, q, T \in \mathbb{N}$ with $0< 1/n \ll \epsilon, \sigma, 1/T, 1/k, 1/q, \nu \le 1/2$\COMMENT{Need $\epsilon \le 1/2$ so to ensure that $(\delta - \nu \pm \epsilon)$ is reasonably large.} and $0<1/n \ll  \xi \ll  \nu < \sigma/2 < 1$ and $\delta = \max\{1/2,\delta_k^{\text{reg}}\} +\sigma$ and $q$ divides $T$.
Let $G$ be an $n$-vertex multi-graph with edge-multiplicity at most $q$, such that for all $v\in V(G)$ we have  
$$d_G(v) = (\delta \pm \xi)qn.$$
Then there exists a subset $V'\subseteq V(G)$ with $|V'|\leq \epsilon n$ and $k$ dividing $|V(G)\backslash V'|$,  and there exist pairwise edge-disjoint subgraphs $F_{1,1},\dots, F_{1,\kappa}, F_{2,1}, \dots, F_{T,\kappa}$ with $\kappa = (\delta - \nu \pm \epsilon)\frac{qn}{T(k-1)}$ satisfying the following. 
\begin{enumerate}
\item[\rm (B1)$_{\ref{lem: factor decomposition}}$] For each $(t',i)\in [T]\times[\kappa]$, we have that $V(F_{t',i}) \subseteq V(G) \backslash V'$ and $F_{t',i}$ is a vertex-disjoint union of at least $(1-\epsilon)n/k$ copies of $K_k$,
\item[\rm (B2)$_{\ref{lem: factor decomposition}}$] for each $v\in V(G)\setminus V'$, we have  $|\{ (t',i) \in [T]\times [\kappa] : v\in V(F_{t',i})\}| \geq T \kappa - \epsilon n,$ 
\item[\rm (B3)$_{\ref{lem: factor decomposition}}$] for all $t'\in [T]$ and $u,v \in V(G)$, we have 
$|\{ i \in [\kappa] :  u \in N_{F_{t',i}}(v)  \}|\leq 1.$
\end{enumerate}
\end{lemma}
\begin{proof}
It suffices to prove the lemma for the case when $T=q.$
The general case then follows by relabelling. (We can split each group obtained from the $T=q$ case into $T/q$ equal groups arbitrarily.)%
\COMMENT{To obtain the general case, apply the case where $T=q$ to obtain $F'_{1,1},\dots, F'_{1,\kappa'}, F'_{2,1}, \dots, F'_{q,\kappa'}$ for some $\kappa' = (\delta - \nu \pm \epsilon/2)\frac{n}{(k-1)}.$ 
Let $\kappa := \lfloor \frac{q\kappa'}{T} \rfloor.$ 
For each $i \in [q]$ relabel $F'_{i,1},\dots, F'_{i,T\kappa/q}$ as $F'_{\frac{(i-1)T}{q} + 1,1},\dots, F'_{\frac{(i-1)T}{q}+1,\kappa}, F'_{\frac{(i-1)T}{q} + 2,1}, \dots, F'_{\frac{iT}{q},\kappa}.$}
We choose a new constant $\mu$ such that 
$$1/n \ll \mu \ll \epsilon, \xi, \sigma, 1/k, 1/q.$$
For an edge colouring $\phi: E(G) \rightarrow [q]$ and $c \in [q]$, we let $G^c \subseteq G$ be the subgraph with edge set $\{e \in E(G): \phi(e)=c\}$. 
We wish to show that there exists an edge-colouring $\phi: E(G) \rightarrow [q]$ satisfying the following for all $v\in V(G)$ and $c\in [q]$:
\begin{enumerate}
\item[($\Phi$1)$_{\ref{lem: factor decomposition}}$] 
$d_{G^c}(v) = (\delta \pm 2\xi)n,$
\item[($\Phi$2)$_{\ref{lem: factor decomposition}}$]  $G^c$ is a simple graph.
\end{enumerate}

Recall that $e_G(u,v)$ denotes the number of edges of $G$ between $u$ and $v$.
 For each $\{u,v\} \in \binom{V(G)}{2}$, we choose a set $A_{\{u,v\}}$ uniformly at random from $\binom{[q]}{ e_G(u,v) }$.
 For each $e\in E(G)$, we let $\phi(e)\in [q]$ be such that $\phi$ is bijective between $E_{G}(u,v)$ and $A_{\{u,v\}}$.
 This ensures that ($\Phi$2)$_{\ref{lem: factor decomposition}}$ holds.
 It is easy to see that ($\Phi$1)$_{\ref{lem: factor decomposition}}$ also holds with high probability by using Lemma~\ref{lem: chernoff}.\COMMENT{
For $v\in V(G)$ and $c \in [q]$ let $X_{v}^c  = |N_{G^c}(v)|.$
We have
 $$\mathbb{E}[X_{v}^c] = \sum_{ u \in V(G)\setminus\{v\} }\frac{
 \binom{q-1}{e_G(u,v)-1}}{ \binom{q}{e_G(u,v)}} = \sum_{ u \in V(G)\setminus\{v\} } \frac{e_G(u,v) } {q} = \frac{1}{q}\sum_{u \in V(G)\setminus\{v\} }  e(u,v) = \frac{1}{q}d_G(v).$$
By Lemma~\ref{lem: chernoff} 
$$ \mathbb{P}[|X_{v}^c-\frac{1}{q}d_G(v) | \ge \xi n ] \le 2e^{-(\xi n)^2/(2n)} < e^{-\xi^2 n/2}.$$
By a union bound and the fact that $1/n \ll 1/q, \xi$ the probability that ($\Phi$1)$_{\ref{lem: factor decomposition}}$ holds for all $v \in V(G)$ and $c \in [q]$ is at least 
$$1 -  q n e^{-\xi^2 n/2} > 0.$$
It follows that there exists an edge colouring $\phi: E(G) \rightarrow [q]$ such that  ($\Phi$1)$_{\ref{lem: factor decomposition}}$ and  ($\Phi$2)$_{\ref{lem: factor decomposition}}$ hold.}

Since $\delta \geq 1/2+\sigma$ and $\xi \ll \nu,\sigma$, Theorem~\ref{thm: hamilton decomp} implies that, for each $c\in [q]$, there exists a $(\delta - \nu)n$-regular spanning subgraph $G_*^c$ of $G^c$.
(By adjusting $\nu$ slightly we may assume that $(\delta -\nu)n$ is an even integer.)
Since $\delta-\nu > \delta_k^{\text{reg}} + \sigma/2$ and $1/n \ll \mu$, 
 the graph $G_*^c$ has a $K_k$-packing $\cQ^c:=\{Q^c_{1},\dots, Q^c_{t}\}$
 of size 
\begin{align}\label{eq: size q}
t:= \frac{(\delta -\nu - \mu)n^2}{k(k-1)}.
\end{align}\COMMENT{Since $1/n_0 \ll \mu^2$, we have a $K_k$-packing of size at least $(1-\mu^2 )e(G^c_*)/e(K_k) \geq \frac{(\delta -\nu - \mu)n^2}{k(k-1)}.$ }
For each $c\in [q]$,  let $\cH^c$ be the $k$-graph with 
$V(\cH^c) = V(G^c_*)$ and $E(\cH^c):= \{ V(Q^c_{i}) : i \in [t]\}.$
By construction of $\cH^c$, we have
\begin{align}\label{eq: max degree cH}
\Delta(\cH^c) \leq \frac{\Delta(G^c_*)}{k-1} \leq \frac{(\delta -  \nu)n}{k-1}.
\end{align}
As $\cQ^c$ is a $K_k$-packing in $G^c_*$, any pair $\{u,v\} \in \binom{V(G)}{2}$ belongs to at most one edge in $\cH^c$. Thus for  $\{u,v\} \in \binom{V(G)}{2}$,
\begin{align}\label{eq: codegree G}
c_{\cH^c}(u,v) \leq 1.
\end{align}
Let 
$$V'':=  \bigcup_{c\in [q]}\Big\{ v \in V(G): \left|\{ i\in [t]: v \in V(Q^c_{i}) \}\right| < \frac{1}{k-1}(\delta - \nu- \mu^{1/3})n \Big\},$$
and let $V'$ be a set consisting of the union of $V''$ as well as at most $k-1$ vertices arbitrarily chosen from $V(G)\backslash V''$ such that $k$ divides $|V(G)\backslash V'|$.  
Note that for each $c\in [q]$, we have
$$e(G^c_*) - e(\cQ^c) \leq \frac{1}{2}(\delta-\nu)n^2 - \binom{k}{2} t \stackrel{\eqref{eq: size q}}{\leq} \mu n^2.$$
On the other hand, since $G^c_*$ is a $(\delta-\nu)n$-regular graph, we have
\begin{align}
 |V'| & \le k+1 + \sum_{c\in [q]} \frac{1}{\mu^{1/3}n}\sum_{v\in V(G)} \left(d_{G^c_*}(v) - (k-1)d_{{\cH}^c}(v) \right)  \nonumber \\
 & = k+1 + \sum_{c\in [q]} \frac{  2(e(G^c_*)-e(\cQ^c))}{\mu^{1/3}n }  \leq \frac{3 q \mu n^2}{\mu^{1/3}n }  \leq \mu^{1/2} n. \label{eq: V' size}
\end{align}
Let $\tilde{\cH}^c$ be the $k$-graph with 
$V(\tilde{\cH}^c):= V(G^c_*)\setminus V'$ and $E(\tilde{\cH}^c) := \{ e \in E(\cH^c) : e \cap V' =\emptyset\}.$
Note that for any $v\in V(\tilde{\cH}^c) = V(\cH^c) \setminus V'$, 
\begin{eqnarray}\label{eq: H degree v}
d_{\tilde{\cH}^c}(v) =  d_{{\cH}^c}(v) \pm \sum_{u\in V'} c_{\cH^c}(u,v)
\stackrel{\eqref{eq: codegree G}}{=}  d_{{\cH}^c}(v) \pm |V'| 
\stackrel{\eqref{eq: V' size},\eqref{eq: max degree cH}}{=} \frac{ (\delta - \nu \pm 2\mu^{1/3} ) n}{k-1}.
\end{eqnarray}
Note that we obtain the final equality from the definition of $V'$ and the assumption that $v\notin V'$.
 Thus for each $c\in [q]$, we have
$\delta(\tilde{\cH}^c) \geq (1 - \mu^{1/4}) \Delta(\tilde{\cH^c}).$
Together with \eqref{eq: codegree G} and the fact that $1/n \ll \mu \ll \epsilon, 1/k, 1/q$, this ensures that we can apply Theorem~\ref{thm: hypergraph chromatic index} to see that for each $c\in [q]$,
$E(\tilde{\cH}^c)$ can be partitioned into $\kappa':= \frac{(\delta - \nu+\epsilon^3/q)n}{k-1}$ matchings $M^c_{1},\dots, M^c_{\kappa'}$.
Let
$$\mathcal{M}^c:=\{ M^c_{i} : i\in [\kappa']\} \enspace \text{and} \enspace \mathcal{M}_*^c:=\{ M^c_{i} : i\in [\kappa'], |M^c_{i}| < (1-\epsilon)n/k\}.$$
As $|M^c_{i}| \leq n/k$ for any $i\in [\kappa']$ and $c\in [q]$, we have
$$\frac{(\delta - \nu- 3\mu^{1/3}) n^2}{k(k-1)} \stackrel{\eqref{eq: V' size},\eqref{eq: H degree v}}{ \leq} |E(\tilde{\cH}^c)| = \sum_{i\in [\kappa']} |M^c_{i}| < \frac{|\mathcal{M}_*^c|(1-\epsilon)n }{k} + \frac{(\kappa'- |\mathcal{M}_*^c|)n}{k}.$$
 This gives 
 \begin{align}\label{eq: M' size}
 |\mathcal{M}_*^c|  \leq \frac{(\epsilon^3/q+ 3\mu^{1/3})kn^2 }{\epsilon n k(k-1)} \leq \frac{2\epsilon^2  n}{q(k-1)}.
 \end{align}
We let
\begin{align}\label{eq: s size}
\kappa:=  \min_{c\in [q]} \{ |\mathcal{M}^c \setminus \mathcal{M}^c_*| \}
=  \kappa' - \max_{c\in [q]} \{|\mathcal{M}_*^c|\} 
 = \frac{(\delta-\nu )n \pm 2\epsilon^2 n/q}{k-1}.
\end{align}
Thus, by permuting indices, we can assume that for each $c\in [q]$, we have $M^c_1,\dots, M^c_{\kappa} \subseteq \mathcal{M}^c\setminus \mathcal{M}_*^c$.
For each $(c,i) \in [q]\times [\kappa]$, let 
$$F_{c,i} := \bigcup_{ j: V(Q^c_j) \in M^c_i } Q^c_j.$$

The fact that $\mathcal{M}^c \setminus \mathcal{M}_*^c$ is a collection of pairwise edge-disjoint matchings of $\tilde{\cH}^c \subseteq \cH^c$ together with \eqref{eq: codegree G} implies that, for each $c\in [q]$, the collection $\{F_{c,i} : i\in [\kappa]\}$ consists of pairwise edge-disjoint subgraphs of $G^c_* \subseteq G$, each of which is a union of at least $(1-\epsilon)n/k$ vertex-disjoint copies of $K_k$. This with ($\Phi$2)$_{\ref{lem: factor decomposition}}$ shows that (B3)$_{\ref{lem: factor decomposition}}$ holds.
As $G^1_*,\dots, G^q_*$ are pairwise edge-disjoint subgraphs, $\{F_{c,i} : (c,i) \in [q]\times [\kappa]\}$ forms a collection of pairwise edge-disjoint subgraphs of $G$. Thus (B1)$_{\ref{lem: factor decomposition}}$ holds. 

Moreover, for each $c\in [q]$ and each vertex $v\in V(G)\setminus V'$, we have
\begin{eqnarray*}
|\{ i\in [\kappa] : v\in V(F_{c,i})\}|
&\geq&  |\{M \in \{M^c_1,\dots, M^c_\kappa\}  : v\in V(M) \}| \\
&\ge&  |\{M \in \mathcal{M}^c  : v\in V(M) \}| - (\kappa'-\kappa) \\
& \ge& d_{\tilde{H}^c}(v) -  \kappa' + \kappa
\stackrel{\eqref{eq: H degree v}}{\geq} \kappa - \epsilon n/q.
\end{eqnarray*}
 Thus (B2)$_{\ref{lem: factor decomposition}}$ holds.
\end{proof}

\subsection{Graph packing tools}

The following two results from \cite{KKOT} will allow us to pack many bounded degree graphs into appropriate super-regular blow-ups.
Lemma~\ref{Pack} first allows us to pack graphs into internally regular graphs which still have bounded degree, and Theorem~\ref{Blowup} allows us to pack the internally regular graphs into an appropriate dense $\epsilon$-regular graph.
The results in \cite{KKOT} are actually significantly more general, mainly because they allow for more general reduced graphs $R.$

\begin{lemma} \cite[Lemma 7.1]{KKOT}
\label{Pack}
Suppose $n, \Delta, q, s, k, r \in \mathbb{N}$ with $0<1/n \ll \epsilon \ll 1/s \ll 1/\Delta, 1/k$ and $\epsilon \ll 1/q \ll 1$ and $k$ divides $r$.
 Suppose that $0< \xi < 1$ is such that $s^{2/3} \le \xi q$.
Let $R$ be a graph on $[r]$ consisting of $r/k$ vertex-disjoint copies of $K_k.$
Let $V_1, \dots, V_r$ be a partition of some vertex set $V$ such that $|V_i| = n$ for all $i \in [r].$
Suppose for each $j \in [s]$, $L_j$ is a graph admitting the vertex partition $(R, X^{j}_1,\dots, X^{j}_r)$ such that $\Delta (L_j) \le \Delta$ and for each $ii'\in E(R)$, we have 
$$\sum_{j=1}^{s} e(L_j [X_i^j, X_{i'}^j]) = (1 - 3\xi \pm \xi) q n,$$
and $|X_i^j| \leq n$.
 Also suppose that for all $j \in [s]$ and $i \in [r]$, we have sets $W_i^j \subseteq X_i^j$ such that $|W_i^j| \le \epsilon n$.
 Then there exists a graph $H$ on $V$ which is internally $q$-regular with respect to $(R,V_1,\dots, V_r)$ and a function $\phi$ which packs $\{L_1, \dots, L_s\}$ into $H$ such that $\phi(X_i^j) \subseteq V_i$, and such that for all distinct $j, j' \in [s]$ and $i \in [r],$ we have $\phi(W_i^j) \cap \phi(W_i^{j'}) = \emptyset$.
\end{lemma}

\begin{theorem} [Blow-up lemma for approximate decompositions {\cite[Theorem 6.1]{KKOT}}]
\label{Blowup}
Suppose $ n,$ $q$, $s$, $k, r \in \mathbb{N} $ with $0<1/n \ll \epsilon \ll \alpha, d, d_0, 1/q, 1/k \le 1$ and $1/n \ll 1/r$ and $k$ divides $r$.
Suppose that $R$ is a graph on $[r]$ consisting of $r/k$ vertex-disjoint copies of $K_k$.
Suppose $s \le \frac{d}{q} ( 1 - \alpha / 2) n$ and the following hold. \begin{enumerate}
\item[{\rm (A1)}$_{\ref{Blowup}}$] $G$ is  $(\epsilon, d)$-super-regular with respect to the vertex partition $(R,V_1,\dots, V_r)$.
\item[{\rm(A2)}$_{\ref{Blowup}}$] $\mathcal{H} = \{H_1, \dots ,H_s\}$ is a collection of graphs, where each $H_j$ is internally $q$-regular with respect to the vertex partition $(R,X_1,\dots, X_k),$ and $|X_i| = |V_i|=n$ for all $i\in [r]$.
\item[{\rm(A3)}$_{\ref{Blowup}}$]  For all $j \in [s]$ and $ i \in [r],$ there is a set $W_i^j \subseteq X_i$ with $|W_i^j| \le \epsilon n$ and for each $w\in W_i^j$, there is a set $A_w^j \subseteq V_i $ with $|A_w^j| \ge d_0 n.$
\item[{\rm(A4)}$_{\ref{Blowup}}$]  $\Lambda$ is a graph with $V(\Lambda) \subseteq [s] \times \bigcup_{i=1}^{r}X_i$ and $\Delta(\Lambda) \le (1-\alpha)d_0 n$\COMMENT{$\Delta(\Lambda) \le (1-\alpha)d_0n$ instead of $\Delta(\Lambda) \le (1-2\alpha)d_0n$ is ok here since we apply result in \cite{KKOT} with $\alpha/2$ playing the role of $\alpha$ there} such that for all $(j,x) \in V(\Lambda)$ and $j' \in [s]$, we have $|\{x': (j',x') \in N_\Lambda((j,x)) \}| \le q^2.$
Moreover, for all $j \in [s]$ and $i \in [r]$, we have $|\{ (j,x) \in V(\Lambda) : x \in X_i\}| \le \epsilon|X_i|.$
\end{enumerate}
Then there is a function $\phi$ packing $\mathcal{H}$ into $G$ such that, writing $\phi_j$ for the restriction of $\phi$ to $H_j$, the following hold for all $j \in [s]$ and $i \in [r].$\begin{enumerate}
\item[{\rm(B1)}$_{\ref{Blowup}}$]  $\phi_j(X_i) = V_i,$ 
\item[{\rm(B2)}$_{\ref{Blowup}}$]  $\phi_j(w) \in A_w^j$ for all $w \in W_i^j$, 
\item[{\rm(B3)}$_{\ref{Blowup}}$]  for all $(j,x)(j',y) \in E(\Lambda),$ we have that $\phi_j(x) \ne \phi_{j'}(y).$
\end{enumerate}
\end{theorem}

\subsection{Miscellaneous}
In the proof of Theorem~\ref{thm:main}, we often partition various graphs into parts with certain properties. 
The next two lemmas will allow us to obtain such partitions.
Lemma~\ref{Reservoir2} follows by considering a random equipartition and applying concentration of the hypergeometric distribution.
Lemma~\ref{star} can be proved by assigning each edge of $G$ to $G_1,\dots, G_s$ independently at random according to $(p_1,\dots, p_s)$, and applying Lemma~\ref{lem: chernoff}. We omit the details.

\begin{lemma} 
\label{Reservoir2}
Suppose $n, T, r \in \mathbb{N}$ with $0< 1/n \ll  1/T, 1/r \le 1$. 
Let $G$ be an $n$-vertex graph.
Let $V \subseteq V(G)$ and let $V_1 \dots, V_r$ be a partition of $V$.
Then there exists an equipartition $Res_1, \dots, Res_{T}$ of $V$ such that the following hold. \begin{enumerate}[label=(\roman*)]
\item For all $t\in [T]$, $i\in [r]$ and $v\in V(G)$, we have
$d_{G, Res_t\cap V_i}(v) = \frac{1}{T}d_{G,V_i}(v) \pm n^{2/3},$
\item  for all $t\in [T]$, $i\in [r]$, we have  $|Res_t \cap V_i| =  \frac{1}{T}|V_i| \pm n^{2/3} $.
\end{enumerate}
\end{lemma}\COMMENT{
Consider an equipartition, $Res_1, \dots, Res_T$ of $V$, chosen uniformly at random from the set of all equipartitions.
Fix a vertex $v\in V(G)$, $t \in [T]$ and $i \in [r]$. 
Let $Y:= d_{G, V_i\cap Res_t}(v).$ 
Then $Y$ has hypergeometric distribution with parameters $n, n/T, d_{G, V_i}(v)$.
We have
 $$ \mathbb{E}[Y] = \frac{1}{T}|N_G(v) \cap V_i|. $$
 By a Chernoff bound for the hypergeometric distribution,
$$\mathbb{P}[|Y - \mathbb{E}[Y]| > n^{2/3}] \le 2\exp( -(n^{2/3})^2/(3n)) \leq \exp(-n^{1/4}).$$ 

Fix $i \in [r]$ and $t \in [T]$. Let $Z: = |Res_t \cap V_i|$.
 We have that $\mathbb{E}[Z] = \frac{1}{T}|V_i|$ and
$$\mathbb{P}[|Z - \mathbb{E}[Z]| > n^{2/3}] \le \exp( -n^{1/4}).$$
By the union bound the probability that (i) holds for all $v \in {V(G)}$, $s \in [T]$ and $i \in [r]$ and (ii) holds for all $i\in [r]$ and all $s\in [T]$ is at least $1 - Trn \exp (-n^{1/4}) -2Tr \exp(-n^{1/4}) > 0.$
Thus it follows that we can choose such an equipartition satisfying both (i) and (ii).}

\begin{lemma}
\label{star}
Suppose $n, s \in \mathbb{N}$ with $0 < 1/n  \ll  \epsilon \ll  1/s \le1$ and $m_i \in [n]$ for each $i \in [2]$.
Let $G$ be an $n$-vertex graph.
Suppose that  $\mathcal{U}$ is a collection of $m_1$ subsets of $V(G)$ and $\mathcal{U}'$ is a collection of $m_2$ pairs of disjoint subsets of $V(G)$ such that each $(U_1, U_2) \in \mathcal{U'}$ satisfies $|U_1|, |U_2| >n^{3/4}$. Let $0 \leq p_1, \dots, p_s \le 1$ with $\sum_{i=1}^s p_i = 1.$
Then there exists a decomposition $G_1, \dots ,G_s$ of $G$ satisfying the following. \begin{enumerate}[label=(\roman*)]
\item For all $i \in [s]$, $U\in \mathcal{U}$ and $v \in V(G)$, we have 
$d_{G_i,U}(v) = p_i d_{G,U}(v)  \pm n^{2/3},$
\item for all $i \in [s]$ and $(U_1,U_2)\in \mathcal{U}'$ such that $G[U_1, U_2]$ is $(\epsilon,d_{(U_1,U_2)})$-regular for some $d_{(U_1, U_2)}$, we have that $G_i[U_1, U_2]$ is $(2\epsilon, p_id_{(U_1,U_2)})$-regular.
\end{enumerate}
\end{lemma}

The following lemma allows us to find well-distributed subsets of a collection of
large sets. The required sets can be found via a straightforward greedy approach (while
avoiding the vertices which would violate (B3)$_{\ref{lem: choice}}$ in each step).
So we omit the details.

\begin{lemma}\label{lem: choice}
Suppose $n, s,  r\in \mathbb{N}$ and $0<1/n, 1/s \ll \epsilon \ll d < 1$.
Let $A$ be a set of size $n$, and
for each $(i,j)\in [s]\times [r]$ let $A_{i,j} \subseteq A$ be of size at least $dn$, and let $m_{i,j} \in \mathbb{N} \cup \{0\}$ be such that for all $i \in [s]$ we have $\sum_{j=1}^{r} m_{i,j} \leq \epsilon n$.
Then there exist sets $B_{1,1},\dots, B_{s,r}$ satisfying the following.
\begin{enumerate}[label=(\roman*)]
\item[{\rm(B1)}$_{\ref{lem: choice}}$]  For all $i \in [s]$ and $j \in [r],$ we have $B_{i,j}\subseteq A_{i,j}$ with $|B_{i,j}| = m_{i,j}$,
\item[{\rm(B2)}$_{\ref{lem: choice}}$]  for all $i \in [s]$ and $j'\neq j''\in [r]$, we have  $B_{i,j'}\cap B_{i,j''}=\emptyset$,
\item[{\rm(B3)}$_{\ref{lem: choice}}$]  for all $v \in A$, we have $|\{ (i,j) \in [s]\times [r]: v\in B_{i,j}\} | \leq \epsilon^{1/2}  s$.
 \end{enumerate}
\end{lemma}%

The following lemma guarantees a set of $k$-cliques in a graph $G$ which cover every vertex a
prescribed number of times.

\begin{lemma}\label{lem: hypergraph degree}
Let $n,m,k, t \in \mathbb{N}$ and $0< 1/n \ll 1/t \ll \sigma, 1/k < 1$ with $k\mid n$.
Let $G$ be an $n$-vertex graph
with $\delta(G)\geq (1- \frac{1}{k}+\sigma)n$. 
Suppose that for each $v\in V(G)$, we have $d_v \in [m]\cup \{0\}$.
Then there exists a multi-$k$-graph $H$ on vertex set $V(G)$ satisfying the following.
\begin{enumerate}
\item[{\rm(B1)}$_{\ref{lem: hypergraph degree}}$] For each $e\in E(H)$, we have $G[e]\simeq K_k$,
\item[{\rm(B2)}$_{\ref{lem: hypergraph degree}}$] for each $v\in V(G)$, we have
$d_H(v) - d_v = (t+1) m \pm 1 $.
\end{enumerate}
\end{lemma}
\begin{proof}
Let 
$$m' := \max_{u,v\in V(G)} \{d_u-d_v\}.$$ 
Then $m'\in [m]$.
For a multi-hypergraph $H$ on vertex set $V(G)$ and $v\in V(G)$, let $p_H(v):= d_H(v) - d_v$.
We will prove that for each $\ell \in [m'-1] \cup \{0\}$, there exists a hypergraph $H_\ell$ satisfying the following.
\begin{enumerate}
\item[(H1)$_{\ref{lem: hypergraph degree}}^{\ell}$] For each $e\in E(H)$, we have $G[e]\simeq K_k$,
\item[(H2)$_{\ref{lem: hypergraph degree}}^{\ell}$] $\Delta(H_{\ell}) \leq  \ell (t+1)$, 
\item[(H3)$_{\ref{lem: hypergraph degree}}^{\ell}$] $\max_{u,v \in V(G)} \{
p_{H_{\ell}}(v) - p_{H_{\ell}}(u)\} \leq m'- \ell.$

\end{enumerate}
Note that $H_{0} =\emptyset$ satisfies (H1)$_{\ref{lem: hypergraph degree}}^{0}$--(H3)$_{\ref{lem: hypergraph degree}}^{0}$.
Assume that for some $\ell \in [m'-2] \cup \{0\}$, we have already constructed $H_{\ell}$ satisfying (H1)$_{\ref{lem: hypergraph degree}}^{\ell}$--(H3)$_{\ref{lem: hypergraph degree}}^{\ell}$.
We will now construct $H_{\ell+1}.$

If $\max_{u \in V(G)} \{ p_{H_{\ell}}(u) \}- \min_{u\in V(G)} \{ p_{H_{\ell}}(u)\} \le 1$, then as $\ell\leq m'-2$, we can let $H_{\ell+1}:=H_{\ell}$, then (H1)$_{\ref{lem: hypergraph degree}}^{\ell+1}$--(H3)$_{\ref{lem: hypergraph degree}}^{\ell+1}$ hold.
Thus assume that 
\begin{align}\label{eq: difference at least two}
\max_{u \in V(G)} \{ p_{H_{\ell}}(u) \}- \min_{u\in V(G)} \{ p_{H_{\ell}}(u)\}  \geq 2.
\end{align}
Let
$$A:=\{ v\in V(G) : p_{H_{\ell}}(v) > \min_{u\in V(G)}\{p_{H_{\ell}}(u)\}\} \text{ and } A_{\max}:= \{ v\in V(G): p_{H_{\ell}}(v) = \max_{u\in V(G)}\{p_{H_{\ell}}(u)\}\}.$$
First assume that $|A| \geq k$.
Let $A'\subseteq A$ be a set of at most $k-1$ vertices such that $k$ divides $|A|+|A'|$ and $p_{H_{\ell}}(v) \geq \max_{u\in A\setminus A'} p_{H_{\ell}}(u)$ for all $v\in A'$. Note that we have either $A'\subseteq A_{\max}$ or $A_{\max}\subseteq A'$. Then we can take a collection $\cA:=\{A_1,\dots, A_{t+1}\}$ of (possibly empty) subsets of $A$ such that the following hold for each $i\in [t+1]$.
\begin{itemize}
\item $|A_i|$ is divisible by $k$,
\item $|A_i| \leq |A|/t+k$, 
\item every vertex in $A'$ belongs to exactly two sets in $\cA$ and every vertex in $A\setminus A'$ belongs to exactly one set in $\cA$.
\end{itemize}\COMMENT{Such a collection exists. 
We delete $k-|A'|$ elements from $A$ to obtain a set $A''$ whose size is divisible by $k$ and such that $A'\subseteq A''$. We partition $A''$ into $t$ sets $A'_1,\dots, A'_t$ such that each set has size divisible by $k$, and $|A'_i| -|A'_j| \in \{-k,0,k\}$ for each $i,j\in [t]$.
Let $A'_{t+1}$ consists of the elements in $A'\cup (A\setminus A'')$. Thus $|A'_{t+1}|=k$. By repeatedly moving sets of $k$ vertices (not containing any vertex in $A'$) from (some of) the $A'_1,\dots, A'_t$ to
 $A'_{t+1}$ we obtain sets $A_1,\dots, A_{t+1}$ such that each set has size divisible by $k$, and $|A_i| -|A_j| \in \{-k,0,k\}$ for each $i,j\in [t+1]$.
This is our desired collection.
}
Now, for each $i\in [t+1]$, we have
$$\delta( G-A_i) \geq \delta(G)-|A_i| \geq (1-1/k+\sigma)n - n/t-k 
\geq (1 - 1/k + \sigma -2/t) n  \geq (1-1/k) n.$$
Since $V(G)\backslash A_i$ contains at most $n$ vertices, and $|V(G)\backslash A_i|$ is divisible by $k$, the Hajnal-Szemer\'edi theorem implies that there exists a collection $\cK_i$ of copies of $K_k$ in $G$ covering all the vertices in $V(G)\backslash A_i$ exactly once.
For each $i\in [t+1]$, let $E_i:= \{V(K): K\in \cK_i \}$. 
Then $\bigcup_{i=1}^{t+1}E_i$ covers every vertex in $V(G)\backslash A$ exactly $t+1$ times, while it covers vertices in $A\setminus A'$ exactly $t$ times and vertices in $A'$ exactly $t-1$ times.
Let $H_{\ell+1}$ be the multi-$k$-graph on vertex set $V(G)$ with 
$$E(H_{\ell+1}):=H_{\ell} \cup \bigcup_{i=1}^{t+1}E_i.$$
Then the above construction with (H1)$_{\ref{lem: hypergraph degree}}^{\ell}$ implies (H1)$_{\ref{lem: hypergraph degree}}^{\ell+1}$.
Also (H2)$_{\ref{lem: hypergraph degree}}^{\ell}$ implies that 
$\Delta(H_{\ell+1}) = \Delta(H_{\ell})+ (t+1) \leq (t+1)(\ell+1)$, thus (H2)$_{\ref{lem: hypergraph degree}}^{\ell+1}$ holds.
If $A'\subsetneq A_{\max}$, then every vertex in $A_{\max}\setminus A'$ is covered exactly $t$ times by $\bigcup_{i=1}^{t+1}E_i$. Thus, by \eqref{eq: difference at least two}, we have
$$\max_{u\in V(G)} \{p_{H_{\ell+1}}(u)\} = \max_{u\in V(G)}\{ p_{H_{\ell}}(u)\} +t
\enspace \text{and} \enspace \min_{u\in V(G)} \{p_{H_{\ell+1}}(u)\} = \min_{u\in V(G)}\{ p_{H_{\ell}}(u)\} +t+1.$$
If $A_{\max} \subseteq A'$, then every vertex in $A_{\max}$ is covered exactly $t-1$ times while every vertex in $A$ is covered either $t-1$ times or $t$ times by $\bigcup_{i=1}^{t+1}E_i$. Thus, by \eqref{eq: difference at least two}, we have
$$\max_{u\in V(G)} \{p_{H_{\ell+1}}(u)\} = \max_{u\in V(G)}\{ p_{H_{\ell}}(u)\} +t-1
\text{ and } \min_{u\in V(G)} \{p_{H_{\ell+1}}(u)\} \geq \min_{u\in V(G)}\{ p_{H_{\ell}}(u)\} +t.$$
In both cases, we have
\begin{eqnarray*}
\max_{u,v \in V(G)} \left\{ p_{H_{\ell+1}}(u) - p_{H_{\ell+1}}(v) \right\} 
\leq \max_{u,v\in V(G)}  \{ p_{H_{\ell}}(u)- p_{H_{\ell}}(v) \} -1 \stackrel{\text{(H3)$_{\ref{lem: hypergraph degree}}^{\ell}$}}{\leq} m'-\ell-1.
\end{eqnarray*}
Thus (H3)$_{\ref{lem: hypergraph degree}}^{\ell+1}$ holds.

Next assume that $|A|< k$.
Then we take two sets $B$ and $C$ in $V(G)$ such that $B\cap C=A$ and $|B|=|C|=k$.
Then similarly as before, we can take two collections $E_1$ and $E_2$ of sets of size $k$ such that $E_1$ covers every vertex in $V(G)\setminus B$ exactly once, and $E_2$ covers every vertex in $V(G)\setminus C$ exactly once while $G[e]\simeq K_k$ for all $e\in E_1 \cup E_2$.
Let $H_{\ell+1}$ be the multi-$k$-graph with $E(H_{\ell+1}):=H_{\ell} \cup E_1\cup E_2$.
Then, it is easy to see that both (H1)$_{\ref{lem: hypergraph degree}}^{\ell+1}$ and (H2)$_{\ref{lem: hypergraph degree}}^{\ell+1}$ hold.
Also $E_1\cup E_2$ covers all vertices in $V(G)\setminus A$ exactly once or twice, while it does not cover the vertices in $A$. 
Then as before, by using the fact that $\max_{u \in V(G)} \{ p_{H_{\ell}}(u) \}- \min_{u\in V(G)} \{ p_{H_{\ell}}(u)\}  \geq 2$,  we can show that (H3)$_{\ref{lem: hypergraph degree}}^{\ell+1}$ holds.
\COMMENT{
\begin{eqnarray*}
\max_{u,v \in V(G)} \left\{ p_{H_{\ell+1}}(u) - p_{H_{\ell+1}}(v) \right\} &\leq& (\max_{u\in V(G)}  \{ p_{H_{\ell}}(u)\} + 0)-  (\min_{u\in V(G)}  \{p_{H_{\ell}}(u)\} +1)\\
&\leq& \max_{u,v\in V(G)}  \{ p_{H_{\ell}}(u)- p_{H_{\ell}}(v) - 1 \} \stackrel{\text{(H3)$_{\ref{lem: hypergraph degree}}^{\ell}$}}{\leq} m'-\ell-1.
\end{eqnarray*}}

Hence, this shows that there exists a hypergraph $H_{m'-1}$ which satisfies (H1)$_{\ref{lem: hypergraph degree}}^{m'-1}$--(H3)$_{\ref{lem: hypergraph degree}}^{m'-1}$.
Let $m'':= \max_{v\in V(G)} \{p_{H_{m'-1}}(v)\} $.
 Then (H2)$_{\ref{lem: hypergraph degree}}^{m'-1}$ implies that $m''\le (t+1)m$.
Also, by  (H3)$_{\ref{lem: hypergraph degree}}^{m'-1}$ every vertex $v\in V(G)$ satisfies  $p_{H_{m'-1}}(v) \in \{m'' -1, m''\}.$
Recall that $\delta(G)\geq (1-1/k)n$ and $k$ divides $n$.
Thus the Hajnal-Szemer\'edi theorem guarantees a collection $E$ of sets of size $k$ which covers every vertex of $G$ exactly once, while $G[e]\simeq K_k$ for all $e\in E$.
Thus, by adding all $e\in E$ to $H_{m'-1}$ exactly $(t+1)m - m''$ times, we obtain a multi-$k$-graph satisfying (B1)$_{\ref{lem: hypergraph degree}}$ and (B2)$_{\ref{lem: hypergraph degree}}$.
\end{proof}

\COMMENT{Only needed in proof of \ref{lem: clique extension} and only for case when t=1.
\begin{claim*}\label{lem: 2k vertices}
Let $r,k,t \in \mathbb{N}$, and $R$ be an $r$-vertex graph with $\delta(R) \geq \Big(1 - \frac{1}{k} \Big)r+t $. 
For any $2k$ (not necessarily distinct) vertices $x_1,\dots,x_{2k}$, there are at least $t$ vertices $z$ such that 
$$|\{ j\in [2k]: x_j \in N_{R}(z)\}| \geq 2k-1.$$\COMMENT{Note that we don't write $N_{R}(z)\cap \{x_1,\dots, x_{2k}\}$ since $x_i=x_j$ for some $i\neq j$ could happen.}
\end{claim*}
\begin{proof}
For each $z\in V(R)$, let $d_z:= |\{ j\in [2k]: x_j\in N_{R}(z)\}|$.
Then we have
$$\sum_{z\in V(R)} d_z = \sum_{i=1}^{2k} d_{R}(x_i) \geq (2k-2)r+ 2k t.$$
As $d_z \leq 2k$, there are at least $t$ vertices $z\in V(R)$ with $d_z \geq 2k-1$.
\end{proof}
}

The following lemma is due to Koml\'os, S\'ark\"ozy and  Szemer\'edi \cite{KSS4}.
Assertion (B3)$_{\ref{lem: clique extension}}$ is not explicitly stated in \cite{KSS4}, but follows immediately from the proof given there (see Section~3.1 in \cite{KSS4}). 
Given embeddings of graphs $H_i$ and $H_j$ into blown-up $k$-cliques $Q_i \subseteq G$ and $Q_j \subseteq G$, the `clique walks' guaranteed by Lemma~\ref{lem: clique extension} will allow us to find suitable connections between (the images of) $H_i$ and $H_j$ in $G$.

\begin{lemma}\label{lem: clique extension}
Let $r,k\in \mathbb{N}\setminus \{1\}$. Suppose that $R$ is an $r$-vertex graph with $\delta(R) \geq \big(1 - \frac{1}{k}\big)r+1$. 
Suppose that $Q_1, Q_2$ are two not necessarily disjoint subsets of $V(R)$ of size $k$ such that $Q_1 =\{x_1,\dots, x_k\}$ and $Q_2 =\{y_1,\dots, y_{k}\}$ with $R[Q_1]\simeq K_k$ and $R[Q_2] \simeq K_k$.
 Then there exists a walk $W=(z_1,\dots, z_t)$ in $R$ satisfying the following. 
\begin{enumerate}
\item[{\rm(B1)}$_{\ref{lem: clique extension}}$] $3k\leq t \leq 3 k^3$ and $k \mid t$,
\item[{\rm(B2)}$_{\ref{lem: clique extension}}$] for all $i,j \in [t]$ with $|i-j| \leq k-1$, we have $z_i z_j \in E(R)$,
\item[{\rm(B3)}$_{\ref{lem: clique extension}}$] 
for each $i\in [k]$, we have $z_i=x_i$ and $z_{t-k+i} = y_i$. 
\end{enumerate}
\end{lemma}

The following lemma also can be proved using a simple greedy algorithm. We omit the proof.
\begin{lemma}\label{eq: k-independent set}
Let $\Delta,k, t \in \mathbb{N}\setminus \{1\}$.
Let $H$ be a graph with $\Delta(H) \leq \Delta$ and let $X \subseteq V(H)$ be a set with $|X| \geq  \Delta^{k} t$. Then there exists a $k$-independent set $Y\subseteq X$ of $H$ with $|Y| =t$.
\end{lemma}
\COMMENT{\begin{proof}
We sequentially choose vertices $v_1, v_2,\dots, v_{t}$ such that for each $i\in [t]$ we have
$v_{i} \in  X\setminus N^{k-1}_H(\{v_1,\dots, v_{i-1}\}).$
This is possible since for each $i\in [t]$, we have
$$|X \setminus N^{k-1}_H(\{v_1,\dots, v_{i-1}\})| \geq \Delta^{k}t - (i-1)(\Delta^{k-1}+ \Delta^{k-2}+ \dots + \Delta+1) > 0.$$
Then the resulting set $\{v_1,\dots, v_{t}\}$ is as required.
\end{proof}}

\begin{lemma}\label{lem: k-1 clique extend}
Let $r,k,q,s \in \mathbb{N}\setminus \{1\}$ with $0<1/r \ll 1/k, 1/q \leq 1$.
Let $R$ be an $r$-vertex graph with $\delta(R) \geq (1- \frac{1}{k}) r.$ 
Let $\cF$ be a multi-$(k-1)$-graph on $V(R)$ with $\Delta(\cF)\leq q$ and $E(\cF) = \{F_1, \dots, F_s\}$ such that $R[F_i]\simeq K_{k-1}$ for all $i \in [s]$. 
Then there exists a multi-$k$-graph $\cF^*$ on $V(R)$ with $E(\cF^*) = \{F_1^*, \dots, F_s^*\}$ and such that \begin{enumerate}
\item[{\rm(B1)}$_{\ref{lem: k-1 clique extend}}$] $\Delta(\cF^*) \leq (k+1)q$,
\item[{\rm(B2)}$_{\ref{lem: k-1 clique extend}}$] for all $i\in [s]$, we have $F_i \subseteq F^*_i$ and $R[F^*_{i}]\simeq K_k$.
\end{enumerate}
\end{lemma}
\begin{proof}
Since $\cF$ is a multi-$(k-1)$-graph, we have 
$s \leq \Delta(\cF)r/(k-1)\leq qr.$
We consider an auxiliary bipartite graph $Aux$ with vertex partition $(E(\cF) ,V(R) \times [k q])$ such that $F_i$ is adjacent to $(v,j) \in V(R) \times [k q]$ if $v\in N_R(F_i)$.
For any set $X$ of $k-1$ vertices in $R$, we have
$d_{R}(X) \geq r/k.$
Thus, any vertex $F_i$ of the graph $Aux$ has degree at least 
$ k q d_{R}(F_i) \geq kq \cdot (r/k) \geq s = |E(\cF)|.$
Thus, the graph $Aux$ contains a matching $M$ covering every $F_i\in E(\cF)$.
For each $(F_i, (v,j)) \in M$, let $F^*_i:= F_i\cup \{v\}$.
Then (B2)$_{\ref{lem: k-1 clique extend}}$ holds.
On the other hand, for any vertex $v\in V(R)$, we have
$
d_{\cF^*}(v) = d_{\cF}(v) + |\{ j  \in  [kq] : d_M((v,j))=1\}|
\leq d_{\cF}(v)+ kq \leq 
(k+1)q.
$
Thus (B1)$_{\ref{lem: k-1 clique extend}}$ holds too.
\end{proof}

The final tool we will collect 
implies that a $(k,\eta)$-chromatic $\eta$-separable bounded degree graph has a small separator $S$
and a $(k+1)$-colouring in which one colour class is small and only consists of vertices far away from $S$.
\begin{lemma}\label{lem: chromatic speparable}
Suppose that $n,t,\Delta,k\in \mathbb{N}$ and $\Delta\geq 2$.
Suppose that $H$ is an $\eta$-separable $n$-vertex graph with $\Delta(H)\leq \Delta$. If $H$ admits a $(k+1)$-colouring with colour classes $W_0,\dots, W_k$ with $|W_0| \leq \eta n$, then there exists a $\Delta^{t+2} \eta$-separator $S$ of $H$ with $N^{t}_H(S)\cap W_0=\emptyset$.
\end{lemma}
\begin{proof}
As $H$ is $\eta$-separable, there exists an $\eta$-separator $S'$ of $H$.
Consider $S:= (S' \cup N^{t+1}_H(W_0))\setminus N^{t}_{H}(W_0)$.
It is obvious that such a choice satisfies $N^{t}_H(S)\cap W_0=\emptyset$. Furthermore, as $|W_0|\leq \eta n$ and $\Delta\geq 2$, we have $|S|\leq \Delta^{t+2}\eta n$.
Moreover, any component of $H-S$ is either a subset of a 
a component of $H-S'$ or a subset of $N^{t}_H(W_0)$. Hence, it has size at most $\Delta^{t+2}\eta n$, and $S$ is a separator as desired.
\end{proof}

\section{Constructing an appropriate partition of a separable graph}
\label{sec: H structure}

In Section 6 we will decompose the host graph $G$ into graphs $G_t, F_t$ and $F'_t$ with $t \in [T]$ for some bounded $T$. 
We will also construct an exceptional set $V_0$ and reservoir sets $Res_t.$
We now need to partition each graph $H \in \cH$ so that this partition reflects the above decomposition of $G$. 
This will enable us to apply the blow-up lemma for approximate decompositions (Theorem~\ref{Blowup}) in Section~\ref{sec: main lemma}. 
The next lemma ensures that we can prepare each graph $H\in \cH$ in an appropriate manner. It gives a partition of $V(H)$ into $X,Y,Z, A$.
Later we will aim to embed the vertices in $A$ into $V_0$, and vertices in $Y\cup Z$ will be embedded into $Res_t$ using Lemma~\ref{lem: embedding lemma}.
Most of the vertices in $X$ will be embedded into a super-regular blown-up $K_k$-factor in $G_t$ via Theorem~\ref{Blowup}, while the remaining vertices of $X$ will be embedded into $Res_t.$ The set $Z$ will contain a suitable separator $H_0$ of $H$. 
The neighbourhoods of the exceptional vertices $a_\ell \in A$ will be allocated to $Y$.
Moreover, (A2)$_{\ref{lem: separable partition}}$ and (A3)$_{\ref{lem: separable partition}}$ ensure that we allocate them to sets corresponding to (evenly distributed) cliques of $R -$the latter enables us to satisfy the second part of (B3)$_{\ref{lem: separable partition}}$.

\begin{lemma}\label{lem: separable partition}
Suppose $n,m, r,  k, h, \Delta \in \mathbb{N}$ with
$0<1/n \ll \eta \ll \epsilon \ll 1/h \ll 1/k, \sigma, 1/\Delta< 1$ and
$0<\eta \ll 1/r<1$ such that $k \mid r$. 
Let $H$ be an $n$-vertex $(k,\eta)$-chromatic graph with $e(H)=m$ and $\Delta(H)\leq \Delta$. Let $R$ and $Q$ be graphs with $V(R)=V(Q)=[r]$ such that $Q$ is a union of $r/k$ vertex-disjoint copies of $K_k$. For $n'\in[\epsilon n]$, let 
$C_{1},\dots, C_{n'}$ be subsets of $[r]$ of size $k-1$, and $C^*_1,\dots, C^*_{n'}$ be subsets of $[r]$ of size $k$.
Let $\cF$ and $\cF^*$ be multi-hypergraphs on $[r]$ with 
$E(\cF)=\{ C_{1},\dots, C_{n'}\}$ and $E(\cF^*)=\{ C^*_1,\dots, C^*_{n'}\}$.
Suppose that $n_1,\dots, n_r$ are integers.
 Suppose the following hold.
\begin{enumerate}
\item[{\rm(A1)}$_{\ref{lem: separable partition}}$] $\delta(R) \geq (1 - \frac{1}{k}+\sigma)r$,
\item[{\rm(A2)}$_{\ref{lem: separable partition}}$] for each $\ell \in [n']$, we have
$C_\ell \subseteq C^*_\ell$ and $R[C^*_\ell]\simeq K_k$,
\item[{\rm(A3)}$_{\ref{lem: separable partition}}$]  $\Delta(\cF^*)\leq \epsilon^{2/3}n/r$,
\item[{\rm(A4)}$_{\ref{lem: separable partition}}$] for each $i\in [r]$, we have $n_i = (1\pm \epsilon^{1/2})n/r$, and $n'+\sum_{i\in [r]} n_i = n$.
\end{enumerate}
Then there exists a randomised algorithm which always returns an ordered partition $(X_1,\dots, X_r,$ $Y_1,\dots, Y_r,$ $Z_1, \dots, Z_r,$ $A)$ of $V(H)$ such that $A=\{a_1,\dots, a_{n'}\}$ is a $3$-independent set of $H$ and the following hold, where $X:=\bigcup_{i\in [r]} X_i, Y:= \bigcup_{i\in [r]} Y_i,$ and $Z:=\bigcup_{i\in [r]} Z_i$.
\begin{enumerate}
\item[{\rm(B1)}$_{\ref{lem: separable partition}}$] For each $\ell \in [n']$, we have $d_H(a_\ell ) \leq \frac{2(1+1/h)m}{n}$,
\item[{\rm(B2)}$_{\ref{lem: separable partition}}$] for each $\ell \in [n']$, we have $N_H(a_\ell) \subseteq \bigcup_{i \in C_{\ell}} Y_i \setminus N^1_H(Z)$,
\item[{\rm(B3)}$_{\ref{lem: separable partition}}$] 
$H[X]$ admits the vertex partition $(Q, X_1,\dots, X_r)$,
and $H\setminus E(H[X])$ admits the vertex partition $(R,X_1\cup Y_1\cup Z_1,\dots, X_r\cup Y_r\cup Z_r)$,

\item[{\rm(B4)}$_{\ref{lem: separable partition}}$]  for each $ij \in E(Q)$, we have $e_H(X_i, X_j) =  \frac{2m \pm \epsilon^{1/5} n}{(k-1)r}$,
\item[{\rm(B5)}$_{\ref{lem: separable partition}}$]  for each $i\in [r]$, we have  $|X_i|+|Y_i| +|Z_i| = n_i \pm \eta^{1/4} n$ and
$|Y_{i}| \leq 2 \epsilon^{1/3}n/r$,
\item[{\rm(B6)}$_{\ref{lem: separable partition}}$]  $N^{1}_H(X)\setminus X \subseteq Z$ and $|Z| \leq 4 \Delta^{3k^3} \eta^{0.9} n$.
\end{enumerate}
Moreover, the algorithm has the following additional property, where the expectation is with respect to all possible outputs.
\begin{enumerate}
\item[{\rm(B7)}$_{\ref{lem: separable partition}}$]  
For all $\ell \in [n']$ and $ i \in C_\ell$, we have $\mathbb{E}[ N_{H}(a_\ell ) \cap Y_{ i } ] \leq \frac{2(1+1/h)m}{(k-1)n}.$
\end{enumerate}
\end{lemma}
\rm(B1)$_{\ref{lem: separable partition}}$ and \rm(B7)$_{\ref{lem: separable partition}}$ ensure that each embedding of some $H$ in $G$ does not use too many edges incident to the exceptional set $V_0$.
\begin{proof}
Write $r':=r/k$ and $Q = \bigcup_{s=1}^{r'} Q_{s}$, where each $Q_s$ is a copy of $K_k$, and 
let $\binom{R}{K_k}=\{Q'_1,\dots, Q'_q\}$ be the collection of all copies of $K_k$ in $R$. By permuting indices if necessary, we may assume that $V(Q'_1)=\{1,\dots, k\}$. Note that $q\leq r^k$.
As $Q$ is a $K_k$-factor on $[r]$, for each $i\in [r]$, there exists a unique $j\in [r']$ such that $i \in Q_j$.
For all $s\in [r']$, $s' \in [q]$ and $k'\in [k]$, we define $q_{s}(k'), q'_{s'} (k') \in [r]$ to be the $k'$-th smallest number in $V(Q_s)$ and $V(Q'_{s'})$ respectively. 
Thus
$$V(Q_{s})= \{ q_s(1), \dots, q_{s}(k)\} \text{ and } V(Q'_{s'})= \{ q'_{s'}(1), \dots, q'_{s'}(k) \}.$$
For all $s\in [q]$ and $k'\in [k]$, let
\begin{align}\label{eq: Q d def}
Q'_{s,k'}:= Q'_s \setminus \{q'_s (k')\} \enspace \text{ and } \enspace
d_{s,k'}:= |\{ \ell \in [n'] : C^*_\ell = V(Q'_{s}) \text{ and } C_\ell = V(Q'_{s,k'}) \}|.
\end{align}
Note that for each $i\in [r]$ we have 
\begin{align}\label{eq: dellk sum}
\sum_{s\in [q]: i\in V(Q'_s)} \sum_{k'\in [k]} d_{s,k'} = d_{\cF^*}(i) \quad \text{and} \sum_{ (s,k')\in [q]\times[k]} d_{s,k'} = n'.
\end{align}
Our strategy is as follows. Consider a $(k+1)$-colouring $(W_0,\dots, W_k)$ of $H$ with $|W_0| \leq \eta n$ and an $\Delta^{3k^3+3}\eta n$-separator $S$ of $H$ guaranteed by Lemma~\ref{lem: chromatic speparable} (applied with $t=3k^3+1$).
Thus we can partition the $k$-chromatic graph $H\setminus W_0$ into $H_0,\dots, H_t$ such that each $H_{t'}$ is small, there are no edges between $H_{t'}$ and $H_{t''}$ whenever $0\notin \{t',t''\}$  and $V(H_0)=S$.
We will distribute the vertices of each graph $H_{t'}$ into $\bigcup_{i\in V(Q_{s})} X_i$ or $\bigcup_{i\in V(Q'_{s})} (Y_i\cup Z_i)$ for an appropriate $s$.
In particular, $V(H_0)$ will be allocated to $\bigcup_{i\in V(Q'_1)} Z_i= \bigcup_{i \in [k]} Z_i$.
As $Q'_s$ and $Q_s$ are copies of $K_k$ in $R$ and $Q$, respectively,
and as  $H_{t'}$ is $k$-chromatic, this would allow us to achieve (B3)$_{\ref{lem: separable partition}}$ if we ignore the edges incident to $V(H_0) \cup W_0$. 
In Steps 5 and 6 we will use `clique walks' obtained from Lemma~\ref{lem: clique extension} to connect up the $H_{t'}$ with $H_0$ in a way which respects the colour classes of $H\setminus W_0$. 
We can thus allocate the vertices in $N_H^{3k^3}(V(H_0))$ in a way that will satisfy (B3)$_{\ref{lem: separable partition}}$.
Finally, we will allocate the vertices in $W_0$. As $W_0$ is far from $V(H_0)$, each vertex in $W_0$ only has its neighbours in a single $H_{t'}$, hence it will be simple to assign each vertex in $W_0$ to some $Z_i$ with $i\in [r]$ according to where the vertices of $H_{t'}$ are assigned.
 \newline

\noindent {\bf Step 1. Separating $H$.} As $H$ is $(k,\eta)$-chromatic, applying Lemma~\ref{lem: chromatic speparable} with $t=3k^3+1$ implies that there exists a partition $(W_0,W_1,\dots, W_k)$ of $V(H)$ into independent sets and an $\eta^{0.9}$-separator $S$ such that 
\begin{align}\label{eq: chrosep}
|S|, |W_0| \leq \eta^{0.9} n \text{ and } W_0 \cap N_H^{3k^3+1}(S) =\emptyset.
\end{align}
Since $S$ is an $\eta^{0.9}$-separator of $H$, it follows that there exists a partition $S=:\widetilde{V}_0,\dots, \widetilde{V}_t$ of $V(H)$ such that the following hold, where $V_{t'}:= \widetilde{V}_{t'}\setminus W_0$
and $H_{t'}:= H[V_{t'}]$ for each $t'\in [t]\cup \{0\}$.
\begin{enumerate}[label=\text{\rm (H\arabic*)$_{\ref{lem: separable partition}}$}]
\item \label{lem H1} $ \eta^{-0.9}/2 \leq t\leq 2  \eta^{-0.9}$,
\item\label{lem H2}  $\eta^{0.9} n/2 \leq |V_{t'}| \leq 2 \eta^{0.9} n$ for $t' \in [t]$,
\item \label{lem H3} 
for $t' \neq t'' \in [t]$, we have that $E_H( \widetilde{V}_{t'}, \widetilde{V}_{t''}) =\emptyset$, and $m -2\Delta  \eta^{0.9} n \leq \sum_{t' \in [t]} e(H_{t'}) \leq m$.
\end{enumerate}
Indeed, as $S$ is an $\eta^{0.9}$-separator of $H$, $H \setminus S$ only consists of components of size at most $\eta^{0.9} n$.
By letting $\tilde{V}_0:= S$ (and thus $V_0=S$) and letting each of $\widetilde{V}_1,\dots, \widetilde{V}_{t}$ be appropriate unions of components of $H\setminus S$, we can ensure that both \ref{lem H1} and \ref{lem H2} hold. By the construction, the first part of \ref{lem H3} holds too. Since there are at most $\Delta(H)|S\cup W_0| \leq 2\Delta \eta^{0.9} n$ edges which are incident to some vertex in $W_0\cup V_0$, the second part of \ref{lem H3} holds as well. 

For each $t'\in [t]\cup \{0\}$ and $k'\in [k]$, we let
$$ W^{t'}_{k'}:=  V_{t'} \cap W_{k'}.$$ 
\vspace{0.05cm}

\noindent {\bf Step 2. Choosing the exceptional set $A$.}
Let $$L:= \{ x \in V(H) : d_{H}(x) \leq \frac{2(1+1/h) m}{n} \}.$$
$L$ contains the `low degree' vertices within which we will choose $A$ in order to satisfy (B1)$_{\ref{lem: separable partition}}$.
Note that
$2m = \sum_{x \in V(H)} d_{H}(x)
\geq  \frac{2(1+1/h)m }{n}( n-|L| )$, thus
\begin{eqnarray}\label{eq: Lqs size}
|L| \geq n/(2h).
\end{eqnarray}
For each $t'\in [t]$, let $k(t') \in [k]$ be an index such that
\begin{align}\label{eq: kt' def}
|L \cap W^{t'}_{k(t')}| \geq \frac{1}{k} |L\cap V(H_{t'})|.
\end{align}
Such a number $k(t')$ exists as $W^{t'}_1,\dots, W^{t'}_k$ forms a partition of $V_{t'}=V(H_{t'})$. 

Now, we choose a partition $\cH, \cH'_{1,1},\dots, \cH'_{1,k},\cH'_{2,1},\dots, \cH'_{q,k}$ of $\{H_1,\dots, H_{t}\}$ satisfying the following for each $(s,k')\in [q]\times [k]$. 
 \begin{enumerate}[label=\text{\rm (H{\arabic*})$_{\ref{lem: separable partition}}$}]
 \setcounter{enumi}{3}
 \item \label{lem H4}  $\displaystyle v(\cH'_{s,k'}) = \epsilon^{-1/10} d_{s,k'} + 2k \eta^{2/5}n \pm \eta^{2/5}n$ and\newline
  $\sum_{t': H_{t'} \in \cH'_{s,k'}} |V(H_{t'})\cap L| \geq \epsilon^{-1/11} d_{s,k'} + \eta^{1/2}n $.
\end{enumerate}
We will choose $A$ within the vertex sets of the graphs in  $\cH'_{1,1}, \dots, \cH'_{q,k}$. 
Moreover, we will allocate all the other vertices of the graphs in each $\cH'_{s,k'}$ to $Y\cup Z$.
\begin{claim}\label{cl: H4 exist}
There exists a partition $\cH, \cH'_{1,1},\dots, \cH'_{1,k},\cH'_{2,1},\dots, \cH'_{q,k}$ of $\{H_1,\dots, H_{t}\}$ satisfying \ref{lem H4}.
\end{claim}
\begin{proof}
For each $t'\in [t]$, we choose $i_{t'}$ independently at random from $[q]\times [k] \cup \{(0,0)\}$ such that for each $(s,k')\in [q]\times [k]$ we have
$$\mathbb{P}[i_{t'} =  (s,k') ] = \frac{ \epsilon^{-1/10} d_{s,k'}}{  n} + 2k \eta^{2/5} 
\enspace \text{ and } \enspace 
 \mathbb{P}[ i_{t'} = (0,0) ] = 1- \frac{\epsilon^{-1/10} n'}{ n}-  2qk^2 \eta^{2/5}.$$

An easy calculation based on \eqref{eq: dellk sum} shows that this defines a probability distribution.\COMMENT{Since $n' \leq \epsilon n$ and $ \eta \ll \epsilon \ll 1/k$, and $\eta \ll 1/r$, we have $\mathbb{P}[ i_{t'} = (0,0) ] >0.$
Note that 
{\small $$\sum_{(s,k')\in [q]\times[k]} \mathbb{P}[i_{t'}=(s,k')]   \stackrel{\eqref{eq: dellk sum}}{=} \epsilon^{-1/10} n'/n+ 2qk^2 \eta^{2/5} .$$}}
For each $(s,k') \in [q]\times[k]$, we let
$$\cH:= \{H_{t'} : t'\in [t], i_{t'}= (0,0)\} \enspace \text{ and } \enspace \cH'_{s,k'}:=\{ H_{t'} :t'\in [t], i_{t'} = (s,k')\}.$$
Then it is easy to combine a Chernoff bound (Lemma~\ref{lem: chernoff}) with \ref{lem H1}, \ref{lem H2}, \eqref{eq: Lqs size}  and the fact that $|V(H)|=n$ to check that the resulting partition satisfies \ref{lem H4} with positive probability. 
This proves the claim.
\COMMENT{
The fact that $v(\{H_1,\dots,H_t\}) = n\pm 2\eta^{0.9} n$, and \eqref{eq: Lqs size} implies that for each $(s,k')\in [q]\times [k]$, we have
$$\mathbb{E}[ v(\cH'_{s,k'}) ] = \epsilon^{-1/10}d_{s,k'}+ 2k\eta^{2/5} n \pm \eta^{3/4} n, \text{ and }
\mathbb{E}[ \sum_{t': H_{t'} \in \cH'_{s,k'}} |V(H_{t'})\cap L| ] \geq \epsilon^{-1/10}d_{s,k'}/(2h)+ k\eta^{2/5}n/h \pm \eta^{3/4} n.$$
By using a Chernoff bound (Lemma~\ref{lem: chernoff}) with \ref{lem H1} and \ref{lem H2} we conclude that 
$$\mathbb{P}[ v(\cH'_{s,k'}) =  \epsilon^{-1/10}d_{s,k'}+ 2k\eta^{2/5} n\pm \eta^{2/5} n] \geq 1- 2\exp( -\frac{(\eta^{2/5}n/2)^2}{ 2(2\eta^{0.9}} n)^2t  ) \geq 1 - \eta,$$
$$\mathbb{P}[ \sum_{t': H_{t'} \in \cH'} |V(H_{t'})\cap L| \geq \epsilon^{-1/11} d_{s,k'} + \eta^{1/2}n ] \geq 1- 2\exp( -\frac{(\eta^{2/5}n/(2h))^2}{ 2(2\eta^{0.9} n)^2t } ) \geq 1 - \eta.$$
This with the union bound shows that there exists $\cH'$ satisfying \ref{lem H4}. 
}
\end{proof}
By permuting indices on $[t]$, we may assume that for some $t_*\in [t]$, we have
$$\cH = \{H_1,\dots, H_{t_*}\} \enspace \text{and} \enspace \bigcup_{(s,k')\in [q]\times [k]} \cH'_{s,k'} = \{H_{t_*+1},\dots, H_{t}\}.$$
For each $(s,k')\in [q]\times[k]$, let
\begin{align}\label{eq: Lsk' def}
L_{s,k'}:= \bigcup_{t' : H_{t'} \in \cH'_{s,k'} } (L\cap W^{t'}_{k(t')})  \setminus N^{3k^3+2}_H(V_0\cup W_0).
\end{align}
Then by \eqref{eq: chrosep} and \eqref{eq: kt' def} we have
\begin{eqnarray*}
|L_{s,k'} |
\geq  \hspace{-0.2cm} \sum_{t' : H_{t'} \in \cH'_{s,k'} } \hspace{-0.2cm}  \frac{1}{k}|L\cap V(H_{t'})|  - 8\Delta^{3k^3+2}\eta^{0.9} n 
 \stackrel{\ref{lem H4}}{\geq}  \epsilon^{-1/11} d_{s,k'}/k + \eta^{1/2}n/(2k) \geq  \Delta^{3} d_{s,k'}.
\end{eqnarray*}
For each $(s,k')\in [q]\times[k]$, we apply Lemma~\ref{eq: k-independent set} to $L_{s,k'}$ to obtain a subset of $L_{s,k'}$ with size exactly $d_{s,k'}$ which is $3$-independent in $H$. Write this $3$-independent set as 
\begin{align}\label{eq: pj def}
\{ a_{\ell}: \ell \in [n'], C^*_\ell = V(Q'_{s}) \text{ and } C_\ell = V(Q'_{s,k'}) \}.
\end{align}
This is possible by \eqref{eq: Q d def} and \eqref{eq: dellk sum} and defines vertices $a_1,\dots, a_{n'}$. 
Let $A:=\{a_1,\dots, a_{n'}\}$.
By \eqref{eq: Lsk' def} and \ref{lem H3}$, A$ is still a $3$-independent set in $H$.
As $a_\ell \in L$, we know that 
\begin{align}\label{eq: a ell degree}
d_{H}(a_\ell) \leq 2(1+1/h)m/n.
\end{align}
Moreover, for $\ell \in [n']$ and $t'\in [t]$, we have the following.
\begin{equation}\label{eq: pj location}
\begin{minipage}[c]{0.9\textwidth}
\em
If $a_\ell \in V_{t'}$, then $t'\in [t]\setminus [t_*]$ and $a_\ell \in W^{t'}_{k(t')} \setminus N^{3k^3+2}_H(V_0\cup W_0)$.
\end{minipage}
\end{equation}
\COMMENT{Indeed, if $H_{t'} \in \cH_{s,k'}$, then $a_\ell \in L_{s,k'} \cap V(H_{t'}) \subseteq  W^{t'}_{k(t')}$. Thus \eqref{eq: pj location} holds.}
In particular, we have $N_{H}(a_{\ell}) \cap N^{3k^3+1}_H(V_0\cup W_0) =\emptyset$. Thus if $a_{\ell} \in V_{t'}$, then 
\begin{align}\label{eq: pj neighbor location}
N_{H}(a_\ell) \subseteq \bigcup_{k'' \in [k]\setminus \{k(t')\}} W^{t'}_{k''} \setminus N^{3k^3+1}_H(V_0\cup W_0).
\end{align}
\vspace{0.05cm}

\noindent {\bf Step 3. Allocating the neighbourhood of $A$.}
We will allocate $N_H(A)$ to $Y$. 
We will achieve this by suitably allocating $V(\cH'_{s,k'})$ for each $(s,k') \in [q]\times [k]$. 
This will allocate $N_H(A)$ via \eqref{eq: pj neighbor location}.
Note that all choices until now are deterministic. Next we run the following random procedure.
\begin{equation}\label{eq: algorithm}
\begin{minipage}[c]{0.9\textwidth}
\em For each $t' \in [t]\setminus [t_*]$, let $(s,k')\in [q]\times [k]$ be such that $H_{t'} \in \cH'_{s,k'}$, and choose a permutation $\pi_{t'}$ on $[k]$ independently and uniformly at random among all permutations such that $\pi_{t'}(k') = k(t')$.
\end{minipage}
\end{equation}
(Note that this is the only place that our choice is random.)
Thus one value of $\pi_{t'}$ is fixed, while all other $k-1$ values are chosen at random. 
We choose $\pi_{t'}$ in this way because we wish to distribute $N_{H}(a_{\ell})$ to $\bigcup_{i\in C_{\ell}} Y_i$, so that later (B2)$_{\ref{lem: separable partition}}$ is satisfied. 
Setting $\pi_{t'}(k') = k(t')$ will ensure that no vertex in $N_{H}(a_{\ell})$ will be distributed to $Y_i$ with $i \in C^*_{\ell}\setminus C_{\ell}$. 
Moreover, as $\pi_{t'}$ is chosen uniformly at random,  $N_{H}(a_{\ell})$ will be distributed to $\bigcup_{i\in C_{\ell}} Y_i$ in a uniform way, which will guarantee that (B7)$_{\ref{lem: separable partition}}$ holds.

Indeed, for $\ell \in [n'], (s,k')\in [q]\times [k]$ and $t'\in [t]\setminus[t_*]$ such that $a_{\ell} \in L_{s,k'}\cap V_{t'}$, and
for any $k''\in [k]\setminus \{k'\}$, the number $\pi_{t'}(k'')$ is chosen uniformly at random among $[k]\setminus \{k(t')\}$, thus we have
\begin{eqnarray}\label{eq: expectation of the random}
\mathbb{E}[ |N_{H}(a_\ell) \cap W^{t'}_{\pi_{t'}(k'')}| ]\leq \frac{ d_{H}(a_\ell) }{k-1} \stackrel{\eqref{eq: a ell degree}}{\leq}\frac{2(1+1/h)m}{(k-1) n}.
\end{eqnarray}
For each $i\in [r]$, let 
\begin{align}\label{eq: tilde Yi def}
\tilde{Y}_i :=
\bigcup_{(s,k'): i= q'_s(k')}\bigcup_{k''\in [k]} \bigcup_{H_{t'} \in \cH'_{s,k''}} W^{t'}_{\pi_{t'}(k')} \setminus A \qquad \text{and} \qquad \tilde{Y}:=\bigcup_{i\in [r]} \tilde{Y}_i.
\end{align}
\vspace{0.05cm}

\noindent {\bf Step 4.  Allocating the remaining vertices to $X$ and $Y$.}
Later the vertices in $\tilde{Y}_i$ will be assigned to $Y_i$ (except those which are too close to $V_0$ in $H$, which will be assigned to $Z$). 
The sizes of the sets $X_i$ will be almost identical.
(Note that because of (B3)$_{\ref{lem: separable partition}}$, it is not possible to prescribe different sizes for $X_i$ and $X_j$ if $i$ and $j$ lie in the same copy of $K_k$ in $Q$.) 
Thus, in order to ensure (B5)$_{\ref{lem: separable partition}}$, we need to decide how many more vertices other than $\tilde{Y}_i$ we will assign to the set $Y_i$.
As part of this we now decide which of the $H_{t'}\in \cH$ are allocated to $X$ and which are allocated to $Y$ (again, vertices close to $V_0$ will be assigned to $Z$). 
Note that we have
\begin{eqnarray}\label{eq: tilde Yi size}
|\tilde{Y}_i| &\leq&  \sum_{(s,k'): i= q'_s(k')}\sum_{k''\in [k]}\sum_{H_{t'} \in \cH'_{s,k''}} |H_{t'}|
\stackrel{\ref{lem H4}}{\leq} \sum_{s: i\in V(Q'_s)} \sum_{ k''\in [k] } (\epsilon^{-1/10} d_{s,k''} + 3k\eta^{2/5}n)\nonumber \\
& \stackrel{\eqref{eq: dellk sum}}{\leq}& \epsilon^{-1/10} d_{\cF^*}(i) + 3k^2 q \eta^{2/5} n \stackrel{\text{(A3)$_{\ref{lem: separable partition}}$}}{\leq} \epsilon^{1/2}n/r  .
\end{eqnarray}

For each $i\in [r]$, let $\tilde{n}:= (1-2 \epsilon^{1/2})n/r$, and
\begin{eqnarray}\label{eq: tilde n}
 \tilde{n}_i &:=& n_i -\tilde{n} - |\tilde{Y}_i| \stackrel{\text{(A4)$_{\ref{lem: separable partition}}$} }{\leq}  \frac{\epsilon^{1/3}n}{(h+1)r}, \text{ then } 
 \tilde{n}_i  \stackrel{\text{(A4)$_{\ref{lem: separable partition}}$}}{\geq} \epsilon^{1/2}n/r - |\tilde{Y}_i|\stackrel{\eqref{eq: tilde Yi size} }{\geq} 0.
\end{eqnarray}
By applying Lemma~\ref{lem: hypergraph degree} with $R,h,\sigma,  \epsilon^{1/3}n/((h+1)r)$ and $\tilde{n}_i$ playing the roles of $G, t,\sigma, m$ and $d_v$, respectively, we obtain a multi-$k$-graph $\cF^{\#}$ on $[r]$ such that  for each $Q\in E(\cF^{\#})$, we have $R[Q]\simeq K_k$, and
\begin{equation}\label{eq: cF** degrees}
\begin{minipage}[c]{0.9\textwidth}\em
for each $i\in [r]$, we have $d_{\cF^{\#}}(i) = \tilde{n}_i+ \frac{\epsilon^{1/3}n}{r} \pm 1.$
\end{minipage}
\end{equation}
This implies
\begin{eqnarray}\label{eq: sum F sharp and tilde}
N&:=& \sum_{i\in [r]} ( \tilde{n} - \frac{\epsilon^{1/3}n}{r} + d_{\cF^{\#}}(i) )-|V_0\cup W_0| \stackrel{\eqref{eq: tilde n}}{=} \sum_{i\in [r]} (n_i  - |\tilde{Y}_i|  \pm 1) -|V_0\cup W_0| \nonumber \\ &\stackrel{\text{(A4)$_{\ref{lem: separable partition}}$}}{=}& n - n' - |\tilde{Y}| -|V_0\cup W_0| \pm r.
\end{eqnarray}
Note that we have
\begin{align}\label{eq: cH sizesese}
v(\cH) = |V(H) \setminus (\tilde{Y}\cup A\cup V_0\cup W_0)| = N \pm r.
\end{align}
Our target is to assign roughly $d_{\cF^{\#}}(i)$ extra vertices to $Y_i$ in addition to $\tilde{Y}_i$, and assign roughly $\tilde{n}-\frac{\epsilon^{1/3}n}{r}$ vertices to $X_i$, and a negligible amount of vertices to $Z_i$. Then $|X_i|+|Y_i| +|Z_i|$ will be close to $n_i$ as required in (B5)$_{\ref{lem: separable partition}}$.

To achieve this, we partition $\cH = \{H_1,\dots, H_{t_*}\}$ into 
$\cH_1,\dots, \cH_{r'}, \cH^{\#}_1,\dots,\cH^{\#}_q$ satisfying the following for all $i \in [r']$ and $s \in [q]$.
 \begin{enumerate}[label=\text{\rm (H{\arabic*})$_{\ref{lem: separable partition}}$}]
 \setcounter{enumi}{4}
\item \label{lem H5}  $\displaystyle v(\cH_i) = k \tilde{n} - \frac{k\epsilon^{1/3}n}{r} \pm \eta^{2/5} n$ and $\displaystyle e(\cH_i) = \frac{k(m \pm \epsilon^{2/7}n)}{r}$,
\item\label{lem H6} $v(\cH^{\#}_{s}) = k \cdot \text{mult}_{\cF^{\#}}(V(Q'_s)) \pm \eta^{2/5} n.$
\end{enumerate}
(Recall that $\text{mult}_{\cF^{\#}}(V(Q'_s))$ denotes the multiplicity of the edge $V(Q'_s)$ in $\cF^{\#}$.)
Indeed, such a partition exists by the following claim. 
\begin{claim}\label{cl: H5H6 exist}
There exists a partition $\cH_1,\dots, \cH_{r'}, \cH^{\#}_1,\dots,\cH^{\#}_q$ of $\{H_1,\dots, H_{t_*}\}$ satisfying \ref{lem H5}-- \ref{lem H6}.
\end{claim}
\begin{proof}
For each $t'\in [t_*]$, we choose $i_{t'}$ independently at random from $\{(0,1),\dots,(0,r'), (1,1),\dots, (1,q)\}$ such that for each $i \in [r']$ and $s\in [q]$:
$$ \mathbb{P}[ i_{t'} = (0,i) ] = \frac{ k \tilde{n} - \frac{k\epsilon^{1/3}n}{r} - \frac{k|V_0\cup W_0|}{r}}{N} \enspace \text{ and } \enspace \mathbb{P}[i_{t'} =  (1,s) ] = \frac{k\cdot \text{mult}_{\cF^{\#}}(V(Q'_s))}{N} .$$
Since $\sum_{s\in [q]} k \cdot \text{mult}_{\cF^{\#}}(V(Q'_s)) = k |E(\cF^{\#})| = \sum_{i\in [r]} d_{\cF^{\#}}(i)$, an easy calculation based on \eqref{eq: sum F sharp and tilde} shows that this defines a probability distribution.
For all $i \in [r']$ and $s \in [q]$, we let
$$\cH_i:= \{H_{t'} : t'\in [t_*], i_{t'}= (0,i) \} \enspace \text{ and } \enspace \cH^{\#}_{s}:=\{ H_{t'} :t'\in [t_*], i_{t'} = (1,s)\}.$$
Then it is easy to combine a Chernoff bound (Lemma~\ref{lem: chernoff}) with \ref{lem H1}, \ref{lem H2} and \eqref{eq: cH sizesese} to check that the resulting partition satisfies \ref{lem H5} and \ref{lem H6} with positive probability.
 This proves the claim.
\end{proof}

By permuting indices on $[t_*]$, we may assume that for some $t^*\in [t_*]$ we have
$$\bigcup_{i\in [r']} \cH_{i} = \{H_1,\dots, H_{t^*}\} \enspace \text{and} \enspace \bigcup_{s\in [q]} \cH^{\#}_{s} = \{H_{t^*+1},\dots, H_{t_*}\}.$$

In order to obtain (B3)$_{\ref{lem: separable partition}}$--(B5)$_{\ref{lem: separable partition}}$, we need to distribute vertices of the graphs in $\cH_{i}$ into $\{ X_j : j\in V(Q_i)\}$ and vertices of the graphs in $\cH^{\#}_s$ into $\{ Y_j : j\in V(Q'_s)\}$
so that the resulting vertex sets and edge sets are evenly balanced.
 For this, we define a permutation $\pi_{t'}$ on $[k]$ for each $t'\in [t_*]$ which will determine how we will distribute these vertices.
We will choose these permutations $\pi_1,\dots, \pi_{t_*}$ such that the following hold for all $i\in [r']$, $s\in [q]$ and $k'\neq k'' \in [k]$.
 \begin{enumerate}[label=\text{\rm (H{\arabic*})$_{\ref{lem: separable partition}}$}]
 \setcounter{enumi}{6}
\item \label{lem H7}  {
$\displaystyle \hspace{-0.3cm} \sum_{t': H_{t'} \in \cH_{i}} \hspace{-0.2cm}   | W^{t'}_{\pi_{t'}(k')}| =  \tilde{n} - \frac{\epsilon^{1/3}n}{r} \pm \eta^{2/5} n \enspace \text{ and }\hspace{-0.1cm} 
\sum_{t': H_{t'} \in \cH_{i}} \hspace{-0.2cm}  |E_H(W^{t'}_{\pi_{t'}(k')}, W^{t'}_{\pi_{t'}(k'')}) |  =\frac{2m \pm \epsilon^{1/4}n} {(k-1)r},$
}
\item \label{lem H8}  {
$\displaystyle \hspace{-0.3cm} \sum_{t': H_{t'} \in \cH_{s}^{\#}} \hspace{-0.2cm}   | W^{t'}_{\pi_{t'}(k')}| = \text{mult}_{\cF^{\#}}(V(Q'_s)) \pm \eta^{2/5} n.$}
\end{enumerate}
To see that such permutations exist we consider for each $t' \in [t_*]$ a permutation $\pi_{t'}:[k]\rightarrow [k]$ chosen independently and uniformly at random.\COMMENT{
Then for all $i\in [r']$ and $k'\neq k'' \in [k]$ and $s\in [q]$, we have
{\small 
\begin{eqnarray*}
\mathbb{E}[ \hspace{0.1cm} \sum_{t': H_{t'} \in \cH_{i}} \hspace{0.1cm}  |W^{t'}_{\pi_{t'}(k')}| ]
&=&  \frac{1}{k}  \sum_{t': H_{t'} \in \cH_{i}} |V(H_{t'})| \stackrel{\ref{lem H5}}{=}  \tilde{n} - \frac{\epsilon^{1/3}n}{r} \pm \eta^{2/5}n/k ,\\
\mathbb{E}[\hspace{-0.2cm} \sum_{t': H_{t'} \in \cH_{i}} \hspace{-0.3cm}  |E_H(W^{t'}_{\pi_{t'}(k')}, W^{t'}_{\pi_{t'}(k'')}) | ] &=& \frac{2}{k(k-1)} \sum_{t': H_{t'} \in \cH_{i}} |E(H_{t'})| \stackrel{\ref{lem H5}}{=} \frac{2(m\pm \epsilon^{2/7}n)}{(k-1)r},\\
\mathbb{E}[\hspace{-0.2cm} \sum_{t': H_{t'} \in \cH_{s}^{\#}} \hspace{-0.1cm}   |W^{t'}_{\pi_{t'}(k')}|  &=& \frac{1}{k}  \sum_{t': H_{t'} \in \cH^{\#}_{s}} |V(H_{t'})| \stackrel{\ref{lem H6}}{=} \text{mult}_{\cF^{\#}}(V(Q'_s)) \pm \eta^{2/5} n/k,\\
\end{eqnarray*}}
By using union bound, and Chernoff's inequality(Lemma~\ref{lem: chernoff}) with \ref{lem H1}, \ref{lem H2} we conclude that 
\begin{eqnarray*}
\mathbb{P}[ \text{\ref{lem H7}, \ref{lem H8} hold} ]
& \geq & 1 - 2 k t_* \exp( - \frac{(\eta^{2/5}n/2)^2}{3t (2 \eta^{0.9} n)^2}) -2k^2t_*\exp( - \frac{(\epsilon^{1/4}n/(2(k-1)r))^2}{3t (2\Delta \eta^{0.9} n)^2 }) 
\geq 1- \eta.
\end{eqnarray*}}
Then, by a Chernoff bound (Lemma~\ref{lem: chernoff}) combined with \ref{lem H1} and \ref{lem H2}, it is easy to check that $\pi_1,\dots, \pi_{t_*}$ satisfy \ref{lem H7}  and \ref{lem H8} with positive probability. \newline 

\noindent {\bf Step 5. Clique walks.}
Recall that $V_0$ is a separator of both $H$ and $H\setminus W_0$.
The vertices in $V_0$ will be allocated to the sets $Z_1, \dots, Z_k$ which initially correspond to the clique $Q_1' \subseteq R$ (recall that $V(Q'_1)= \{1,\dots, k\}$).
We now identify an underlying structure in $R$ that will be used in Step 6 to ensure that while allocating $V(H)\backslash (V_0\cup W_0\cup  A)$ to $X$, $Y$ and $Z$, we do not violate the vertex partition admitted by $R$ (c.f.~ (B3)$_{\ref{lem: separable partition}}$). (This is a particular issue when considering edges between separator vertices and the rest of the partition.) 

To illustrate this, let $s\in S$ be a separator vertex allocated to $Z_{k'}$. 
Let $x$ be some vertex in some $H_{t'}$ with $xs \in E(H)$.
Suppose $H_{t'}$ is assigned to some clique $Q_i\subseteq Q$ and that this would assign $x$ to some set $X_{i'}$, where $i'\in V(Q_i)$.
Furthermore, suppose $i'k'$ is not an edge in $R$.
We cannot simply reassign $x$ to another set $X_j$ to obey the vertex partition admitted by $R$ without also considering the neighbourhood of $x$ in $H_{t'}$.
To resolve this, we apply Lemma~\ref{lem: clique extension} to obtain a suitable `clique walk' $P$ between $Q_1'$ and $Q_i$,
i.e.~the initial sement of $P$ is $V(Q_1')$, its final segment is $V(Q_i)$ and each segment of $k$ consecutive vertices in $P$ corresponds to a $k$-clique in $R$.
We initially assign $x$ to a set $Z_{k''}$ for some $k'' \in [k] \setminus  \{ k' \}$.
We then assign the vertices which are close to $x$ to some $Z_{k'''}$, where 
the choice of $k''' \in [r]$ is determined by $P$. (In order to connect $Y$ to $V_0$, we also choose similar clique walks starting  with $Q'_1$ and ending with $Q'_s$ for each $s\in [q]$.)

To define the clique walks formally, for each $t'\in [t]$, let 
\begin{align}\label{eq: F sub t' def}
P_{t'}:= \left\{\begin{array}{ll}
Q_{i} & \text{ if } H_{t'} \in \cH_{i} \text{ for some } i\in [r'], \\
Q'_{s} & \text{ if } H_{t'} \in \cH^{\#}_{s} \text{ for some } s \in [q],\\
Q'_{s} & \text{ if } H_{t'} \in \cH'_{s,k'} \text{ for some } (s, k') \in [q]\times [k],
\end{array}\right. \text{ and }
\left.\begin{array}{l}
\{ p_{t'}(1), \dots, p_{t'}(k) \}:=P_{t'}, \\ 
\text{where } p_{t'}(1)< \dots < p_{t'}(k).
\end{array}\right.
\end{align}
By using (A1)$_{\ref{lem: separable partition}}$, we can apply Lemma~\ref{lem: clique extension} for each $t'\in [t]$ with $V(Q'_1)$ and $V(P_{t'})$ playing the roles of $Q_1$ and $Q_2$  in order to obtain\COMMENT{ 
{\small
\noindent
{ 
\begin{tabular}{c|c|c|c|c}
object/parameter & $V(Q'_1)$ & $V(P_{t'})$ & $\pi_{t'}(k'): k'\in [k]$ & $p_{t'}(k') : k'\in [k]$  \\ \hline
playing the role of & $Q_1$ & $Q_2$ & $x_{k'} : k'\in [k]$ & $y_{k'} : k'\in [k]$ 
\end{tabular}
}}\vspace{0.2cm}}
a walk $j(t',1),\dots, j(t', b_{t'} k )$ in $R$ such that
\begin{equation}\label{eq: ell property}
\begin{minipage}[c]{0.9\textwidth}\em
for all distinct $i,i' \in [b_{t'}k ]$ with $|i-i' |\leq k-1$, we have $j(t',i)j(t',i') \in E(R)$, and 
for each $k'\in [k]$ we have $j(t',k')= \pi_{t'}(k')$ and $j(t',(b_{t'}-1)k+k') = p_{t'}(k')$.
\end{minipage}
\end{equation}
Moreover, for each $t'\in [t]$, we have
\begin{align}\label{eq: bt size}
3\leq b_{t'} \leq 3k^2.
\end{align}
As described above we will later distribute some vertices of $V_{t'}\cap N^{(b_{t'}-1)k}(V_0)$ to $\bigcup_{k'\in [(b_{t'} -1)k]} Z_{j(t',k')}$ so that we can ensure  (B3)$_{\ref{lem: separable partition}}$ and  (B6)$_{\ref{lem: separable partition}}$ hold. \newline

\noindent {\bf Step 6. Iterative construction of the partition.}
Now, we will distribute the vertices of each $H_{t'}$ into $X_1,\dots, X_r, Y_1,\dots, Y_r, Z_1,\dots, Z_r$ in such a way that (B1)$_{\ref{lem: separable partition}}$--(B7)$_{\ref{lem: separable partition}}$ hold. (In particular, as discussed earlier, we will have $\widetilde{Y}_i\subseteq Y_i$.) To achieve this,
for each $t' =0,1,\dots, t$, we iteratively define sets $X^{t'}_1,\dots, X^{t'}_r,$ $ Y^{t'}_{1},\dots, Y^{t'}_{r},$ $Z^{t'}_1,\dots, Z^{t'}_{r}.$
First, for each $k'\in [k]$, let
$Z^{0}_{k'}:= W^{0}_{k'}$
and for all $i\in [r]$ and $i'\in [r]\setminus [k]$, let 
$$X^{0}_i:=\emptyset, \enspace Y^{0}_i:=\emptyset \enspace \text{and} \enspace 
 Z^{0}_{i'}:=\emptyset.$$
We will write
$$V^{t'}:=\bigcup_{t''=0}^{t'} V_{t''}, \quad X^{t'}:= \bigcup_{i \in [r]} X^{t'}_i, \quad Y^{t'}:=\bigcup_{i\in [r]} Y^{t'}_i \text{\quad  and \quad } Z^{t'}:=\bigcup_{i\in [r]} Z^{t'}_i .$$
Assume that for some $t'\in [t]$, we have already defined a partition $X^{t'-1}_1,\dots, X^{t'-1}_r,$
$Y^{t'-1}_{1},\dots, Y^{t'-1}_{r},$ $Z^{t'-1}_1,\dots, Z^{t'-1}_{r}$ of $V^{t'-1}$ satisfying the following. \newline
\begin{enumerate}[label=\text{\rm (Z{\arabic*})$^{t'-1}_{\ref{lem: separable partition}}$}]
\item  For all $i'\in [r']$ and $i\in V(Q_{i'})$, let $k'$ be so that $i= q_{i'}(k')$.
Then we have
$$\bigcup_{ t'' \in [t'-1]: H_{t''} \in \cH_{i'}} \hspace{-0.2cm} W^{t''}_{\pi_{t''}(k')}  \setminus N^{ (b_{t''}-1)k}_{H}(V_0) \subseteq X^{t'-1}_{i}
\subseteq \hspace{-0.2cm} \bigcup_{ t'' \in [t'-1]: H_{t''} \in \cH_{i'}} \hspace{-0.2cm} W^{t''}_{\pi_{t''}(k')},$$
\item for each $i\in [r]$, we have
$$ \bigcup_{k'\in[k]} \bigcup_{\enspace \substack{ t''\in [t'-1]\setminus[t^*] : \\ p_{t''}(k')=i} }   W^{t''}_{\pi_{t''}(k')} \setminus N^{(b_{t''}-1)k}_{H}(V_0)  \subseteq Y^{t'-1}_{i}\subseteq \bigcup_{k'\in[k]} \bigcup_{\enspace \substack{ t''\in [t'-1]\setminus[t^*]  : \\ p_{t''}(k')=i} }   W^{t''}_{\pi_{t''}(k')},$$
\item for all $ij\notin E(Q)$, we have 
$e_H(X^{t'-1}_i, X^{t'-1}_j) = 0$,
\item for all $ij \notin E(R)$, we have $e_H(X^{t'-1}_i, Z^{t'-1}_j) = e_H(Y^{t'-1}_i, Z^{t'-1}_j) = e_H(Y^{t'-1}_i, Y^{t'-1}_j)= e_H(Z^{t'-1}_i, Z^{t'-1}_j)= 0,$
\item $N^{1}_H(X^{t'-1}) \setminus X^{t'-1} \subseteq Z^{t'-1} \subseteq N^{3k^3}_H( V_0)$,
\item for each $k'\in [k]$, we have $W^0_{k'} \subseteq Z^{t'-1}_{k'}$,
\item for each $t''\in [t'-1]$, we have
$|\{ i\in [r]:  (X_i^{t'-1}\cup Y_i^{t'-1}) \cap V_{t''} \neq \emptyset \}| \leq k.$
\end{enumerate}
Using that $Q'_1$ is a copy of $K_k$ in $R$ and $V(Q'_1) =\{1,\dots, k\}$, it is easy to see that (Z1)$^{0}_{\ref{lem: separable partition}}$--(Z7)$^{0}_{\ref{lem: separable partition}}$ hold with the above definition of $X^{0}_i, Y^{0}_i, Z^{0}_i$.
We now distribute the vertices of $H_{t'}$ by setting
{
\begin{align*}
X^{t'}_{i}&:= \left\{ \begin{array} {ll}
X^{t'-1}_{i} \cup \left(W^{t'}_{\pi_{t'}(k')} \setminus N^{(b_{t'} -2)k+k'}_{H}(V_0) \right) & \text{ if }t'\in [t^*] \text{ and } i=p_{t'}(k') \text{ for some } k'\in [k],\\
X^{t'-1}_{i} & \text{ otherwise,}
\end{array}\right.\\
Y^{t'}_{i}&:= \left\{ \begin{array} {ll}
Y^{t'-1}_{i} \cup \left(W^{t'}_{\pi_{t'}(k')} \setminus N^{(b_{t'} -2)k+k'}_{H}(V_0) \right) & \text{ if } t'\in [t]\setminus[t^*]  \text{ and } i = p_{t'}(k') \text{ for some } k'\in [k],\\
Y^{t'-1}_{i} & \text{ otherwise,}
\end{array}\right.\\
Z^{t'}_{i}&:= 
Z^{t'-1}_{i} \cup \bigcup_{ \substack{(b,k') \in [b_{t'}-1]\times[k] : \\ i=j(t',(b-1)k+k')}} \left(W^{t'}_{\pi_{t'}(k')} \cap \left( N^{(b-1)k+k' }_{H}(V_0)\setminus N^{(b-2)k+k'}_{H}(V_0) \right)\right).
\end{align*}}
Let $H':=H \setminus W_0$. Recall that $N^{3k^3+1}_H(V_0)$ does not contain any vertex in $W_0$ (see \eqref{eq: chrosep}).
Hence $N^i_{H}(V_0)=N^i_{H'}(V_0)$ for any $i\leq 3k^3+1$.

Note that the above definition of $X^{t'}_{i}, Y^{t'}_{i}, Z^{t'}_{i}$ uniquely distributes all vertices of $V^{t'}$.
Indeed, first note that either $Y^{t'}_i = Y^{t'-1}_i$ for all $i\in [r]$ or
$X^{t'}_i= X^{t'-1}_i$ for all $i\in [r]$ depending on whether $H_{t'}\in \cH_{c}$ for some $c\in [r']$ (in which case $t'\in [t^*]$)
or $H_{t'}\in \cH^{\#}_s$ for some $s\in [q]$ or $H_{t'}\in \cH'_{s,k'}$ for some $(s,k')\in [q]\times [k]$  (in the latter two cases we have $t'\in [t]\setminus [t^*]$).
Now, consider $W^{t'}_{k''}\cap (N^{a}_H(V_0) \setminus N_H^{a-1}(V_0))$ for $k''\in [k]$ and $a\in \mathbb{N}$. 
Note $k'' = \pi_{t'}(k')$ for some $k'\in [k]$. 
Then either $a >(b_{t'}-2)k+k'$ or $a\in [(b'-1)k+k']\setminus [(b'-2)k+k']$ for some unique $b' \in [b_{t'} -1]$.
 Thus indeed every vertex of $V^{t'}$ belongs to exactly one of $X^{t'}_i$ or $Y^{t'}_i$ or $Z^{t'}_i$.

It is easy to see that the above definition with \eqref{eq: bt size},
(Z1)$^{t'-1}_{\ref{lem: separable partition}}$ and (Z2)$^{t'-1}_{\ref{lem: separable partition}}$ implies (Z1)$^{t'}_{\ref{lem: separable partition}}$ and (Z2)$^{t'}_{\ref{lem: separable partition}}$.
Also, (Z7)$^{t'}_{\ref{lem: separable partition}}$ is obvious from the construction.
Moreover, (Z3)$^{t'-1}_{\ref{lem: separable partition}}$ and \ref{lem H3} imply (Z3)$^{t'}_{\ref{lem: separable partition}}$ while 
(Z6)$^{t'-1}_{\ref{lem: separable partition}}$ implies (Z6)$^{t'}_{\ref{lem: separable partition}}$.
Similarly, we have $e_H(Y_i^{t'}, Y_j^{t'}) = 0$ if $ij \notin E(R).$
We now verify the remaining assertions of (Z4)$^{t'}_{\ref{lem: separable partition}}$.
First suppose that
$$E_H(X^{t'}_i,Z^{t'}_{i'})\setminus E_H(X^{t'-1}_i,Z^{t'-1}_{i'}) \neq \emptyset \enspace \text{or} \enspace
E_H(Y^{t'}_i,Z^{t'}_{i'})\setminus E_H(Y^{t'-1}_i,Z^{t'-1}_{i'}) \neq \emptyset.$$
Then by \ref{lem H3}, 
we have $i= p_{t'}(k')$ for some $k'\in [k]$ and 
$i' = j(t',(b-1)k+k'')$ for some $k''\in [k]$ and $b\in [b_{t'}-1]$, and $H$ contains an edge between 
$$W^{t'}_{\pi_{t'}(k')} \setminus N^{(b_{t'}-2)k+k'}_H(V_0) \text{ and }
W^{t'}_{\pi_{t'}(k'')} \cap N^{(b-1)k+k''}_H(V_0).$$ This means that
$(b_{t'}-2)k+k' \leq (b-1)k + k''$.
Thus $b= b_{t'}-1$ and $k' \leq k''$.
Moreover, since $W^{t'}_{\pi_{t'}(k')}$ is an independent set of $H$, we have $k'\neq k''$. 
Since \eqref{eq: ell property} implies that 
$i= p_{t'}(k') = j(t',(b_{t'}-1)k+k')$ and $i' = j(t', (b_{t'}-2)k+k'')$ with 
$0<  (b_{t'}-1)k+k' - ( (b_{t'}-2)k+k'' ) < k$,
again this with \eqref{eq: ell property} implies that $ii' \in E(R)$.
Now suppose that
$$xy \in E_H(Z^{t'}_i, Z_{i'}^{t'})\setminus E_H(Z_i^{t'-1}, Z_{i'}^{t'-1}) \text{ with } x,y \notin V_0.$$
Then by \ref{lem H3}, we have
$i=  j(t',(b-1)k+ k')$ and $i'= j(t',(b'-1)k + k'')$ for some $b,b'\in [b_{t}-1]$ and
$k'\neq k''\in [k]$. 
However, the definition of $Z_{i}^{t'}$ implies that such an edge only exists when  $|((b-1)k+k') - ((b'-1)k+k'')| \leq k-1$. In this case, \eqref{eq: ell property} implies that $ii' \in E(R)$. 
Finally, suppose that $$xy \in E_H(Z^{t'}_i, Z_{i'}^{t'})\setminus E_H(Z_i^{t'-1}, Z_{i'}^{t'-1}) \text{ with } x \in V_0\cap Z^{t'}_i.$$
Then the definition of $Z_{i}^{t'}$ implies that $i \in [k]$, $x \in W_i^0$ and
$i' = j(t', k')$ for some $k' \in [k]$.
 \eqref{eq: ell property} implies that $j(t',k')=\pi_{t'}(k')$.
As $W^{0}_{\pi_{t'}(k')}\cup W^{t'}_{\pi_{t'}(k')}$ is an independent set of $H$, we have $i \ne \pi_{t'}(k')$.  
However, as $R[[k]] = R[V(Q'_1)]\simeq K_k$, we know that $ii'\in E(R)$.
Thus (Z4)$^{t'}_{\ref{lem: separable partition}}$ holds. 
By the definition of $X^{t'}_{i}$ and $Z^{t'}_{i}$ with \eqref{eq: bt size}, it is obvious that  (Z5)$^{t'}_{\ref{lem: separable partition}}$ holds too.

Thus, by repeating this, we obtain a partition $X^{t}_1,\dots, X^{t}_r, Y^{t}_1,\dots, Y^{t}_r, Z^{t}_{1},\dots, Z^{t}_{r}$ of $V(H)\setminus W_0$ satisfying (Z1)$^{t}_{\ref{lem: separable partition}}$--(Z7)$^{t}_{\ref{lem: separable partition}}$. For each $i\in [r]$, let
$$X_i:= X^{t}_i, \enspace X:= X^{t}, \enspace Y_i:= Y^{t}_i\setminus A, \enspace Y:= Y^{t} \setminus A, \enspace {Z'_i}:=Z^{t}_i  \enspace \text{and} \enspace Z':=Z^{t}.$$
Note that $A \subseteq Y^t$ by \eqref{eq: pj location} and (Z2)$^{t}_{\ref{lem: separable partition}}$. 
Moreover, $X,Y,Z',A$ forms a partition of $V(H)\setminus W_0$. Now we consider the vertices in $W_0$. 
For each $w\in W_0$, let 
$$I_w:=\{ i\in [r]: N_H(w)\cap (X_i\cup Y_i)\neq \emptyset \}.$$
By \eqref{eq: chrosep}, we have $W_0\cap V_0=\emptyset$. Hence, for each vertex $w\in W_0$, there exists $t'\in [t]$ such that $w \in \widetilde{V}_{t'}$. As $W_0$ is an independent set, \eqref{eq: chrosep} with \ref{lem H3} implies $N_H(w) \subseteq V_{t'}$.
This with (Z7)$^{t}_{\ref{lem: separable partition}}$ implies that 
$|I_w|\leq k$. As $|N_R(I_w)|>0$ by (A1)$_{\ref{lem: separable partition}}$, we can assign $w$ to $Z'_i$ for some $i\in N_R(I_w)$.
Let $Z_1,\dots,Z_r,Z$ be the sets obtained from $Z'_1,\dots,Z'_r,Z'$ by assigning all vertices in $W_0$ in this way.
By \eqref{eq: chrosep}, \eqref{eq: pj location} and (Z5)$^{t}_{\ref{lem: separable partition}}$ for each $w\in W_0$ we have $N_H(w)\subseteq X\cup Y$. Thus
 
\begin{equation}\label{eq: W0 Z}
\begin{minipage}{0.9\textwidth} \em
for all $i\in [r]$, $w\in W_0\cap Z_i$ and $x\in N_H(w)$, we have $x\in X_j\cup Y_j$ for some $j\in N_{R}(i)$.
\end{minipage}
\end{equation}
The sets $X, Y, Z, A$ now form a partition of $V(H)$.
\newline

\noindent {\bf Step 7. Checking the properties of the partition.}
We now verify that this partition satisfies (B1)$_{\ref{lem: separable partition}}$-(B7)$_{\ref{lem: separable partition}}$.
Note that \eqref{eq: a ell degree} implies (B1)$_{\ref{lem: separable partition}}$. Consider any $\ell\in [n']$, and let $t'\in [t]\setminus [t_*]$ and $(s,k')\in [q]\times[k]$ be such that $a_{\ell} \in H_{t'} \in \cH'_{s,k'}$. Then
\begin{eqnarray*}
N_H(a_{\ell}) \hspace{-0.3cm} &\stackrel{\eqref{eq: pj neighbor location}}{\subseteq}& \hspace{-0.3cm}  \bigcup_{k''\in [k]\setminus \{k(t')\}} W^{t'}_{k''} \setminus N^{3k^3+1}_H(V_0\cup W_0)  \stackrel{\eqref{eq: algorithm}}{=}  \bigcup_{k''\in [k]\setminus \{k'\} } W^{t'}_{\pi_{t'}(k'')} \setminus  N^{3k^3+1}_H(V_0\cup W_0)  \\
\hspace{-0.3cm} &\stackrel{\text{(Z2)$^{t}_{\ref{lem: separable partition}}$, (Z5)$^{t}_{\ref{lem: separable partition}}$}}{\subseteq}& \hspace{-0.3cm} \bigcup_{k''\in [k]\setminus \{k'\} } Y_{p_{t'}(k'')} \setminus N^{1}_H(Z) \stackrel{\eqref{eq: Q d def},\eqref{eq: F sub t' def}}{=} \bigcup_{i \in V(Q'_{s,k'}) } Y_{i}\setminus N^{1}_H(Z) \stackrel{\eqref{eq: pj def}}{=} \bigcup_{i\in C_\ell} Y_{i} \setminus N^{1}_H(Z).
\end{eqnarray*}
This proves (B2)$_{\ref{lem: separable partition}}$.
Moreover, whenever $\ell,t'$ and $(s,k')$ are as in the proof of (B2)$_{\ref{lem: separable partition}}$,
for each $j' \in C_\ell$, we have $j'= p_{t'}(k'')$ for some $k''\in [k]\setminus \{k'\}$. 
Thus by \eqref{eq: pj neighbor location} and (Z2)$^{t}_{\ref{lem: separable partition}}$, we have
$$\mathbb{E}[ |N_H(a_\ell)\cap Y_{j'}|] 
\leq \mathbb{E}[ |N_H(a_\ell)\cap W^{t'}_{\pi_{t'}(k'')}| ] \stackrel{\eqref{eq: expectation of the random}}{\leq} \frac{2(1+1/h)m}{(k-1)n}.$$
This proves (B7)$_{\ref{lem: separable partition}}$.

Properties (Z3)$^{t}_{\ref{lem: separable partition}}$, (Z4)$^{t}_{\ref{lem: separable partition}}$, (Z5)$^{t}_{\ref{lem: separable partition}}$ and \eqref{eq: W0 Z}  imply (B3)$_{\ref{lem: separable partition}}$.

For each $ij \in E(Q)$, let $s\in [r']$ and $ k',k''\in [k]$ be such that $i=q_s(k')$ and $j=q_s(k'')$. Thus 
\begin{eqnarray*}
e_H(X_i,X_j) \hspace{-0.2cm} &\stackrel{\ref{lem H3},\text{(Z1)$^{t}_{\ref{lem: separable partition}}$}  }{=}&\hspace{-0.4cm} \sum_{t'\in  [t^*]: H_{t'}\in \cH_{s} }|E_H(W^{t'}_{\pi_{t'}(k')}, W^{t'}_{\pi_{t'}(k'')} )| \pm \Delta |N^{3k^3}_H(V_0)| 
\stackrel{\ref{lem H2}, \ref{lem H7}}{=}\frac{2m \pm \epsilon^{1/5} n}{(k-1)r}.
\end{eqnarray*}
Thus \rm(B4)$_{\ref{lem: separable partition}}$ holds.
Moreover, given $i\in [r]$, let $s\in [r']$ and $k'\in [k]$ be such that $i=q_s(k')$. Then
\begin{eqnarray*}
|X_i| \hspace{-0.2cm}&\stackrel{\text{(Z1)$^{t}_{\ref{lem: separable partition}}$}  }{=}&\hspace{-0.5cm} \sum_{t'\in [t^*]: H_{t'}\in \cH_{s} }\hspace{-0.3cm} |W^{t'}_{\pi_{t'}(k')}| \pm |N^{3k^3}_H(V_0)| 
\stackrel{\ref{lem H7}}{=} \tilde{n}- \epsilon^{1/3}n/r \pm \eta^{1/3}n.
\end{eqnarray*}
Similarly, for $i\in [r]$, since by \eqref{eq: pj location} the vertices of $A$ only belong to $V(H_{t'})$ for $t'\in [t]\setminus [t_*]$,
\begin{eqnarray*}
|Y_{i}| \hspace{-0.2cm} &\stackrel{\text{(Z2)$^{t}_{\ref{lem: separable partition}}$}  }{= }&\hspace{-0.5cm}
\sum_{(t',k') : p_{t'}(k')=i, t'\in [t]\setminus [t^*] } |W^{t'}_{\pi_{t'}(k')} \setminus A|
  \pm |N^{3k^3}_H(V_0)| \\
&\stackrel{\eqref{eq: F sub t' def}}{=}	& \hspace{-0.5cm} 
  \sum_{ (s,k') : q'_{s}(k')= i} \sum_{~ t': H_{t'}\in \cH^{\#}_s}  \hspace{-0.3cm}  |W^{t'}_{\pi_{t'}(k')}|
+ \sum_{ (s,k') : q'_{s}(k')= i} \sum_{k''\in [k]} \sum_{~ t': H_{t'} \in \cH'_{s,k''}} \hspace{-0.3cm}  |W^{t'}_{\pi_{t'}(k')}\setminus A|
  \pm  \eta^{1/2}n  \\
  &\stackrel{\ref{lem H8},\eqref{eq: tilde Yi def}}{= }& \hspace{-0.5cm} 
\sum_{ (s,k') : q'_{s}(k')= i}  \text{mult}_{\cF^{\#}}(V(Q'_s))
+ |\tilde{Y}_i|   \pm 2q \eta^{2/5}n = d_{\cF^{\#}}(i) + |\tilde{Y}_i| \pm 2q\eta^{2/5}n \\
 &\stackrel{ \eqref{eq: tilde n},\eqref{eq: cF** degrees}}{=}&   n_i -\tilde{n} + \epsilon^{1/3}n/r \pm \eta^{1/3} n.
\end{eqnarray*}
Together with \eqref{eq: chrosep}, (Z5)$^{t}_{\ref{lem: separable partition}}$ and \ref{lem H2},\COMMENT{and the fact that $\eta \ll 1/r$ and $q\leq r^k$} this now implies that for each $i\in [r]$
$$|X_i|+|Y_i|+|Z_i| = n_i \pm \eta^{1/4}n.$$ 
Also, the definition of $\tilde{n}$ with (A4)$_{\ref{lem: separable partition}}$ implies that $|Y_i| \leq 2\epsilon^{1/3}n/r$.
Thus (B5)$_{\ref{lem: separable partition}}$ holds.
Finally, \eqref{eq: chrosep} and
(Z5)$_{\ref{lem: separable partition}}$ imply (B6)$_{\ref{lem: separable partition}}$. 
\end{proof}

\section{Packing graphs into a super-regular blow-up}\label{sec: main lemma}

In this section, we prove our main lemma. 
Roughly speaking, this lemma says the following. 
Suppose we have disjoint vertex sets $V$, $Res_t$ and $V_0$ and suppose that we have a super-regular $K_k$-factor blow-up $G[V]$ on vertex set $V$, and suitable graphs $G[Res_t],$ $G[V,Res_t],$ $F[V,Res_t]$ and $F'[Res_t,V_0]$ are also provided. 
Then we can pack an appropriate collection $\cH$ of graphs into $G\cup F\cup F'$.
 Here $V_0$ is the exceptional set obtained from an application of Szemer\'edi's regularity lemma and $Res_t$ is a suitable `reservoir' set where $V_0$ is much smaller than $Res_t$, which in turn is much smaller than $V$. 
The $k$-cliques provided by the multi-$k$-graph $\cC^*_t$ below will allow us to find a suitable embedding of the neighbours of the vertices mapped to $V_0$.
When we apply Lemma \ref{embedone} in Section 6, the reservoir set $Res_t$ will play the role of the set $U\cup U_0$ below.
$U_0$ will correspond to a set of exceptional vertices in $Res_t$.
{\rm (A9)}$_{\ref{embedone}}$ will allow us to embed the neighbours of the vertices mapped to $U_0$.

Note that the packing $\phi$ is designed to cover most of the edges of the blown-up $K_k$-factor $G[V]$, but only covers a small proportion of the edges of $G$ incident to $U.$
{\rm (A7)}$_{\ref{embedone}}$ provides the edges incident to the vertices mapped to $V_0$, and {\rm (A8)}$_{\ref{embedone}}$ allows us to embed the neighbourhoods of these vertices.

\begin{lemma}
\label{embedone}
Suppose $n,n',k,\Delta, r, T \in \mathbb{N}$ with 
$0 < 1/n, 1/n' \ll \eta \ll   \epsilon \ll 1/T \ll  \alpha \ll d \ll 1/k, \sigma, \nu,1/\Delta< 1$ and $\eta \ll 1/r \ll \sigma$ and
$k \mid r$.
Suppose that $R$ and $Q$ are graphs with $V(R)=V(Q)=[r]$ such that  
$Q$ is a union of $r/k$ vertex-disjoint copies of $K_k$.
Suppose that $V_0,\dots, V_r, U_0, \dots, U_r$ is a partition of a set of $n$ vertices
such that $|V_0|\leq \epsilon n$, $|U_0|\leq \epsilon n$ and for all $i\in[r]$
$$n' = |V_i| = \frac{(1-1/T \pm 2\epsilon)n }{r} \enspace \text{and} \enspace |U_i| = \frac{(1\pm 2\epsilon)n}{Tr}.$$
Let $V:= \bigcup_{i\in [r]} V_i$ and $U:= \bigcup_{i\in [r]} U_i.$
Suppose that $G, F, F'$ are edge-disjoint graphs such that 
$V(G)= V\cup U\cup U_0$, $F$ is a bipartite graph with vertex partition $(V,U)$, and $F'$ is a bipartite graph with vertex partition $(V_0,U)$ such that $F'=\bigcup_{t\in [T]} \bigcup_{v\in V_0} F'_{v,t}$, where all the $F'_{v,t}$ are pairwise edge-disjoint stars with centre $v$.

Suppose that  $\mathcal{H}$ is a collection of $(k,\eta)$-chromatic $\eta$-separable graphs on $n$ vertices, and for each $t\in [T]$ we have a multi-$(k-1)$-graph $\mathcal{C}_t $ on $[r]$ and a multi-$k$-graph $\mathcal{C}^*_t$ on $[r]$ with $E(\mathcal{C}_t) = \{ C_{v,t}  : v \in V_0 \}$ and
$E(\mathcal{C}^*_t) = \{ C^*_{v,t} : v\in V_0\}$.
Assume the following hold. 
\begin{enumerate}
\item[{\rm (A1)}$_{\ref{embedone}}$] For each $H\in \cH$, we have  $\Delta(H)\leq \Delta$ and $e(H)\geq n/4$, 

\item[{\rm (A2)}$_{\ref{embedone}}$] $n^{7/4} \le e(\mathcal{H}) \le (1-\nu)(k-1)\alpha n^2/(2r),$

\item[{\rm (A3)}$_{\ref{embedone}}$] $G[V]$ is $( T^{-1/2},\alpha)$-super-regular with respect to the vertex partition $(Q,V_1,\dots, V_r)$,
\item[{\rm (A4)}$_{\ref{embedone}}$] 
for each $ij\in E(R)$, the graphs $G[V_i, U_j]$ and $G[U_i, U_j]$ are both 
 $(\epsilon^{1/50},(d^3))^{+}$-regular,

\item[{\rm (A5)}$_{\ref{embedone}}$] $\delta(R) \geq (1 - 1/k+\sigma)r$,

\item[{\rm (A6)}$_{\ref{embedone}}$]  for all $ij \in E(Q)$ and $u\in U_i$, we have
$d_{F,V_j}(u) \geq d^3 n'$,

\item[{\rm (A7)}$_{\ref{embedone}}$] 
for all $v\in V_0$ and $t\in [T]$ and $i\in C_{v,t}$, we have
$d_{F'_{v,t},U_i}(v) \geq (1-d)\alpha |U_i|,$

\item[{\rm (A8)}$_{\ref{embedone}}$] for all $v\in V_0$ and $t\in [T]$, 
we have $C_{v,t} \subseteq C^*_{v,t}$,
$R[C^*_{v,t}] \simeq K_{k}$, and $\Delta(\cC^*_t) \leq \frac{ \epsilon^{3/4} n}{r}$,

\item[{\rm (A9)}$_{\ref{embedone}}$] for each $u\in U_0$, we have
$$|\{ i\in [r] : d_{G, V_j}(u) \geq d^3 n' \text{ for all } j\in N_{Q}(i) \}|> \epsilon^{1/4}r.$$

\end{enumerate}
Then there exists a packing $\phi$ of $\mathcal{H}$ into $G \cup F \cup F'$ such that 
\begin{enumerate}
\item[{\rm (B1)}$_{\ref{embedone}}$] $\Delta(\phi(\mathcal{H})) \le 4 k \Delta \alpha n/r,$ 
\item[{\rm (B2)}$_{\ref{embedone}}$] for each $u\in U$, we have
$d_{\phi(\mathcal{H})\cap G}(u) \le 2\Delta \epsilon^{1/8}n/r,$
\item[{\rm (B3)}$_{\ref{embedone}}$] for each $i \in [r]$, we have $e_{\phi(\mathcal{H})\cap G}(V_i, U\cup U_0) < \epsilon^{1/2} n^2/r^2.$
\end{enumerate}

\end{lemma}

Roughly, the proof of Lemma~\ref{embedone} will proceed as follows.
In Step 1 we define a partition of $U_0$ and an auxiliary digraph $D$.
In Step 2 we define a partition of each $H \in \cH.$
For each graph $H \in \cH$ we apply Lemma \ref{lem: separable partition} to  partition $V(H)$ into $X^H, Y^H, Z^H, A^H$.
We will embed $A^H$ into $V_0$ and the remainder of $H$ into $V \cup U \cup U_0.$
In Step 3, we apply Lemma~\ref{lem: embedding lemma} to find an appropriate function $\phi'$ packing $\{ H[Y^H\cup Z^H\cup A^H] : H\in \cH\}$ into $G[U]\cup F'$. 
Guided by the auxiliary digraph $D$, in Step 4 we modify the partition by removing a suitable $W^H$ from $X^H$ (so that we can later embed $X^H \backslash W^H$ into $V$).
We will also find a function $\phi''$ packing $\{H[W^H] : H\in \cH\}$ into $G[U]$ in an appropriate way, which ensures that later we can also pack $\{ H[X^H\setminus W^H, W^H] : H \in\cH\}$ into $F[V,U]\cup G[V,U]$.
In Step 5 we will partition $\cH$ into subcollections $\cH_{1,1},\dots, \cH_{T,w}$ and use Lemma~\ref{Pack} to pack $\{H[X^H\setminus W^H]: H\in \cH_{t,w'}\}$ into an internally $q$-regular graph $H_{t,w'}$ (for some suitable $q$). 
Finally, in Step 6 we apply the blow-up lemma for approximate decompositions (Theorem~\ref{Blowup}) to pack $\{H_{t,w'}: t\in [T], w'\in [w]\}$ into $G[V]$ such that the packing obtained is consistent with $\phi'\cup \phi''$.

\begin{proof}
Let $r':=r/k$ and $Q_1,\dots, Q_{r'}$ be the copies of $K_k$ in $Q$.
Let $n_0:= |V_0|$ and $V_0=:\{ v_1,\dots, v_{n_0}\}.$
By (A1)$_{\ref{embedone}}$, for each $H\in \cH$, we have
\begin{align}\label{eq: number of edges in H}
e(H) \leq \Delta n.
\end{align}
Moreover,
\begin{eqnarray}\label{eq: kappa def}
\kappa:= |\cH| \stackrel{\text{(A1)$_{\ref{embedone}}$},\text{(A2)$_{\ref{embedone}}$}}{\leq} 2(1-\nu)(k-1)\alpha n/r.
\end{eqnarray}
\vspace{0.05cm}

\noindent {\bf Step 1. Partition of $U_0$ and the construction of an auxiliary digraph $D$.}
In Step 2, we will find a partition of each $H \in \cH$ which closely reflects the structure of $G$.
However we need the partitions to match up exactly.
The following auxiliary graph will enable us to carry out this adjustment in Step~4.
Let $D$ be the directed graph with $V(D) =[r]$ and 
\begin{align}\label{eq: def of digraph D}
E(D) = \{ \vec{ij} : i\neq j \in [r],  N_{Q}(i) \subseteq N_{R}(j)  \}.
\end{align}
For each $ij\in E(R)$, we let 
$$U_i(j):= \{ u\in U_i: d_{G,V_j}(u) \geq (d^3-\epsilon^{1/50})n'\}.$$
Then (A4)$_{\ref{embedone}}$ with Proposition~\ref{prop: super} implies that $|U_i(j)|\geq (1-2\epsilon^{1/50})|U_i|.$
For each $\vec{ij} \in E(D)$, we define 
\begin{align}\label{eq: UDji def}
U_{j}^D(i):=\bigcap_{i'\in N_{Q}(i)} U_j(i'),
\end{align}
then we have 
\begin{align}\label{eq: UDji size}
|U_j^D(i)| \geq (1 - 2(k-1) \epsilon^{1/50})|U_j| \geq n/(2Tr).
\end{align}
In Step 4 we will map some vertices $x \in V(H)$ whose `natural' image would have been in $V_i$ to $U_j^D(i)$ instead, in order to `balance out' the vertex class sizes. 

\begin{claim}\label{cl: path from i* to j}
There exists a set $I^* = \{ i^*_1,\dots,i^*_{k} \} \subseteq  [r]$ of $k$ distinct numbers such that for any $k'\in [k]$ and $j\in [r]$, there exists a directed path $P(i^*_{k'},j)$ from $i^*_{k'}$ to $j$ in $D$.
\end{claim}
\begin{proof}
First, we claim that all $i\neq j\in [r]$ satisfy that $N^{-}_D(i)\cap N^{-}_D(j)\neq \emptyset$.
Indeed, as $|N_{R}( \{i,j\} )| \geq 2\delta(R)- r\geq (1- 2/k+2\sigma)r$, we have that 
$$|\{ s\in [r'] : |N_{R,V(Q_s)}( \{i,j\})| \geq k-1\}| \geq \sigma r \geq 3.$$
\COMMENT{
$(k-2)r/k + 2|\{ s\in [r'] : |N_{R,Q_s}( \{i,j\})| \geq k-1\}| = (k-2)|\{ s\in [r'] : |N_{R,Q_s}(\{i,j\})| \leq k-2\}| + k|\{ s\in [r'] : |N_{R,Q_s}(\{i,j\})| \geq k-1\}|  \geq d_{R}(i,j) \geq (1-2/k+2\sigma)r$.
Thus $|\{ s\in [r'] : |N_{R,Q_s}(\{i,j\})| \geq k-1\}| \geq \sigma r$.}
Thus there exists $s\in [r']$ such that 
$i,j\notin V(Q_s)$ while $|N_{R,V(Q_s)}(\{i,j\})| \geq k-1$. We choose $j' \in V(Q_s)$ such that $Q_s\setminus \{j'\} \subseteq N_{R}(\{i,j\})$, then \eqref{eq: def of digraph D} implies that $i,j \in N^+_{D}(j')$.

Now, we consider a number $i\in [r]$ which maximizes $|A(i)|$, where
$$A(i) = \{j\in [r]: \text{ there exists a directed path from }i \text{ to } j \text{ in }D\}.$$
If there exists $j\in [r]$ such that $j\notin A(i)$, then by the above claim, there exists $j'\in [r]$ such that $i, j\in N^+_D(j')$.
Then%
\COMMENT{any directed path from $i$ to $j'' \in [r]$ together with $\vec{j'i}$ forms a directed path from $j'$ to $j''$ in $D$.} $A(i)\cup \{j\} \subseteq A(j')$, which is a contradiction to the maximality of $A(i)$. Thus, 
we have $A(i) = [r]$. Let $i^*_1:=i$.

Since $d_{R}(i^*_1) \geq \delta(R) \geq  (1-1/k+\sigma)r$ by (A5)$_{\ref{embedone}}$, we have
$|\{ s\in [r'] : N_{R,V(Q_s)}(i^*_1) =k \}|\geq \sigma r.$
Thus, there exists $s \in [r']$ such that $V(Q_s) \subseteq N_R(i^*_1)$, and this with \eqref{eq: def of digraph D} implies that $V(Q_s) \subseteq N_D^{-}(i^*_1)$.
We let $i^*_2,\dots, i^*_{k}$ be $k-1$ arbitrary numbers in $V(Q_s)$.
 Then for all $k'\in [k]$ and $j\in [r]$, there exists a directed path from $i^*_{k'}$ to $i^*_1$ and a directed path from $i^*_1$ to $j$ in $D$.
 Thus there exists a directed path from $i^*_{k'}$ to $j$ in $D$.
 This proves the claim.
\end{proof}

We will now determine the approximate class sizes $\tilde{n}_i$ that our partition of $H$ will have.
For this, we first partition $U_0$ into $U'_1,\dots, U'_{r}$ in such a way that the vertices in $U_i'$ are `well connected' to the blow-up of the $k$-clique in $Q$ to which $i$ belongs.
\begin{equation}\label{eq: U' properties}
\begin{minipage}[c]{0.9\textwidth}\em
For all $i\in [r], u\in U'_i$ and $j\in N_{Q}(i)$, we have $d_{G,V_j}(u) \geq d^3 n'$ and $|U'_i| \leq 2\epsilon^{3/4}n/r.$
\end{minipage}
\end{equation}
Indeed, it is easy to greedily construct such a partition by using the fact that $|U_0|\leq \epsilon n$ and (A9)$_{\ref{embedone}}$.%
\COMMENT{
We enumerate vertices in $U_0$ into $u_1,\dots, u_{|U_0|}$.
For some $i\in [|U_0|-1]\cup \{0\}$, assume that we have a partition $U'^i_1,\dots, U'^i_r$ of $\{ u_1,\dots, u_i\}$ such that
\begin{itemize}
\item for all $j\in [r]$, we have $|U'^{i}_j| \leq  \lceil \epsilon^{3/4}n/r\rceil$ and 
\item for all $u\in U'^{i}_j$ and $j'\in N_{Q}(j)$, we have $d_{G,V_{j'} }(u) \geq d^3 n'$.
\end{itemize}
Now we consider $u_{i+1}$, and let 
$$A:= \{ j\in [r]: d_{G,V_{j'}}(u_{i+1}) \geq d^3 n \text{ for all } j'\in N_{Q}(j)\}.$$
Then
$$|\bigcup_{j\in A} U'^{i}_j | \leq  |U_0| \leq \epsilon n.$$
Since  (A9)$_{\ref{embedone}}$ implies that  $|A| > \epsilon^{1/4} r$,
there exists $j_* \in A$ such that 
$|U'^{i}_{j_*}|<\lceil \epsilon^{3/4} n/r \rceil$.
We let $U'^{i+1}_{j_*} := U'^{i}_{j_*} \cup \{u_{i+1}\}$ and 
for all $i \in [r]\setminus \{j_*\}$, we let $U'^{i+1}_{j} := U'^{i}_{j}$.
Then we have sets $U'^{i+1}_j$ such that
\begin{itemize}
\item for all $j\in [r]$, we have $|U'^{i+1}_j| \leq  \epsilon^{3/4}n/r$ and 
\item for all $u\in U'^{i+1}_j$ and $j'\in N_{Q}(j)$, we have $d_{G,V_{j'} }(u_i) \geq d^3 n'$.
\end{itemize}
By repeating this, we obtain the desired sets $U'_j := U'^{|U_0|}_j$ for each $j\in [r]$.
}

For $i \in I^*$, we will slightly increase the partition class sizes (cf.~\eqref{eq: tilde ni def}
and \ref{main lem X5}) as this will allow us to subsequently move any excess vertices from classes corresponding to $I^*$ to another arbitrary class via the paths provided by Claim \ref{cl: path from i* to j}.
For each $i\in [r]$, we let 
\begin{align}\label{eq: ni def}
n_i:= n' + |U_i| + |U'_i| = |V_i| + |U_i|+ |U'_i|,
\end{align}
 then we have
\begin{align}\label{eq: size n_i}
n_i = (1-1/T\pm 2\epsilon)n/r + (1\pm 2\epsilon)n/(Tr) \pm 2\epsilon^{3/4}n/r = (1\pm \epsilon^{2/3}/2)n/r \text{ and }\sum_{i\in [r]} n_i = n- n_0.
\end{align}
For each $i\in [r]$ we let
\begin{align}\label{eq: tilde ni def}
 \tilde{n}_i:= \left\{ \begin{array}{ll}
n_i + (r'-1)\eta^{1/5}n & \text{ if } i\in I^*,\\
n_i - \eta^{1/5}n & \text{ if } i\in [r]\setminus I^*.\\
\end{array}\right.
\end{align}
This with \eqref{eq: size n_i} implies that for each $i\in [r]$,
\begin{align}\label{eq: size tilde ni}
\tilde{n}_i = \frac{(1\pm \epsilon^{2/3})n}{r} \enspace \text{and} \enspace \sum_{i\in [r]} \tilde{n}_i = \sum_{i\in [r]} n_i = n- n_0. 
\end{align}
\vspace{0.06cm}

\noindent {\bf Step 2. Preparation of the graphs in $\cH$.}
First, we will partition $\cH$ into $T$ collections $\cH_{1},\dots, \cH_{T}$. 
Later we will pack each $\cH_{t}$ into $G\cup F \cup \bigcup_{v\in V_0} F'_{v,t}$.
(Recall that the $F'_{v,t}$ form a decomposition of $F'$.)
As $G\cup F\cup F'$ has vertex partition $V_0,\dots, V_r, U_1,\dots, U_r, U'_1,\dots, U'_r$, for each $H\in \cH$, we also need a suitable partition of $V(H)$ which is compatible with the partition of the host graph $G\cup F\cup F'$.
To achieve this, we will apply Lemma~\ref{lem: separable partition} to each graph $H\in \cH_{t}$ with the hypergraphs $\mathcal{C}_{t}$ and $\mathcal{C}^*_t$ to find the desired partition of $V(H)$.

   By \eqref{eq: number of edges in H} we can partition $\cH$ into $\cH_{1},\dots, \cH_{T}$ such that for each $t\in [T]$, 
\begin{eqnarray}\label{eq: edges in cHi}
e(\cH_{t}) \hspace{-0.2cm} &=& \hspace{-0.2cm}  e(\cH)/T \pm \Delta n \stackrel{\text{(A2)$_{\ref{embedone}}$} }{\leq}  (1-2\nu/3)\alpha(k-1) n^2/(2Tr), \text{ and }\nonumber\\
|\cH_{t}|  \hspace{-0.2cm}  &\stackrel{\text{(A1)$_{\ref{embedone}}$}}{\leq}&  \hspace{-0.2cm}  4e(\cH_t)/n \leq 2\alpha(k-1) n/(Tr). 
\end{eqnarray}
For each $t \in [T]$, we wish to apply the randomised algorithm given by Lemma~\ref{lem: separable partition} with the following objects and parameters independently for all $H \in \cH_t$.\newline

{\small
\noindent
{ 
\begin{tabular}{c|c|c|c|c|c|c|c|c|c|c|c|c|c|c|c}
object/parameter &$H$ & $R$ & $Q$  & $\cC_t$ & $\cC^*_t$ & $n_0$ & $C_{v_\ell,t}$ & $ C^*_{v_\ell,t}$ & $\lceil 3/d \rceil$ & $\eta$ & $\epsilon$ & $k$ & $\Delta$ & $r$ & $\tilde{n}_i$ \\ \hline
playing the role of & $H$ & $R$ & $Q$  & $\cF$ & $\cF^*$ & $n'$ & $C_\ell$ & $C^*_\ell$ & $h$ & $\eta$ & $\epsilon$ & $k$ & $\Delta$ & $r$ & $n_i$
\end{tabular}
}\newline \vspace{0.2cm}
}

Indeed, 
(A5)$_{\ref{embedone}}$, (A8)$_{\ref{embedone}}$ imply that (A1)$_{\ref{lem: separable partition}}$, (A2)$_{\ref{lem: separable partition}}$ and (A3)$_{\ref{lem: separable partition}}$ hold with the above objects and parameters, respectively. Moreover, \eqref{eq: size tilde ni} implies that (A4)$_{\ref{lem: separable partition}}$ holds too.
Thus we obtain a partition $X^H_1,\dots, X^H_r,$ $Y^H_1,\dots, Y^H_r,$ $Z^H_r, \dots, Z^H_r, A^H$ of $V(H)$ such that  $A^H= \{ a^H_1, \dots, a^H_{n_0}\}$ is  a $3$-independent set of $H$ and the following hold, where
$X^H:=\bigcup_{i\in [r]} X^H_i,\enspace Y^H:= \bigcup_{i\in [r]} Y^H_i,$ and $Z^H:=\bigcup_{i\in [r]} Z^H_i.$

\begin{enumerate}[label=\text{\rm (X{\arabic*})$_{\ref{embedone}}$}]
\item\label{main lem X1} For each $\ell \in [n_0]$, we have $d_H(a^H_\ell) \leq \frac{(2+d)e(H)}{n}$,
\item\label{main lem X2} for each $\ell \in [n_0]$, we have $N_H(a^H_\ell) \subseteq \bigcup_{i\in C_{v_\ell,t}} Y_i^H \setminus N^1_H(Z^H)$,
\item\label{main lem X3}  $H[X^H]$ admits the vertex partition $(Q, X^H_1,\dots, X^H_r)$, and
$H\setminus E(H[X^H])$ admits the vertex partition $(R,X^H_1\cup Y^H_1\cup Z^H_1,\dots, X^H_r\cup Y^H_r\cup Z^H_r)$,
\item\label{main lem X4}  for each $ij \in E(Q)$, we have
$e_H(X^H_i, X^H_j) =  \frac{2 e(H) \pm \epsilon^{1/5} n}{(k-1)r}$,
\item\label{main lem X5} for each $i\in [r]$, we have 
$|Y^H_i|\leq 2\epsilon^{1/3}n/r$ and
$ |X^H_i|+ |Y^H_i| +|Z^H_i| = \tilde{n}_i \pm \eta^{1/4}n$;
in particular, this with \eqref{eq: tilde ni def} implies that
for each $i\in [r]$, we have
$$\widehat{n}^H_i:= |X^H_i|+|Y^H_i|+|Z^H_i| \in \left\{\begin{array}{ll}
\left[ n_i, n_i+ \eta^{1/6}n\right] & \text{ if } i\in I^*, \\
\left[n_i - \eta^{1/6} n, n_i \right] & \text{ otherwise,}
\end{array}\right.$$
\item\label{main lem X6}  $N^{1}_H(X^H) \setminus X^H \subseteq Z^H$, and $|Z^H| \leq 4 \Delta^{3k^3} \eta^{0.9} n$,
\item\label{main lem X7} for all $\ell \in [n_0]$ and $i \in C_{v_\ell,t}$, we have
$\mathbb{E}[ N_{H}(a^H_\ell ) \cap Y^H_{i} ] \leq \frac{(2+d)e(H)}{(k-1)n}.$
\end{enumerate}

By applying this randomised algorithm independently for each $H\in \cH_1\cup\dots\cup \cH_{T}$, we obtain that for all $t \in [T]$, $\ell \in [n_0]$ and $i \in C_{v_\ell,t}$, we have
$\mathbb{E}[ \sum_{H\in \cH_t} |N_{H}(a^{H}_\ell ) \cap Y^H_{i}| ]
\leq \frac{(2+d)e(\cH_{t})}{(k-1) n}.$
Note that for each $H\in \cH_t$, we have
$|N_{H}(a^{H}_\ell) \cap Y^H_{i}| \leq \Delta$. 
As our applications of the randomised algorithm are independent for all $H\in \cH_{t}$, a Chernoff bound (Lemma~\ref{lem: chernoff}) together with (A2)$_{\ref{embedone}}$ implies that for all $t \in [T]$, $\ell \in [n_0]$ and $i \in C_{v_{\ell},t}$, we have
$$\mathbb{P}\Big[ \sum_{H\in \cH_t} |N_{H}(a^{H}_{\ell}) \cap Y^H_{i}| 
\geq \frac{2(1+d)e(\cH_{t})}{(k-1) n}\Big] \leq 2\exp( -\frac{d^2 e(\cH_t)^2 / ((k-1)^2 n^2)}{2\Delta^2 |\cH_{t}|} ) \stackrel{\eqref{eq: edges in cHi}, \rm (A2)_{\ref{embedone}} }{\leq}  e^{-n^{1/3}}.$$
\COMMENT{
$$2\exp( -\frac{d^2 e(\cH_t)^2 / ((k-1)^2 n^2)}{2\Delta^2 |\cH_{t}|} )
\leq 2\exp( - \frac{ d^2 n^{7/2}/ ((k-1)^2 n^2) }{ 2\Delta^2 n })
\leq  e^{- n^{1/3}}.$$
}
By taking a union bound over all $t\in [T], \ell \in [n_0]$ and $i \in C_{v_\ell,t}$,
we can show that the following property \ref{main lem X8} holds with probability at least $1 - k T n_0 e^{-n^{1/3}} > 0$. 
\begin{enumerate}[label=\text{\rm (X{\arabic*})$_{\ref{embedone}}$}]
 \setcounter{enumi}{7}
\item\label{main lem X8} For all $t\in [T]$, $\ell\in [n_0]$ and $i \in C_{v_\ell,t}$, we have
$\sum_{H\in \cH_{t} } |N_{H}(a^H_\ell) \cap Y^H_{i}|  \leq \frac{2(1+d)e(\cH_t)}{(k-1)n}.$
\end{enumerate}
Thus we conclude that for all $H\in \cH$ there exist partitions $X^H_1,\dots, X^H_r,$ $Y^H_1,\dots, Y^H_r,$ $Z^H_r, \dots, Z^H_r$, $A^H$ of $V(H)$ such that $A^H= \{ a^H_1, \dots, a^H_{n_0}\}$ is a $3$-independent set of $H$ and such that \ref{main lem X1}--\ref{main lem X6} and \ref{main lem X8} hold. 

Note that $\sum_{i\in [r]} \widehat{n}^H_i = |V(H)|-|A^H| = n-n_0$. This with \eqref{eq: size n_i} implies that for each $H\in \cH$, we have
\begin{align}\label{eq: hat n sum n sum}
\sum_{i\in I^*} (\widehat{n}^H_{i}- n_i) = \sum_{i\in [r]\setminus I^*} (n_i - \widehat{n}^H_{i}).
\end{align}

The following claim determines the number of vertices that we will redistribute via $D$. 
\begin{claim}\label{cl: D f exists}
For each $H\in \cH$, there exists a function $f^H:E(D) \rightarrow [\eta^{1/7}n]\cup \{0\}$ such that 
for each $i\in [r]$, we have
$$\sum_{j\in N^+_{D}(i)} f^H(\vec{ij}) - \sum_{j\in N^-_{D}(i)} f^H(\vec{ji}) = \widehat{n}^H_{i} - n_i.$$
\end{claim}
\begin{proof}
By \ref{main lem X5}, for each $i\in I^*$, we have $\widehat{n}^H_{i}- n_i\geq 0$ 
and for each $i\in [r]\setminus I^*$, we have $n_i-\widehat{n}^H_i \geq 0.$
Thus by \eqref{eq: hat n sum n sum}, there exists a bijection $g^H$ from 
$$\bigcup_{i\in I^*} \{ i\} \times [\widehat{n}^H_{i}- n_i] \enspace\text{to}\enspace 
 \bigcup_{i\in  [r]\setminus I^* } \{i\} \times [n_i-\widehat{n}^H_i].$$ 
For all $i\in I^*$ and $m\in   [\widehat{n}^H_{i}- n_i]$, 
let $g^H(i,m) =: (g^H_1(i,m), g^H_2(i,m))$ and
 let $P_{i,m}$ be a directed path from $i$ to $g^H_1(i,m)$ in $D$, which exists by Claim~\ref{cl: path from i* to j}.
 As $g^H$ is a bijection, for each $i\in [r]$, we have
 \begin{align}\label{eq: g-1 size}
 |(g^H_1)^{-1}(i)| = \left\{ \begin{array}{ll}
 0 & \text{ if } i\in I^*,\\
n_i - \widehat{n}_i^H & \text{ otherwise.}
 \end{array}\right.
 \end{align}
For each $\vec{ij} \in E(D)$, we let
$$f^H(\vec{ij}) := |\{ (i',m) : i'\in I^*,m\in [\widehat{n}^H_{i'}- n_{i'}] \enspace\text{and} \enspace \vec{ij} \in E( P_{i',m}) \}|.$$
Then for each $\vec{ij}\in E(D)$, we have
$$f^H(\vec{ij}) \leq \Big|\bigcup_{i'\in I^*} \{ i'\} \times [\widehat{n}^H_{i'}- n_{i'}]\Big| \stackrel{\ref{main lem X5}}{\leq } k \eta^{1/6} n \leq \eta^{1/7}n.$$
Note that for any $i\in I^*$ and $m\in [\widehat{n}^H_{i}- n_{i}]$,  the path $P_{i,m}$ starts from a vertex in $I^*$ and ends at $[r]\setminus I^*$.
Thus for each $i\in [r]$ we have
\begin{align*}
& \sum_{j\in N_D^{+}(i)} f^H(\vec{ij}) - \sum_{j\in N_{D}^{-}(i)} f^H(\vec{ji}) \\
&= |\{(i',m) :  m\in [\widehat{n}^H_{i'}- n_{i'}], i = i' \in I^* \}|   - |\{(i',m) : i'\in I^*, m\in [\widehat{n}^H_{i'}- n_{i'}], g^H_1(i',m) = i \}| \\
&= \left\{ \begin{array}{ll}
(\widehat{n}^H_{i}- n_i) - 0= \widehat{n}^H_i - n_i & \text{ if } i\in I^*, \\
 0 - (g^H_1)^{-1}(i) \stackrel{\eqref{eq: g-1 size}}{=} \widehat{n}^H_i-n_i & \text{ otherwise}.
 \end{array}\right. 
\end{align*}
This proves the claim.
\end{proof}
For each $H\in \cH$, we fix a function $f^H$ satisfying Claim~\ref{cl: D f exists}. For each $\vec{ij}\notin E(D)$, it will be convenient to set $f^H(\vec{ij}):=0$. 

We aim to embed vertices in $X^H_i\cup Y^H_i \cup Z_i^H$ into $V_i\cup U_i\cup U'_i$. As $|V_i\cup U_i\cup U'_i| =n_i$, by \eqref{eq: ni def}, it would be ideal if $|X^H_i\cup Y^H_i \cup Z_i^H| =n_i$ and $|X^H_i|=n'$.
 However, \ref{main lem X5} only guarantees that this is approximately true. 
In order to deal with this, we will use $D$ and $f^H$ to assign a small number of `excess' vertices $u\in X^H_i$ into $U_j$ when $\vec{ij}\in E(D)$. 
The definition of $D$ will ensure that the image of $u$ still has many neighbours in $V_{i'}$ for all $i'\in N_{Q}(i)$. \newline

\noindent {\bf Step 3. Packing the graphs $H[Y^H\cup Z^H\cup A^H]$ into $G[U]\cup F'$.} 
Now, we aim to find a suitable function $\phi'$ which packs $\{H[Y^H\cup Z^H\cup A^H] : H\in \cH\}$ into $G[U]\cup F'$. 
In order to find $\phi'$, we will use Lemma~\ref{lem: embedding lemma}. 
Moreover, we choose $\phi'$ in such a way that we can later extend $\phi'$ into a packing of the entire graphs $H\in \cH$. 
One important property we need to ensure is the following: 
for any vertex $x\in X^{H}_j$ which is not embedded by $\phi'$, and any vertices $y_1,\dots, y_i \in N_H(x)\cap (Y^H\cup Z^H)$ which are already embedded by $\phi'$, we need $N_G( \phi'( \{y_1,\dots, y_i\}))\cap V_j$ to be large, so that $x$ can be later embedded into $N_{G}( \phi'(\{y_1,\dots, y_{i}\}))\cap V_j$. For this, we will introduce a hypergraph $\mathcal{N}_{H}$ which encodes information about the set $N_H(x)\cap (Y^H\cup Z^H )$ for each vertex $x\in X^{H}$. 
In order to describe the structure of $G$ and $H$ more succinctly, we also introduce a graph $R'$ on $[2r]$ such that 
$$E(R') = \left\{ ij : (i-r)(j-r) \in E(R) \text{ or } i(j-r) \in E(R) \right\}.$$ 
For all $i\in [r]$ and $H\in \cH$, let $V_{i+r} :=  U_i \text{ and } X^{H}_{i+r} := Y^{H}_{i} \cup Z^{H}_{i}.$
Note that  \ref{main lem X3} and \rm (A4)$_{\ref{embedone}}$ imply that for each $H\in \cH$,
\begin{equation}\label{eq: admits partition}
\begin{minipage}[c]{0.9\textwidth} \em
$H[Y^H\cup Z^H]$ admits  the vertex partition $(R',\emptyset,\dots,\emptyset,X^H_{r+1},\dots, X^H_{2r})$, and \\
$G$ is $(\epsilon^{1/50},(d^3))^+$-regular with respect to the partition $(R', V_1,\dots, V_{2r})$.
\end{minipage}
\end{equation}
For all $H\in \cH$ and $x\in X^H$, let 
$$e_{H,x}:= N_H(x) \setminus X^H \stackrel{\ref{main lem X6} }{=} N_H(x) \cap Z^H.$$
Let $\mathcal{N}_H$ be a multi-hypergraph on vertex set $Z^H$ with
 \begin{eqnarray}\label{eq: NH def}
E(\mathcal{N}_H):= \{ e_{H,x} :  x\in N^1_{H}( Z^H)\cap X^H \}, 
\end{eqnarray}
and let $f_{H}:  E(\mathcal{N}_H) \rightarrow [r]$ be a function such that for all $x\in X^H$, we have that $x \in X^H_{f_H(e_{H,x})}$.
Then $\Delta(\mathcal{N}_H)\leq \Delta$ and
$\mathcal{N}_H$ has edge-multiplicity at most $\Delta$.
Note that, as $\mathcal{N}_H$ is a multi-hypergraph, there could be two distinct vertices $x\neq x' \in X^H$ such that $e_{H,x}$ and $e_{H,x'}$ consists of exactly the same vertices while $f_H(e_{H,x})\neq f_H(e_{H,x'})$.

Our next aim is to construct a function $\phi'$ which packs $\{ H[Y^H\cup Z^H\cup A^H]: H\in \cH\}$ into $G[U] \cup F'$ in such a way that
the following hold for all $H\in \cH$.
\begin{enumerate}[label=\text{\rm($\Phi'${\arabic*})$_{\ref{embedone}}$}]
\item \label{main lem phi' 1} For each $e \in E(\mathcal{N}_{H})$, we have $|N_G(\phi'(e)) \cap V_{f_{H}(e)}| \geq d^{5\Delta}|V_{f_{H}(e)}|$,  
\item \label{main lem phi' 2} for each $v\in V(G)$,
we have $|\{ H\in \cH : v\in \phi'(H[Y^H\cup Z^H]) \}| \leq \epsilon^{1/8}n/r$, 
\item \label{main lem phi' 3} for all $i\in [r]$ and $H\in \cH$, 
we have $\phi'(Y^{H}_i\cup Z^{H}_i) \subseteq U_{i}$, and
\item \label{main lem phi' 4} $\phi'(A^H) = V_0$.
\end{enumerate}

\begin{claim}\label{eq: phi' exists}
There exists a function $\phi'$ packing $\{H[Y^H\cup Z^H\cup A^H]: H\in\cH\}$ into $G[U]\cup F'$ which satisfies \ref{main lem phi' 1}--\ref{main lem phi' 4}.
\end{claim}
\begin{proof}
Let $\phi'_0:\emptyset \rightarrow \emptyset$ be an empty packing. 
Let $H_1,\dots, H_{\kappa}$ be an enumeration of $\cH$. 
For each $s\in [\kappa]$, let 
$$\cH^{s} := \{ H_{s'}[Y^{H_{s'}} \cup Z^{H_{s'}}\cup A^{H_{s'}}]: s'\in [s]\}.$$
Our aim is to successively extend $\phi'_0$ into $\phi'_1,\dots, \phi'_{\kappa}$ in such a way that each $\phi'_{s}$ satisfies the following. 

\begin{enumerate}[label=\text{\rm($\Phi'${\arabic*})$^{s}_{\ref{embedone}}$}]
\item \label{main lem Phi 1} $\phi'_{s}$ packs $\cH^{s}$ into $G[U] \cup F'$, 
\item \label{main lem Phi 2} for all $s'\in [s]$ and $e \in E(\mathcal{N}_{H_{s'}})$, we have $|N_G(\phi'_{s}(e)) \cap V_{f_{H_{s'}}(e)}| \geq d^{5\Delta}|V_{f_{H_{s'}}(e)}|$,  
\item \label{main lem Phi 3} for each $v\in V(G)$, we have
$|\{ s'\in [s]: v\in \phi'_{s}(H_{s'}[Y^{H_{s'}}\cup Z^{H_{s'}}]) \}| \leq \epsilon^{1/8}n/r$, 
\item \label{main lem Phi 4} for all $i\in [2r]\setminus[r]$ and $s'\in [s]$, we have
$\phi'_{s}(X^{H_{s'}}_i) \subseteq V_{i}$,
\item \label{main lem Phi 5} for all $s'\in [s]$ and $\ell \in [n_0]$, we have $\phi'_{s}(a^{H_{s'}}_\ell) = v_\ell$,
\item \label{main lem Phi 6} for all $s'\in [s]$, $t\in [T]$ with $H_{s'}\in \cH_t$, we have $\phi'_{s}(H_{s'}[Y^{H_{s'}}\cup Z^{H_{s'}} \cup A^{H_{s'}} ]) \subseteq G[U] \cup \bigcup_{v\in V_0} F'_{v,t}.$
\end{enumerate}

Note that $\phi'_0$ vacuously satisfies ($\Phi'$1)$^{0}_{\ref{embedone}}$--($\Phi'$6)$^{0}_{\ref{embedone}}$. 
Assume we have already constructed $\phi'_{s}$ satisfying \ref{main lem Phi 1}--\ref{main lem Phi 6} for some $s\in [\kappa-1]\cup \{0\}$. 
We will show that we can construct $\phi'_{s+1}$.
Let 
$$G(s):=G\setminus \phi'_{s}(\cH^{s}).$$
For all $\ell \in [n_0]$ and $a^{H_{s+1}}_\ell \in A^{H_{s+1}}$, we first let
\begin{align}\label{eq: def phi' pj}
\psi(a^{H_{s+1}}_\ell):= v_\ell.
\end{align}
For each $i\in [2r]\setminus [r]$, let 
$$V^{\rm bad}_{i} := \Big\{ v\in V_i : |\{s'\in [s] : v\in \phi'_{s'}(H_{s'}[Y^{H_{s'}}\cup Z^{H_{s'}}])\}| \geq \frac{\epsilon^{1/8}n}{r} -1 \Big\}.$$
Note that 
\begin{eqnarray}\label{eq: size bad Wi}
|V^{\rm bad}_{i}| \stackrel{\ref{main lem Phi 4}}{\leq} \frac{\sum_{s'\in [s] } |Y^{H_{s'}}_{i-r}\cup Z^{H_{s'}}_{i-r} |}{\frac{\epsilon^{1/8}n}{r} -1} \stackrel{\ref{main lem X5},\ref{main lem X6}}{\leq}   3\epsilon^{1/3-1/8} \kappa \stackrel{\eqref{eq: kappa def}}{\leq} \frac{\epsilon^{1/5} n}{r}. 
\end{eqnarray}
Let $t \in [T]$ be such that $H_{s+1} \in \cH_{t}$.
For all $i\in [2r]\setminus [r]$ and $x\in X^{H_{s+1}}_{i}$, we let 
$$B_x:= \left\{\begin{array}{ll}
N_{F'_{v_\ell,t}, V_i}(v_\ell) \setminus (N_{\phi'_{s}(\cH^{s})}(v_\ell) \cup V^{\rm bad}_i) & \text{ if } x\in N_{H_{s+1}}(a^{H_{s+1}}_\ell)\cap X^{H_{s+1}}_{i} \text{ for some $\ell \in [n_0]$}, \\
V_{i} \setminus V^{\rm bad}_{i} & \text{ otherwise.}
\end{array}\right.$$
We will later embed $x$ into $B_x$. Note that if $x\in N_{H_{s+1}}(a^{H_{s+1}}_{\ell})$, then 
$x \notin N_{H_{s+1}}(a^{H_{s+1}}_{\ell'})$ for any $\ell' \in [n_0]\setminus \{\ell\}$ as $A^{H_{s+1}}$ is a $3$-independent set in $H_{s+1}$.
Also, if $x \in N_{H_{s+1}}(a^{H_{s+1}}_\ell)\cap X^{H_{s+1}}_{i}$, then
by \ref{main lem X2} we have $i-r \in C_{v_\ell,t}$. Thus in this case
\begin{eqnarray*}
|B_x| \hspace{-0.2cm} &\geq&  \hspace{-0.2cm}  d_{F'_{v_\ell,t},V_i }(v_\ell) - d_{\phi'_{s}(\cH^{s})\cap F'_{v_\ell,t}, V_i}(v_\ell) - |V^{\rm bad}_{i}|  \\
&\stackrel{\text{\rm (A7)$_{\ref{embedone}}$},\eqref{eq: size bad Wi}}{\geq}& (1-d)\alpha|U_{i-r}|   - d_{\phi'_{s}(\cH^{s})\cap F'_{v_\ell,t}, V_i}(v_\ell)  - \epsilon^{1/5}n/r \nonumber \\
 \hspace{-0.2cm}&\stackrel{ \substack{\ref{main lem X2},  \ref{main lem Phi 4}, \\ \ref{main lem Phi 5}, \ref{main lem Phi 6}}}{\geq}&  \hspace{-0.2cm} (1-d)\alpha|U_{i-r}| - \sum_{s'\in [s], H_{s'}\in \cH_t}|N_{H_{s'}}(a_\ell^{H_{s'}})\cap Y_{i-r}^{H_{s'}}| - \epsilon^{1/5}n/r \nonumber \\
 \hspace{-0.2cm} &\stackrel{\ref{main lem X8}}{\geq }&  \hspace{-0.2cm} (1-d)\alpha|U_{i-r}| - \frac{2(1+d)e(\cH_t)}{(k-1)n} - \epsilon^{1/5} n/r \nonumber \\
 \hspace{-0.2cm}&\stackrel{\eqref{eq: edges in cHi}}{\geq} & \hspace{-0.2cm} (1-d)\alpha |U_{i-r}| - \frac{(1+d)(1-2\nu/3)\alpha n}{Tr} -  \epsilon^{1/5} n/r \geq \alpha^2 |U_{i-r}| = \alpha^2 |V_{i}|.
\end{eqnarray*}
If $x \notin N_{H_{s+1}}(a^{H_{s+1}}_\ell )$ for any $\ell \in [n_0]$, then $|B_x| \geq |V_{i}|- |V^{\rm bad}_{i}| \geq (1-\epsilon^{1/10})|V_{i}|.$ So, for all $i\in [2r]\setminus[r]$ and $x\in X^{H_{s+1}}_i$, we have
\begin{eqnarray}\label{eq: size Ax}
B_x\subseteq V_i, \text{ and } |B_x| \geq  \alpha^2 |V_{i}|.
\end{eqnarray}

For each $i\in [r]$, let $P_i:=\emptyset$, and for each $i\in [2r]\setminus [r]$, let $P_i:= X_i^{H_{s+1}}$. 
We wish to apply Lemma~\ref{lem: embedding lemma} with $H[ Y^{H_{s+1}}\cup Z^{H_{s+1}}]$ playing the role of $H$ and with the following objects and parameters.
\newline

{\small
\noindent
{ 
\begin{tabular}{c|c|c|c|c|c|c|c|c|c|c|c|c|c}
object/parameter & $G(s)$ & $R'$    & $V_i$ & $P_i$ & $\epsilon^{1/60}$ &  $\Delta$ & $n'$ &  $\alpha^2$ & $d^3$ & $\mathcal{N}_{H_{s+1}}$ & $f_{H_{s+1}}$ &  $1/(2T)$ & $B_x$ 
\\ \hline
playing the role of & $G$ &  $R$  & $V_i$ & $X_i$ & $\epsilon$& $\Delta$ & $n$ &  $\alpha$ & $d$ & $\mathcal{M}$ & $f$ & $\beta$ & $A_x$   
\end{tabular}
}}\newline \vspace{0.2cm}

Let us first check that we can indeed apply Lemma~\ref{lem: embedding lemma}.
Note that for each $ij\in E(R')$ with $i \in [2r]\setminus [r]$,
\begin{eqnarray*}
e_{G(s)}(V_i, V_j)  & \ge & e_{G}(V_i,V_j) - \Delta\sum_{v\in V_i} |\{ s'\in [s] : v\in \phi'_{s}(H_{s'}[Y^{H_{s'}}\cup Z^{H_{s'}}])\}| \nonumber \\ 
&\stackrel{\ref{main lem Phi 3}}{\ge}& e_{G}(V_i,V_j) - \Delta \epsilon^{1/8} n|V_i| /r
\stackrel{\text{(A4)$_{\ref{embedone}}$} }{\geq} (1- \epsilon^{1/9})e_{G}(V_i,V_j).
\end{eqnarray*}
Thus \eqref{eq: admits partition} with Proposition~\ref{prop: reg subgraph} implies that  (A1)$_{\ref{lem: embedding lemma}}$ of Lemma~\ref{lem: embedding lemma} holds.
Again \eqref{eq: admits partition} implies that (A2)$_{\ref{lem: embedding lemma}}$ holds.
Conditions (A3)$_{\ref{lem: embedding lemma}}$ and  (A4)$_{\ref{lem: embedding lemma}}$ are obvious from (A1)$_{\ref{embedone}}$, \ref{main lem X3} and the definition of $\mathcal{N}_{H_{s+1}}$. Moreover, \eqref{eq: size Ax} implies that (A5)$_{\ref{lem: embedding lemma}}$ also holds.
 Thus by Lemma~\ref{lem: embedding lemma}, we obtain an embedding $\psi': H_{s+1}[ Y^{H_{s+1}}\cup Z^{H_{s+1}}] \rightarrow G(s)[U]$ satisfying the following. 
\begin{enumerate}
\item[(P1)$^{s+1}_{\ref{embedone}}$] For each $x\in  Y^{H_{s+1}}\cup Z^{H_{s+1}}$, we have $\psi'(x)\in B_x$,
\item[(P2)$^{s+1}_{\ref{embedone}}$] for each $e\in E(\mathcal{N}_{H_{s+1}})$, we have
$|N_{G}( \psi'(e)) \cap V_{f_{H_{s+1}}(e)} | \geq  (d^3/2)^{\Delta} |V_{f_{H_{s+1}}(e)}|$.
\end{enumerate}
Let $\phi'_{s+1} := \phi_{s} \cup \psi \cup \psi'.$  By \eqref{eq: def phi' pj} with the definitions of $G(s)$ and $B_x$,
this implies ($\Phi'$1)$^{s+1}_{\ref{embedone}}$ and ($\Phi'$6)$^{s+1}_{\ref{embedone}}$. As $d\ll 1$, (P2)$^{s+1}_{\ref{embedone}}$ implies ($\Phi'$2)$^{s+1}_{\ref{embedone}}$, and the definitions of $B_x$ and $V^{\rm bad}_i$ with (P1)$^{s+1}_{\ref{embedone}}$ and \ref{main lem Phi 3} imply ($\Phi'$3)$^{s+1}_{\ref{embedone}}$.
Property (P1)$^{s+1}_{\ref{embedone}}$ and \eqref{eq: size Ax} imply that  ($\Phi'$4)$^{s+1}_{\ref{embedone}}$ holds.  ($\Phi'$5)$^{s+1}_{\ref{embedone}}$ is obvious from \eqref{eq: def phi' pj}.
By repeating this for each $s \in [\kappa-1]$, we can obtain our desired packing $\phi':= \phi'_{\kappa}$. Since ($\Phi'$1)$^{\kappa}_{\ref{embedone}}$--($\Phi'$5)$^{\kappa}_{\ref{embedone}}$
imply that $\phi'$ is a packing of $\cH^{\kappa}$ into $G[U]\cup F'$ satisfying \ref{main lem phi' 1}--\ref{main lem phi' 4}, this proves the claim.
\end{proof}
\vspace{0.1cm}

\noindent {\bf Step 4. Packing a $3$-independent set $W^H  \subseteq X^H$ into $U \cup U_0$.} 
In the previous step, we constructed a function $\phi'$ packing $\{H[Y^H\cup Z^H\cup A^H]: H\in \cH\}$ into $G[U]\cup F'$. However, for each graph $H\in \cH$, 
the set $\phi'(H)$ only covers a small part of $U$.
 Eventually we need to cover every vertex of $G$ with a vertex of $H$.
 Hence, for each $H\in\cH$ we will choose a subset $W^H \subseteq X^H$ of size exactly $|U\cup U_0| - |Y^H\cup Z^H|$, and we will construct a function $\phi''$ which packs $\{H[W^H] : H\in \cH\}$ into $G[U \cup U_0]$.
 As later we will extend $\phi'\cup \phi''$ into a packing of $\cH$ into $G\cup F \cup F'$,
we again have to make sure that for any $x\in X^H_i\setminus W^H$ with neighbours in $W^H$, there is a sufficiently large set of candidates to which $x$ can be embedded.
In other words, the set $V_i\cap N(\phi''(N_H(x)\cap W^H))$ needs to be reasonably large. 
To achieve this, we choose $W^H$ to be a $3$-independent set, so $|N_H(x)\cap W^H|\leq 1$, and we will map each vertex $y\in N_H(x)\cap W^H$ into a vertex $v$ which has a large neighbourhood in $V_i$.

Accordingly, for all $H\in \cH$ and $i\in [r]$, we choose a subset $W^H_{i} \subseteq X^{H}_{i}$ satisfying the following:
\begin{enumerate}[label=\text{\rm(W{\arabic*})$_{\ref{embedone}}$}]
\item \label{main W 1} $\bigcup_{i\in  [r]} W^H_i$ is a $3$-independent set of $H$,
\item \label{main W 2} for each $i\in [r]$, we have
$$|W^H_i| =
 |X^H_i| - n' \stackrel{\ref{main lem X5}}{=} n_i - n'- |Y^H_i|- |Z^H_i| \pm \eta^{1/6}n
\stackrel{\eqref{eq: ni def},\eqref{eq: U' properties}, \ref{main lem X5}}{=} \frac{(1\pm \epsilon^{1/4})n}{Tr} . $$
\item \label{main W 3} $\bigcup_{i\in  [r]} W^H_i \cap N^2_H(Z^{H}) = \emptyset.$
\end{enumerate}
Indeed, the following claim ensures that there exist such sets $W^H_i$.
\begin{claim}\label{cl: W exists}
For all $H\in \cH$ and $i\in [r]$, there exists $W^{H}_i\subseteq X^H_i$ such that \ref{main W 1}--\ref{main W 3} hold.
\end{claim}
\begin{proof}
We fix $H\in \cH$.
Assume that for some $i\in  [r]$, we have already defined $W^H_{1},\dots, W^H_{i-1}$ satisfying the following.
\begin{enumerate}[label=\text{\rm(W$'${\arabic*})$^{i-1}_{\ref{embedone}}$}]
\item $\bigcup_{i'\in  [i-1]} W^H_{i'}$ is a $3$-independent set of $H$,
\item for each $i'\in  [i-1]$, we have $|W^H_{i'}| = |X^H_{i'}| - n'=  \frac{(1\pm \epsilon^{1/4})n}{Tr}$,
\item $\bigcup_{i'\in  [i-1]} W^H_{i'} \cap N^2_H(Z^{H}) = \emptyset.$
\end{enumerate}
Consider 
$W'^{H}_i:= X^{H}_i \setminus ( \bigcup_{i'\in  [i-1]} N^{2}_H(W^{H}_{i'}) \cup N^2_H(Z^{H}) ).$ Note that \ref{main lem X6} implies that 
$|N^2_H(Z^{H}) |\leq 8\Delta^{3k^3+2} \eta^{0.9}n$.
Also, \ref{main lem X3} with \ref{main lem X6} implies that 
$$\bigcup_{i'\in [i-1]} N^{2}_H(W^{H}_{i'}) \cap  X^{H}_i
\subseteq N^1_H(Z^{H})  \cup \bigcup_{ i'\in N_{Q}(i) \cap [i-1]} N^{2}_H(W^{H}_{i'}).$$
Thus 
\begin{eqnarray*}
|W'^{H}_i| &\geq& |X^{H}_{i}| - |N^2_H(Z^{H})| - \sum_{ i'\in N_Q(i)\cap [i-1]}|N^{2}_{H}(W^{H}_{i'})|  \\
&\stackrel{\text{(W$'$2)$^{i-1}_{\ref{embedone}}$}}{\geq }&
|X_i^H| - 8\Delta^{3k^3+2}\eta^{0.9} n  - \frac{2k\Delta^2 n}{Tr} \stackrel{\ref{main lem X5},\eqref{eq: size tilde ni}}{\geq}  \Delta^3( |X^H_i|-n') .
\end{eqnarray*}
Thus, by Lemma~\ref{eq: k-independent set}, $W'^{H}_i$ contains a $3$-independent set $W^{H}_i$ of size $|X^H_{i}| - n'$.
Then, by the choice of $W^{H}_i$, (W$'$1)$^{i}_{\ref{embedone}}$--(W$'$3)$^{i}_{\ref{embedone}}$ hold. 
By repeating this for all $i\in [r]$ in increasing order, we obtain $W^{H}_i$ satisfying 
(W$'$1)$^{r}_{\ref{embedone}}$--(W$'$3)$^{r}_{\ref{embedone}}$, and thus  satisfying \ref{main W 1}--\ref{main W 3}. This proves the claim.
\end{proof}
For all $H\in \cH$ and $i\in  [r]$, let 
$W^{H}:=\bigcup_{i'\in  [r]} W^H_{i'} \text{ and } W_i:=\bigcup_{H\in \cH} W^{H}_i,$ where we consider the sets $V(H)$ to be disjoint for different $H\in \cH$. 
Note that for all $H\in \cH$ and $i\in [r]$, Claim~\ref{cl: D f exists} implies that
$0\leq \sum_{j\in N_{D}^+(i)} f^H(\vec{ij}) \leq r \eta^{1/7} n.$
For all $H\in \cH$ and $i\in [r]$, we choose a partition $W^{H,F}_i, W^{H,U'}_i, W^{H,D}_i$  of $W^H_i$ such that
\begin{align}\label{eq: W sizes sizes}
|W^{H,U'}_i| = |U'_i| \enspace \text{and} \enspace |W^{H,D}_i| =\sum_{j\in N_{D}^+(i)} f^H(\vec{ij})\leq r\eta^{1/7}n.
\end{align}
Such partitions exist by \eqref{eq: U' properties}, \ref{main W 2} and the fact that $\eta \ll \epsilon \ll 1/T$. 
For each $S\in \{F,D,U'\}$, we let $W^{H,S}:= \bigcup_{i\in [r]} W_i^{H,S}$.

We now construct a function $\phi''$ which maps all the vertices of $W^H$ into $U_0\cup (U\setminus \phi'( Y^{H}\cup Z^{H})) $ for each $H\in \cH$. 
(In Step 6 we will then apply Theorem~\ref{Blowup} to embed all the vertices of $X^{H}\setminus W^{H}$ into $V$.)
We will define $\phi''$ separately for $W^{H,F}, W^{H,D}$ and $W^{H, U'}.$
We first cover the `exceptional' set $U_0$ with $W^{H,U'}.$
\eqref{eq: W sizes sizes} implies that for all $H\in \cH$ and $i\in [r]$,  there exists a bijection $\phi''^{H}_{U',i}$ from $W^{H,U'}_i$ to $U'_i$.
We let $\phi''_{U'} := \bigcup_{H\in \cH}\bigcup_{i\in [r]} \phi''^{H}_{U',i}$.
Then \eqref{eq: U' properties} implies the following.
\begin{equation}\label{eq: phi'' 1}
\begin{minipage}[c]{0.9\textwidth}\em
For all $i\in [r]$ and $H\in \cH$, the function $\phi''_{U'}$ is bijective between $W^{H,U'}_i$ and $U'_i$. Moreover, for all $x\in W^{H,U'}_i$ and $j\in N_{Q}(i)$, we have $d_{G,V_j}(\phi''_{U'}(x))\geq d^3 n'.$
\end{minipage}
\end{equation}
We intend to embed the neighbours of $W_i^H$ into $\bigcup_{j \in N_Q(i)} V_j.$
Thus it is natural to embed $W_i^H$ into $U_i$ and make use of {\rm (A6)}$_{\ref{embedone}}$. 
This is in fact what we will do for $W_i^{H,F}.$
However, the vertices of $W_i^{H,D}$ will first be mapped to a suitable set of vertices in $U^D_j(i) \subseteq U_j$ for $j \in N_D^+(i)$. 
The definition of $D$ and $f^H$ will ensure that the remaining uncovered part of each $U_j$ matches up exactly with the size of each $W_j^{H,F}.$

By \eqref{eq: UDji size}, for all $\vec{ij}\in E(D)$ and $H\in \cH$, we have 
$$|U^D_j(i)\setminus \phi'(Y^H\cup Z^H)| \geq n/(2Tr) - |Y_j^H\cup Z_j^H| \stackrel{\ref{main lem X5},\ref{main lem X6}}{\geq}  |U_j|/3.$$
For $i\in [r]$ and $H\in \cH$, we let
$$b^H_{i} := \sum_{j\in N_D^{-}(i) } f^H(\vec{ji}) \stackrel{{\rm Claim \text{ }}\ref{cl: D f exists}}{\leq} r \eta^{1/7} n\leq \eta^{1/10} |U_i|.$$
Thus, for each $i\in [r]$, we can apply Lemma~\ref{lem: choice} with the following objects and parameters. \newline

{\small
\noindent
{ 
\begin{tabular}{c|c|c|c|c|c|c|c|c|c|c}
object/parameter& $\kappa$ & $r$&  $H \in \cH$ & $U_i$ & $j\in [r]$ & $ U_i^D(j)\setminus  \phi'(Y^H\cup Z^H) $ & $\eta^{1/10}$ &  $ f^{H}(\vec{ji}) $ & $b^{H}_i$ & $1/3$  \\ \hline
playing the role of & $s$ & $r$ & $ i\in [s]$ & $ A$ & $j\in [r]$ & $ A_{i,j}$ & $\epsilon$  & $m_{i,j} $ & $\sum_{j\in [r]} m_{i,j}$ & $d$ 
\end{tabular}
}}\newline \vspace{0.2cm}

(Recall that $f^H(\vec{ji})= 0$ if $\vec{ji}\notin E(D)$.)
Then we obtain sets $U^H_{i,j}\subseteq U_i$ satisfying the following for each $i\in [r]$, where  $U_i^H := \bigcup_{j\in [r]} U_{i,j}^H$.
\begin{enumerate}[label=\text{\rm(U{\arabic*})$_{\ref{embedone}}$}]
\item\label{main lem U1} For each $j \in [r]$ and $H\in \cH$, we have $|U^H_{i,j}|= f^H(\vec{ji})$ and $U^H_{i,j}\subseteq U_i^{D}(j)\setminus  \phi'(Y^H\cup Z^H)$,
\item\label{main lem U2} for $j\neq j' \in [r]$ and $H \in \cH$, we have $U^H_{i,j}\cap U^H_{i,j'}=\emptyset$,
\item\label{main lem U3} for each $v\in U_i$, we have $|\{ H\in \cH: v\in U_i^{H}\}| \leq \eta^{1/20} |\cH| \stackrel{\eqref{eq: kappa def}}{\leq} \eta^{1/20} n$.
\end{enumerate}
Now for all $H\in \cH$ and $i\in [r]$, we partition 
$W_i^{H,D}$ into $W_{i,1}^{H,D},\dots, W_{i,r}^{H,D}$ in such a way that 
$|W_{i,j}^{H,D}| = f^H(\vec{ij})$. Clearly, this is possible by \eqref{eq: W sizes sizes}. 
Thus \ref{main lem U1} implies that for all $(i,j)\in [r]\times [r]$ and $H\in \cH$, we have
$|W_{i,j}^{H,D}| = f^H(\vec{ij}) = |U^H_{j,i}| $.
Thus there exists a bijection $\phi''^H_{D,i,j}: W_{i,j}^{H,D}\rightarrow U^H_{j,i}$.
Let $\phi''_{D} := \bigcup_{ (i,j)\in [r]\times[r]}\bigcup_{H\in \cH} \phi''^{H}_{D,i,j}$.
Then, for $\vec{ij}\in E(D), H\in \cH$ and $y \in W_{i,j}^{H,D}$, we have that 
$$\phi''_{D}(y) \in U^H_{j,i} \subseteq U_j^D(i)\setminus \phi'(Y^H\cup Z^H).$$
Thus, \eqref{eq: UDji def} with  \ref{main lem U1} and \ref{main lem U2} implies the following. 
\begin{equation}\label{eq: phi'' 2}
\begin{minipage}[c]{0.9\textwidth}\em
For each $H\in \cH$, the function $\phi''_{D}$ is bijective between $\bigcup_{i\in[r]} W_{i}^{H,D} =W^{H,D}$ and $\bigcup_{i\in [r]} U_i^H$. Moreover, for all $x\in W_{i}^{H,D}$ and $j'\in N_{Q}(i)$, we have $d_{G,V_{j'}}(\phi''_{D}(x))\geq d^3 n'/2.$
\end{minipage}
\end{equation}
Now, for all $H\in \cH$ and $i\in [r]$
\begin{eqnarray*}
|W^{H,F}_i| &=& |W^{H}_i| - |W^{H,U'}_i| - |W^{H,D}_i| \stackrel{\eqref{eq: W sizes sizes},\ref{main W 2}}{=} 
(|X_i^H| - n') - |U'_i| - \sum_{j\in N_{D}^+(i)} f^H(\vec{ij}) \\
&\stackrel{\ref{main lem X5}}{=}&
\widehat{n}^H_i  - \sum_{j\in N_{D}^+(i)} f^H(\vec{ij}) - |Y^H_i|-|Z^H_i| - n' - |U'_i|  \\
&\stackrel{{\rm Claim\text{ } }\ref{cl: D f exists}}{=}&
n_i- \sum_{j\in N_{D}^{-}(i)} f^H(\vec{ji})  - |Y^H_i|-|Z^H_i| - n' - |U'_i|  \\
&\stackrel{\eqref{eq: ni def},\ref{main lem U1}}{=}&
|U_i|- |Y^H_i| - |Z^H_i| - \sum_{j\in N_{D}^{-}(i)} |U^H_{i,j}|
\stackrel{\ref{main lem phi' 3}}{=} |U_i \setminus (\phi'(Y^H_i\cup Z^H_i) \cup U^H_{i} )|. \\
\end{eqnarray*}
Thus, there exists a bijection $\phi''^{H}_{F,i}$ from $W^{H,F}_i$ to $U_i \setminus (\phi'(Y^H_i\cup Z^H_i) \cup U^H_{i} )$. 
Let $\phi''_F:=\bigcup_{H\in \cH}\bigcup_{i\in [r]}  \phi''^{H}_{F,i}.$
Then (A6)$_{\ref{embedone}}$ implies the following.
\begin{equation}\label{eq: phi'' 3}
\begin{minipage}[c]{0.9\textwidth}\em
For all $H\in \cH$ and $i\in [r]$, the function $\phi''_{F}$ is bijective between $W_i^{H,F}$ and $U_i \setminus (\phi'(Y^H_i\cup Z^H_i) \cup U^H_{i} )$. Moreover,
for all $x\in W^{H,F}_i$ and $j\in N_{Q}(i)$, we have $d_{F,V_{j}}(\phi''_{F}(x))\geq d^3 n'.$
\end{minipage}
\end{equation}
We define 
\begin{equation} \label{hello}
\phi'' :=  \phi''_{U'} \cup \phi''_D  \cup \phi''_F \enspace \text{and} \enspace \phi_*:= \phi'\cup \phi''.
\end{equation}
Then \eqref{eq: phi'' 1}, \eqref{eq: phi'' 2} and \eqref{eq: phi'' 3} imply that
$\phi''$ is bijective between $W^H$ and $(U\cup U_0) \setminus \phi'(Y^H\cup Z^H)$, when restricted to $W^H$ for each $H\in \cH$. 
Thus, we know that 
\begin{equation}\label{eq: phi* injective}
\begin{minipage}[c] {0.8\textwidth}
$\phi_*$ is bijective between $W^H \cup Y^H \cup Z^H \cup A^H$ and $U\cup U_0\cup V_0$ for each $H\in \mathcal{H}$.
\end{minipage}
\end{equation}
Moreover, \eqref{eq: phi'' 1}, \eqref{eq: phi'' 2} and \eqref{eq: phi'' 3} imply that the following hold for all $i\in [r]$ and $H\in \cH$.
\begin{enumerate}[label=\text{\rm($\Phi_*${\arabic*})$_{\ref{embedone}}$}]
\item \label{main lem Phi*1} If $x\in W_i^{H,F}$, then $\phi_*(x)\in U$ and, for each $j\in N_{Q}(i)$, we have $d_{F,V_j}(\phi_*(x))\geq d^3 n'$,
\item \label{main lem Phi*2} if $x\in W_i^{H,D}$, then $\phi_*(x)\in U$ and, for each $j\in N_{Q}(i)$, we have $d_{G,V_j}(\phi_*(x))\geq d^3 n'/2$,
\item \label{main lem Phi*3} if $x\in W_i^{H,U'}$, then $\phi_*(x) \in U_0$ and, for each $j\in N_{Q}(i)$, we have $d_{G,V_j}(\phi_*(x))\geq d^3 n'$.
\end{enumerate}
Furthermore, \ref{main lem phi' 2} with \ref{main lem U3} implies that
\begin{enumerate}[label=\text{\rm($\Phi_*${\arabic*})$_{\ref{embedone}}$}]
 \setcounter{enumi}{3}
\item \label{main lem Phi*4} for $u\in U$, we have 
$|\{H\in \cH : u\in \phi_*(Y^H\cup Z^H \cup W^{H,D})\}|\leq 2\epsilon^{1/8}n/r.$
\end{enumerate}
\vspace{0.1cm}

\noindent {\bf Step 5. Packing the graphs $H[X^{H}\setminus W^H]$ into internally regular graphs.} 
Note that (X6)$_{\ref{embedone}}$ and (W3)$_{\ref{embedone}}$ together imply that $N_H(W^H)\cap (Y^H\cup Z^H \cup A^H)=\emptyset$ for each $H\in \cH$. 
This implies that $\phi_*$ is a function packing $\{ H[Y^H\cup Z^H\cup W^H\cup A^H] : H\in \cH\}$ into $G[U \cup U_0]\cup F'$.
We wish to pack the remaining part $H[X^{H}\setminus W^H]$ of each $H\in \cH$ into $G[V]$ by using Theorem~\ref{Blowup}. 
In order to be able to apply Theorem~\ref{Blowup}, we first need to pack suitable subcollections of $\cH$ into internally $q$-regular graphs. 
More precisely, for each $t\in [T]$, we will partition $\cH_{t}$ into $\cH_{t,1},\dots, \cH_{t,w}$ and apply Lemma~\ref{Pack} to the unembedded part of each graph in $\cH_{t,w'}$ to pack all these parts into a graph $H_{t,w'}$ on $|V|$ vertices which is internally $q$-regular.
We can then use Theorem~\ref{Blowup} to pack all the $H_{t,w'}$ into $G[V]$ in Step 6.

For this purpose, we choose an integer $q$ and a constant $\xi$ such that $1/T \ll 1/q \ll \xi \ll \alpha$ and let 
\begin{align}\label{eq: w upper bound}
w:= \frac{ e(\cH) }{ (1-3\xi) T (k-1)q n/2} \stackrel{\text{(A2)$_{\ref{embedone}}$}}{\leq}   \frac{(1-\nu/2)\alpha n'}{ qT}.
\end{align}
By using \eqref{eq: number of edges in H} and \eqref{eq: edges in cHi}, for each $t\in [T]$, we can further partition $\cH_t$ into 
$\cH_{t,1},\dots, \cH_{t,w}$ such that for each $(t,w') \in [T]\times [w]$, we have
\begin{eqnarray}\label{eq: edgesum}
e(\mathcal{H}_{t,w'}) = (1-3\xi) (k-1) qn/2 \pm 2 \Delta n
= (1 - 3\xi \pm \xi /2)(k-1)qn/2.
\end{eqnarray}
By (A1)$_{\ref{embedone}}$, we have
\begin{eqnarray}\label{eq: q upper bound} 
 |\mathcal{H}_{t,w'}| \le 2(k-1) q \le (q \xi)^{3/2}. 
\end{eqnarray} 
For all $H \in \mathcal{H}$ and $i\in  [r]$, let 
$\widetilde{X}^{H}_i:= X_i^H\setminus W_i^H \text{ and } \widetilde{X}^H := \bigcup_{i\in  [r]} \widetilde{X}^H_i.$ 
Thus, by \ref{main W 2} we have $|\widetilde{X}^H_i|=n'$ for all $H\in \cH$ and $i\in [r]$. Moreover, for all $t\in [T]$, $w'\in [w]$ and $ij\in E(Q)$, we have
\begin{eqnarray}\label{eq: number or HH edge}
\hspace{-0.5cm} \sum_{H\in \cH_{t,w'}} e(H[\widetilde{X}^H_i, \widetilde{X}^H_{j}]) \hspace{-0.4cm} &=&  \hspace{-0.4cm} \sum_{H\in \cH_{t,w'}} (e(H[X^H_i, X^H_{j}]) \pm \Delta(|W_i^H|+|W_{j}^{H}|) \nonumber \\\hspace{-0.4cm}  & \stackrel{\ref{main lem X4}, \ref{main W 2}}{=} & \hspace{-0.4cm} \sum_{H\in \cH_{t,w'}}\hspace{-0.2cm} \left( \frac{2e(H) \pm \epsilon^{1/5}n }{(k-1)r} \pm \frac{3\Delta n}{Tr} \right) \stackrel{\eqref{eq: edgesum}}{=}  (1- 3\xi \pm \xi) q n'.
\end{eqnarray}
When packing $H[\tilde{X}^H]$ and $H'[\tilde{X}^{H'}]$ (say) into the same graph $H_{t,w'}$, we need to make sure that the `attachment sets' of $H[\tilde{X}^H]$ and $H'[\tilde{X}^{H'}]$ are not mapped to the same vertex sets in $H_{t,w'}.$
The attachment set for $H[\tilde{X}^H]$ contains those vertices of $\tilde{X}^H$ which have a neighbour in $W^H \cup Y^H \cup Z^H \cup A^H$ (more precisely, a neighbour in $W^H \cup Z^H$) and is defined in \eqref{eq: def NW NZ}.
Keeping these attachment sets disjoint in  $H_{t,w'}$ ensures that we can make the embedding of each $\tilde{X}^H$ consistent with the existing partial embedding of $H$ without attempting to use an edge of $F$ or $G$ twice.
For all $i\in  [r]$ and $H\in \mathcal{H}$, we let
\begin{align}\label{eq: def NW NZ}
N_{i}^{H,F} := \bigcup_{i'\in N_{Q}(i)} N_{H}^1(W_{i'}^{H,F}) \cap \widetilde{X}^{H}_i \enspace \text{ and }
N_i^{H,G} := N^1_H( Z^H \cup W^{H,D} \cup \bigcup_{i'\in N_{Q}(i)}W^{H,U'}_{i'} )\cap \widetilde{X}^H_i.
\end{align}
Note that \ref{main W 1}, \ref{main W 3} and the fact that $W^{H,F}, W^{H,D}, W^{H,U'}$ form a partition of $W^H$ implies that 
\begin{align}\label{eq NW NZ disjoint}
N_{i}^{H,F} \cap N_i^{H,G}=\emptyset.
\end{align}
Moreover, if $x\in N_i^{H,F}$ then $x$ has a unique neighbour in $W^{H,F}$. 
Similarly, if $x\in N_i^{H,G}$, then either $x$ has a unique neighbour in $W^{H,D}\cup W^{H,U'}$ or $x$ has at least one neighbour in $Z^H$ (but not both). 
Note that for $i\in  [r]$ and $H\in \mathcal{H}$,

\begin{eqnarray}\label{eq: NWH size}
|N_{i}^{H,F} \cup N_i^{H,G}|  \hspace{-0.3cm} &\leq&  \hspace{-0.3cm}
\sum_{i'\in N_{Q}(i)} \Delta( |W_{i'}^{H,F}|+|W_{i'}^{H,U'}|) + \Delta ( |Z^H| + |W^{H,D}|) \nonumber \\
 \hspace{-0.3cm} &\stackrel{\substack{\ref{main lem X6},\\ \ref{main W 2},\eqref{eq: W sizes sizes}}}{\leq} &  \hspace{-0.3cm}  \frac{2\Delta k n}{Tr}  + 4\Delta^{3k^3+1}\eta^{0.9} n +\Delta r^2 \eta^{1/7}n 
  \leq T^{-2/3} n'.
\end{eqnarray}

For each $i\in [r]$, we consider a set $\widehat{X}_i$ with $|\widehat{X}_i|=n'$ such that $\widehat{X}_1,\dots,\widehat{X}_r$ are pairwise vertex-disjoint. For each $(t,w')\in [T]\times [w]$, let $\cH_{t,w'}=:\{ H_{t,w'}^{1},\dots, H_{t,w'}^{h(t,w')}\}$.
Then, by \eqref{eq: q upper bound}, \eqref{eq: number or HH edge}, \eqref{eq: NWH size} and \ref{main lem X3}, we can apply Lemma~\ref{Pack} with the following objects and parameters for each $(t,w') \in [T]\times [w]$.\newline

{\small
\noindent
{ 
\begin{tabular}{c|c|c|c|c|c|c|c|c|c|c}
object/parameter &$H_{t,w'}^{j}[\widetilde{X}^{H_{t,w'}^{j}}]$ & $\widetilde{X}^{H_{t,w'}^{j}}_i$ & $\widehat{X}_i$ & $n'$ & $q$ &  $ \xi $ &  $ T^{-2/3} $ & $ h(t,w') $  & $N_i^{H_{t,w'}^{j},F}\cup N_i^{H_{t,w'}^{j},G}$ & $Q$  \\ \hline
playing the role of & $L_j$ & $ X_i^j$ &$V_i$ & $n$ & $q$ & $ \xi$  &  $ \epsilon$ & $ s $ & $W_i^j$ & $R$  
\end{tabular}
}}\newline \vspace{0.2cm}

Then for each $(t,w')\in [T]\times [w]$, we obtain a function $\Phi_{t,w'}$ packing $\{ H[\widetilde{X}^H] : H\in \cH_{t,w'}\}$ into some graph $H_{t,w'}$ which is internally $q$-regular with respect to the vertex partition $(Q,\widehat{X}_1,\dots, \widehat{X}_r)$.
Moreover, for all $i\in  [r]$ and $H \in \mathcal{H}_{t,w'}$ we have $\Phi_{t,w'}(\widetilde{X}^H_i) = \widehat{X}_i$ and for distinct $H, H'\in \mathcal{H}_{t,w'}$ and $i\in  [r]$, we have
\begin{eqnarray}\label{eq: NW disjoint}
\Phi_{t,w'}(N_i^{H,F} \cup N_i^{H,G}) \cap \Phi_{t,w'}(N_{i}^{H',F}\cup N_i^{H',G}) = \emptyset.
\end{eqnarray} 
Note that for all $(t,w')\in [T]\times [w]$, the graphs $H_{t,w'}$ have same vertex set $\bigcup_{i\in [r]} \widehat{X}_i$.
For all $i\in [r]$ and $(t,w') \in [T]\times [w]$, we let
\begin{align}\label{eq: hat W hat Z def}
L_i^{t,w'}:= \bigcup_{H\in \mathcal{H}_{t,w'} } \Phi_{t,w'}(N_i^{H,F}) \enspace \text{ and } \enspace M_i^{t,w'}:= \bigcup_{H\in \mathcal{H}_{t,w'} } \Phi_{t,w'}(N_i^{H,G}).
\end{align}
Then by \eqref{eq NW NZ disjoint} and \eqref{eq: NW disjoint} we have
\begin{align}\label{eq: hat W hat Z disjoint}
L_i^{t,w'}\cup M_i^{t,w'} \subseteq \widehat{X}_i \enspace \text{ and } \enspace
L_i^{t,w'}\cap M_i^{t,w'}=\emptyset.
\end{align}
By \eqref{eq: q upper bound} and \eqref{eq: NWH size},  for all $(t,w')\in[T]\times [w]$ and $i\in [r]$
 \begin{eqnarray} \label{eq: Yij size}
 |L_i^{t,w'}\cup M_i^{t,w'}| \le q^{3/2} T^{-2/3} n'\leq T^{-1/2}n'. 
 \end{eqnarray}
 \vspace{0.05cm}

\noindent {\bf Step 6. Packing the internally regular graphs $H_{t,w'}$ into $G[V]$.} In the previous step, we constructed a collection $\widehat{\mathcal{H}} := \{H_{1,1}, \dots, H_{T,w}\}$ of internally $q$-regular graphs on $|V|$ vertices.
 We now wish to apply Theorem \ref{Blowup}  to pack $\widehat{\mathcal{H}}$ into $G[V]$.
 However, our packing needs to be consistent with the packing $\phi_*$. 
Note that for each $H\in \cH$ the set $W^{H}\cup Y^{H}\cup Z^{H}\cup A^H$ consists of exactly those vertices of $H$ which are already embedded by $\phi_*$.
Thus by \ref{main lem X3}, \ref{main lem X6}, \eqref{eq: def NW NZ} and \eqref{eq: hat W hat Z def}, it follows that whenever $x \in \widehat{X}_i$ is a vertex of $H_{t,w'}$ such that the set $\Phi^{-1}_{t,w'}(x)$ of pre-images of $x$ contains a neighbour of some vertex which is already embedded by $\phi_*$, then $x \in L_i^{t,w'}\cup M_i^{t,w'}$.
 Thus in order to ensure that our packing of $\widehat{\cH}$ is consistent with $\phi_*$, for each $i\in [r]$, each $(t,w')\in [T]\times[w]$ and each $y\in L_i^{t,w'}\cup M_i^{t,w'}$ we will choose a suitable target set $A_y^{t,w'}$ of vertices of $G[V]$ and will map $y$ into this set.

For all $(t,w')\in [T]\times [w]$, $i\in  [r]$ and any vertex $y \in  L_i^{t,w'}\cup M_i^{t,w'}$, 
\eqref{eq: NW disjoint} implies that there exists a unique graph $H^{t,w'}_y\in \mathcal{H}_{t,w'}$ and a unique vertex $x^{t,w'}_y \in N_{i}^{H^{t,w'}_y,F}\cup N_i^{H^{t,w'}_y,G}$ such that $y = \Phi_{t,w'}(x^{t,w'}_y)$.  Let 
$$J^{t,w'}_y:= N_{H^{t,w'}_y}(x^{t,w'}_y) \cap  ( W^{H^{t,w'}_{y}}\cup Z^{H^{t,w'}_{y}} )=N_{H^{t,w'}_y}(x^{t,w'}_y) \cap (W^{H^{t,w'}_y}\cup Y^{H^{t,w'}_y} \cup Z^{H^{t,w'}_y}\cup A^{H^{t,w'}_y}).$$
The final equality follows from \ref{main lem X6}.
For all $(t,w')\in [T]\times [w]$, $i\in  [r]$ and any vertex $y \in  L_i^{t,w'}\cup M_i^{t,w'}$, we define the target set
$$A^{t,w'}_y:=  \left\{ \begin{array}{ll}
N_{F}(\phi_*(J^{t,w'}_y))\cap V_i & \text{ if } x^{t,w'}_y \in N_{i}^{H^{t,w'}_y,F},\\
N_G( \phi_*(J^{t,w'}_y))\cap V_i & \text{ if } x^{t,w'}_y\in N_{i}^{H^{t,w'}_y,G}.
\end{array}\right.$$
Note that $A^{t,w'}_y$ is well-defined as \eqref{eq NW NZ disjoint} implies that exactly one of the above cases holds. 
Moreover, the following claim implies that these target sets are sufficiently large.
\begin{claim}\label{cl: Aqwy size}
For all $(t,w')\in [T]\times [w]$, $i\in  [r]$ and any vertex $y \in  L_i^{t,w'}\cup M_i^{t,w'}$, we have
$$|A^{t,w'}_y|\geq d^{5\Delta} |V_i|.$$
\end{claim}
\begin{proof}
We fix $(t,w')\in [T]\times [w]$, $i\in [r]$ and a vertex $y\in L_i^{t,w'}\cup M_i^{t,w'}$.
For simplicity, we write $H:= H^{t,w'}_y$, $x:= x^{t,w'}_y$ and $J:= J^{t,w'}_y$. Then \eqref{eq NW NZ disjoint} implies that exactly one of the following two cases holds.

\noindent {\bf Case 1.} $x \in N_{i}^{H,F} $.
In this case, \ref{main W 1} and \ref{main W 3} imply that
\begin{align*}
 J = N_{H}(x) \cap W^{H,F} \stackrel{\ref{main lem X3}}{=} N_{H}(x) \cap \bigcup_{i'\in N_{Q}(i)} W_{i'}^{H,F} \enspace \text{ and } \enspace |J|= 1.
\end{align*}
Then by \ref{main lem Phi*1}, we know $\displaystyle |A^{t,w'}_y| \geq d^{3}|V_i|.$

\noindent {\bf Case 2.} $x\in N_{i}^{H,G}$.
In this case, by \eqref{eq: def NW NZ} and \ref{main W 3}, we have exactly one of the following cases.

\noindent {\bf Case 2.1} $x\in N_{H}^{1}(Z^{H})$.
In this case, $N_{H}(x)\cap W^{H} =\emptyset$ by \ref{main W 3}.
Thus we have $J = N_{H}(x) \cap Z^{H}$.
Then \eqref{eq: NH def} and \ref{main lem phi' 1}  imply that
$\displaystyle |A^{t,w'}_y| = |N_G(\phi'( e_{H,x})) \cap V_{f_{H}(e_{H,x})} |  \geq d^{5\Delta}|V_i|.$

\noindent {\bf Case 2.2} $x\in N_{H}^{1}(W^{H,D}\cup W^{H,U'})$.
In this case, again \ref{main W 1}, \ref{main W 3}   and \ref{main lem X3} imply that
\begin{align*}
 J = N_{H}(x) \cap (W^{H,D}\cup W^{H,U'}) = N_{H}(x) \cap \bigcup_{i' \in N_{Q}(i)}  (W_{i'}^{H,D} \cup W_{i'}^{H,U'} )\enspace \text{ and } \enspace
|J| = 1.
\end{align*}
Thus \ref{main lem Phi*2} or \ref{main lem Phi*3} imply that $\displaystyle |A^{t,w'}_y| \geq d^{3}|V_i|/2.$
This proves the claim.
\end{proof}

Let $\mathbf{S}:= [T]\times [w]$.  Let $\Lambda$ be the graph with 
$$V(\Lambda) := \{(\vec{s},y): \vec{s}\in \mathbf{S}, y\in \bigcup_{\vec{s}\in \mathbf{S}, i\in [r]} L_i^{\vec{s}}\cup M_i^{\vec{s}} \}$$ and 
\begin{eqnarray*}
E(\Lambda) := \left\{(\vec{s},y)(\vec{t},y')  : \vec{s} \neq\vec{t}\in \mathbf{S}, i\in [r],
(y,y') \in  (L_i^{\vec{s}} \times  L_i^{\vec{t}} )\cup  (M_i^{\vec{s}}\times  M_i^{\vec{t}})
\text{ and } \phi_*( J^{\vec{s}}_y) \cap \phi_*(J^{\vec{t}}_{y'} )\neq \emptyset \right\}.
\end{eqnarray*}
Note that $\Lambda$ is the graph indicating possible overlaps of images of distinct edges when we extend $\phi_*$.
Indeed, if $(\vec{s},y)$ and $(\vec{t},y')$ are adjacent in $\Lambda$, there are $z\in N_{H^{\vec{s}}_y}(x^{\vec{s}}_y)$ and $z'\in N_{H^{\vec{t}}_{y'}}(x^{\vec{t}}_{y'})$ such that $\phi_*(z)=\phi_*(z')$. 
If we embed $y$ and $y'$ onto the same vertex, then the two edges $x^{\vec{s}}_yz$ and $x^{\vec{t}}_{y'}z'$ would be embedded onto the same edge of $G\cup F$.
Thus we need to ensure that $\phi(y)\neq \phi(y')$. 
\COMMENT{
Note that if $y\in L_i^{\vec{s}}$ and $y'\in M_i^{\vec{t}}$, then $(\vec{s},y)(\vec{t},y') \notin E(\Lambda)$ in our definition of $\Lambda$.
Suppose $y\in L_i^{\vec{s}}$ and $y'\in M_i^{\vec{t}}$
and suppose $\phi_*(J^{\vec{s}}_{y})\cap \phi_*(J^{\vec{t}}_{y'})\neq \emptyset$.
Then for $v\in \phi_*(J^{\vec{s}}_{y})\cap \phi_*(J^{\vec{t}}_{y'})$, we know that 
$A^{\vec{s}}_{y} \subseteq N_{F}(v)$ and 
$A^{\vec{t}}_{y'} \subseteq N_{G}(v)$. As $G$ and $F$ are edge-disjoint graphs, $\phi(y)$ is never same with $\phi(y')$ as long as $\phi(y)\in A^{\vec{s}}_y$ and $\phi(y') \in A^{\vec{t}}_{y'}$. Thus edge $(\vec{s},y)(\vec{t},y')$ is not necessary}

Note that for all $(\vec{s},y)\in V(\Lambda)$ and $\vec{t}\in \mathbf{S}$, we have {
\begin{eqnarray}\label{eq: q2 upperbound}
|\{ (\vec{t},y') \in N_{\Lambda}((\vec{s},y)) \}| \hspace{-0.2cm} &\leq& \hspace{-0.2cm} 
|\{ y' : H^{\vec{t}}_{y'}\in \cH_{\vec{t}},  \phi_*(J^{\vec{t}}_{y'})\cap \phi_*(J^{\vec{s}}_y)\neq \emptyset \}|\nonumber \\
& \leq&  \hspace{-0.2cm}  \sum_{ v\in \phi_*(J^{\vec{s}}_y) } |\{y': H^{\vec{t}}_{y'}\in \cH_{\vec{t}},  v\in \phi_*(J^{\vec{t}}_{y'})   \}| \nonumber \\
&\leq& \hspace{-0.2cm} 
 \sum_{ v\in \phi_*(J^{\vec{s}}_y) } \sum_{H \in \cH_{\vec{t}}} |\{ x  \in V(H): v\in \phi_*(N_{H}(x))   \}|  \nonumber \\
& \leq &  \sum_{ v\in \phi_*(J^{\vec{s}}_y) } \sum_{H \in \cH_{\vec{t}}} |\{ x \in N_H(x') : v= \phi_*(x'), x'\in V(H) \}| \nonumber \\
&\stackrel{\eqref{eq: phi* injective}}{\leq} & \sum_{ v\in \phi_*(J^{\vec{s}}_y) } \sum_{H \in \cH_{\vec{t}}} \Delta
\leq \Delta^2|\mathcal{H}_{\vec{t}}|  
\stackrel{\eqref{eq: q upper bound} }{\leq} \Delta^2 (\xi q)^{3/2} \leq q^2.
\end{eqnarray}}
(Here the third inequality holds by the definition of $J^{\vec{t}}_{y'}$ and the definition of $x^{\vec{t}}_{y'}$, the fifth inequality holds since \eqref{eq: phi* injective} implies that there is at most one $x'\in V(H)$ with $\phi_*(x')=v$, and the sixth inequality holds since $|J_y^{\vec{s}}|\leq |N_{H_y^{\vec{s}}}(x_{y}^{\vec{s}})|\leq \Delta$.)

Consider any $(\vec{s},y)\in V(\Lambda)$. Then similarly as above we have
\begin{eqnarray*}
d_{\Lambda}((\vec{s},y)) \leq  \sum_{v \in \phi_*(J^{\vec{s}}_y) } 
\sum_{H\in \cH} |\{ x\in N_{H}(x'): v = \phi_*(x'), x'\in V(H) \}| 
\leq  \Delta^2 |\cH| \stackrel{\eqref{eq: kappa def}}{\leq} \alpha^{1/2} n'.
\end{eqnarray*}  
This shows that
 \begin{align}\label{eq: Delta(F)}
 \Delta(\Lambda) \leq \alpha^{1/2} n'< d^{5\Delta} n'/2.
 \end{align}
We can now apply the blow-up lemma for approximate decompositions (Theorem~\ref{Blowup}) with the following objects and parameters. \newline

{\small
\noindent
{ 
\begin{tabular}{c|c|c|c|c|c|c|c|c}
object/parameter& $G[V]$ & $V_i$ & $\widehat{X}_i $ & $H_{t,w'}$ & \textbf{S}$=[T]\times [w]$ &$q$& $T^{-1/2}$ & $Q$    \\ \hline
playing the role of &$ G$ & $V_i$ & $X_i$ &  $H_i$ &$[s]$ &$q$& $\epsilon$  & $R$  \\ \hline\hline
object/parameter & $r$ & $L_i^{t,w'}\cup M_i^{t,w'}$ & $A_y^{t,w'}$  & $\alpha$& $d^{5\Delta}$ & $\nu$ & $\Lambda$ &  $n'$ \\ \hline
playing the role of  & $r$ & $W_i^j$ &$A_w^j$& $d$ & $d_0$ & $\alpha$ & $\Lambda$ &  $n$
\end{tabular}
}}\newline \vspace{0.2cm}

\noindent Indeed, (A3)$_{\ref{embedone}}$ implies that (A1)$_{\ref{Blowup}}$ holds, and (A2)$_{\ref{Blowup}}$ holds by the definition of $H_{t,w'}$.
 Claim~\ref{cl: Aqwy size} and \eqref{eq: Yij size} imply that (A3)$_{\ref{Blowup}}$ holds, and 
\eqref{eq: Yij size}, \eqref{eq: q2 upperbound}  and \eqref{eq: Delta(F)} imply that (A4)$_{\ref{Blowup}}$ holds. 
Moreover, \eqref{eq: w upper bound} implies that the upper bound on $s$ in the assumption of Theorem~\ref{Blowup} holds.

Thus by Theorem \ref{Blowup} we obtain a function $\phi^*$ that packs $\{H_{\vec{s}}: \vec{s}\in \mathbf{S} \}$ into $G[V]$ and satisfies the following, where $\phi^*_{\vec{s}}$ denotes the restriction of $\phi^*$ to $H_{\vec{s}}.$
\begin{enumerate}[label=\text{\rm ($\Phi^*$\arabic*)$_{\ref{embedone}}$}]
\item\label{Phi*1}  for each $\vec{s}\in$ \textbf{S} and
$y \in \bigcup_{i\in [r]} L_i^{\vec{s}}\cup M_i^{\vec{s}}$, we have $\phi^*_{\vec{s}}(y) \in A_y^{\vec{s}}$,
\item\label{Phi*2} for any $(\vec{s},y)(\vec{t},y') \in E(\Lambda),$ we have that $\phi^*_{\vec{s}}(y) \ne \phi^*_{\vec{t}}(y').$
\end{enumerate}
We let 
$$\phi:= \phi^*(\bigcup_{\vec{s}\in \mathbf{S}}\Phi_s) \cup \phi_*.$$
Recall from Step 3 and~\eqref{hello} that $\phi_*=\phi'\cup \phi''$, and that $\phi'$ packs
 $\{H[Y^H\cup Z^H\cup A^H]: H\in \cH\}$ into $G[U]\cup F'$.  Since
 each $\Phi_{\vec{s}}$ is a packing of $\{ H[X^{H}\setminus W^{H}] : H\in \mathcal{H}_{\vec{s}}\} $ into $H_{\vec{s}}$ and $\phi^*$ is a packing of $\{H_{\vec{s}}: \vec{s}\in \mathbf{S}\}$ into $G[V]$,
we know that $\phi$ packs $\{H[X^H\setminus W^H]: H\in \cH\}$ into $G[V]$.
Moreover, \ref{Phi*1}, \ref{Phi*2} with the definitions of $A^{\vec{s}}_y$ and $\Lambda$ imply that $\phi$ packs $\{H[X^H\setminus W^H, W^{H,F}]: H\in \cH\}$ into $F$, and 
$\phi$ packs $\{H[X^H\setminus W^H,  W^{H,U'}]: H\in \cH\}$ into $G[V,U_0]$, and
$\phi$ packs $\{H[X^H\setminus W^H, W^{H,D}\cup Z^H]: H\in \cH\}$ into $G[V,U]$.
Thus, we have the following.
\begin{align}\label{eq: main lem phi}
&\phi(\bigcup_{H\in \cH} E_H( Y^H\cup Z^H\cup A^H)) \subseteq E_G(U)\cup E(F'), \qquad \phi(\bigcup_{H\in \cH} E_H( X^{H}\setminus W^H)) \subseteq  E_G(V), \nonumber \\
&\phi(\bigcup_{H\in \cH} E_H(X^H\setminus W^H, W^{H,F})) \subseteq E_F(V,U), \qquad
\phi(\bigcup_{H\in \cH} E_H(X^H\setminus W^H, W^{H,U'})) \subseteq E_G(V,U_0),   \enspace  \nonumber \\ 
&\phi(\bigcup_{H\in \cH}  E_H(X^H \backslash W^H, W^{H,D}\cup Z^H))\subseteq E_G(V,U).
\end{align}

Also, it is obvious that the restriction of $\phi$ to $V(H)$ is injective for each $H\in \cH$.
As $G[U]\cup F', G[V], F, G[V,U_0]$ and $G[V,U]$ are pairwise edge-disjoint, we conclude that $\phi$ packs $\mathcal{H}$ into $G\cup F\cup F'$.
Moreover, by \eqref{eq: kappa def} we have
$\Delta(\phi(\mathcal{H})) \leq \Delta|\mathcal{H}|  \leq 4 k\Delta \alpha n/r,$
thus (B1)$_{\ref{embedone}}$ holds.  By \eqref{eq: main lem phi} and \ref{main lem phi' 4}, for $u\in U$, we have \vspace{-0.2cm}
 $$d_{\phi(\cH)\cap G}(u) \leq \Delta |\{H \in \cH : u\in \phi_*(Y^H\cup Z^H\cup W^{H,D})\}|
\stackrel{\ref{main lem Phi*4}}{\leq} \frac{ 2\Delta\epsilon^{1/8}n}{r}.$$
Thus (B2)$_{\ref{embedone}}$ holds.


Finally, for $i\in [r]$, by \ref{main lem X3}, \ref{main lem X6}, \eqref{eq: main lem phi} we have
\begin{eqnarray*}
e_{\phi(\mathcal{H})\cap G}(V_{i}, U\cup U_0) &\leq& \sum_{H \in \mathcal{H} } 
\Delta \Big( |Z^H| + \sum_{j\in N_{Q}(i)} |W^{H,U'}_j| +  \sum_{j\in N_Q(i)} |W^{H,D}_j| \Big)\\
&\stackrel{\substack{ \eqref{eq: kappa def},\ref{main lem X6},
\\ \eqref{eq: U' properties}, \eqref{eq: W sizes sizes} } }{\leq} & 
\frac{2k\Delta \alpha n}{r} \left( 4\Delta^{3k^3} \eta^{0.9} n + 2(k-1) \epsilon^{3/4} n/r + (k-1)r \eta^{1/7} n \right) \\
&\leq& \frac{\epsilon^{1/2}n^2}{r^2},
\end{eqnarray*}
 which shows that (B3)$_{\ref{embedone}}$ holds.

\end{proof}


\section{Proof of Theorem~\ref{thm:main}}\label{sec: main proof}

The proof of Theorem \ref{thm:main} proceeds in three steps.
In the first step we will apply the results of Section 3 to construct suitable edge-disjoint subgraphs $G_{t,s}, G^*_t,  F_{t,s}$ and $F'_{t}$ of $G$, where $G_{t,s}$ is a $K_k$-factor blow-up spanning almost all vertices
while $G^*_t,  F_{t,s}$ and $F'_{t}$ are comparatively sparse.
In the (straightforward) second step, we simply partition $\cH$ into collections $\cH_{t,s}$ such that the $e(\cH_{t,s})$ are approximately equal to each other.
Finally, in the third step we will pack each $\cH_{t,s}$ into $G_{t,s} \cup G_t^* \cup F_{t,s} \cup F'_t$ via Lemma \ref{embedone}. 

\begin{proof}[Proof of Theorem~\ref{thm:main}] 
Let $\sigma := \delta - \max\{1/2,\delta^{\rm reg}_k\}>0$.
By \eqref{eq: deltak 1-1/k}, we have $\delta \geq 1-1/k + \sigma$ for any $k\geq 2$. Without loss of generality, we may assume that $\nu< \sigma/2$.\COMMENT{We need $\nu < \sigma/2$ when we apply Lemma \ref{lem: factor decomposition}}
For given $\nu, \sigma>0$ and $\Delta$, $k\in \mathbb{N}\backslash \{1\}$, we choose constants $n_0, \xi, \eta , M, M',$ $\epsilon,T,  q, d$ such that $q \mid T$ and 
\begin{align}\label{eq: hierarchy}
0 <1/n_0 \ll \eta \ll 1/M \ll 1/M'  \ll \epsilon\ll 1/T \ll 1/q \ll \xi \ll d  \ll \nu, \sigma, 1/\Delta, 1/k  \leq 1/2.
\end{align}
Suppose  $n \ge n_0$ and 
let $G$ be an $n$-vertex graph satisfying condition (i) of Theorem~\ref{thm:main}. 
Furthermore, suppose $\mathcal{H}$ is a collection of $(k,\eta)$-chromatic $\eta^2$-separable graphs satisfying conditions (ii) and (iii) of Theorem~\ref{thm:main}. We will show that $\cH$ packs into $G$.
Note that we assume $\cH$ to consist of $\eta^2$-separable graphs here (instead of $\eta$-separable graphs). This is more convenient for our purposes, but still implies Theorem 1.2.
\newline

\noindent {\bf Step 1. Decomposing $G$ into host graphs.} 
In this step, we apply Szemer\'edi's regularity lemma to $G$ and then apply Lemma~\ref{Reservoir2} to obtain a partition of $V(G) \backslash V_0$ into $T$ reservoir sets $Res_t$, where $V_0$ is the exceptional set obtained from Szemer\'edi's regularity lemma. 
We use Lemma~\ref{lem: factor decomposition} to obtain an approximate decomposition of the reduced multi-graph $R'_{\rm multi}$ of $G$ into almost $K_k$-factors and partition these factors into $T$ collections.
Each such almost $K_k$-factor $Q$ gives us an $\epsilon$-regular $Q$-blow-up $G_{t,s}$ in $G$, and we modify it into a super-regular $Q$-blow-up.
We also put aside several sparse `connection graphs' $F_{t,s}$ and $F'_t$, which will be used to link vertices in the reservoir and exceptional set with vertices in the rest of the graph.
These connection graphs will play the roles of $F$ and $F'$ in Lemma~\ref{embedone}. We also put aside a
further sparse connection graph $G^*_t$ which provides additional connections within $V(G)\setminus V_0$.

We apply Szemer\'edi's regularity lemma (Lemma \ref{Szemeredi})  with $(\epsilon^2, d)$ playing the role of $(\epsilon, d)$ to obtain a partition $V_0', \dots, V_{r'}'$ of $V(G)$ and a spanning subgraph $G' \subseteq G$ such that 
\begin{enumerate}[label=\text{\rm (R\arabic*)}]
\item \label{R1} $M' \le r' \le M,$
\item \label{R2} $|V_0'| \le \epsilon^2 n,$
\item \label{R3} $|V_i'| = |V_j'| = (1\pm \epsilon^2)n/r'$ for all $i, j \in [r'],$
\item \label{R4} for all  $v \in V (G)$ we have $d_{G'}(v) > d_G(v)  - 2dn$,
\item \label{R5} $e(G'[V'_i]) = 0$ for all $ i \in [r'],$
\item \label{R6} for any $i, j$ with $1 \le i \le j \le r'$, the graph $G' [V'_i, V'_j ]$ is either empty or $(\epsilon^2, d_{i,j})$-regular for some $d_{i,j} \in [d,1]$.
\end{enumerate}
Let $R'$ be the graph with 
$$V(R')= [r'] \qquad \text{ and } \qquad E(R'):=\{ ij: e_{G'}(V'_i,V'_j)>0 \}.$$
Note that for $i,j\in [r']$, $ij\in E(R')$ if and only if $G'[V'_i,V'_j]$ is $(\epsilon^2, d_{i,j})$-regular with $d_{i,j}\geq d$.
Now, we let $R'_{\rm multi}$ be a multi-graph with 
$V(R'_{\rm multi}) = [r']$ and 
with exactly 
\begin{align}\label{eq: tij def}
q_{i,j}:= \lfloor (1-6d)d_{i,j}q \rfloor
\end{align}
edges between $i$ and $j$ for each $ij\in E(R')$. Note that $R'_{\rm multi}$ has edge-multiplicity at most $q$.
For each $i\in [r']$, we have {\small
\begin{eqnarray}\label{eq: sum dij}
d_{R'_{\rm multi}}( i) \hspace{-0.4cm} &=& \hspace{-0.3cm} \sum_{j \in N_{R'}(i)} \lfloor (1-6d)q(\frac{e_{G'}(V'_i,V'_j)}{|V'_i||V'_j|} \pm \epsilon^2)\rfloor  \stackrel{\ref{R3},\ref{R5}}{=} \frac{\sum_{v \in V'_i} (1-6d) qd_{G', V(G)\setminus V_0}(v) }{|V'_i|^2} \pm \epsilon^2 q r' \pm r' \nonumber \\
&\stackrel{\ref{R2},\ref{R4}}{=}& \hspace{-0.3cm}  \frac{q}{|V'_i|^2} \sum_{v\in V'_i} (d_{G}(v) \pm 10d n) \pm 2r' 
\stackrel{(i)}{=} \frac{ (\delta \pm 11d)q n }{|V'_i| }  \pm 2 r' \stackrel{\ref{R3}}{=} (\delta \pm d^{3/4})q r'. 
\end{eqnarray}}
We apply Lemma~\ref{lem: factor decomposition} with $R'_{\rm multi},r',\epsilon^2,k,\sigma,d^{3/4},\nu/5, T$ and $q$ playing the roles of $G$,$n$,$\epsilon,k,\sigma,\xi,\nu, T$ and $q$, respectively.
Then, by permuting indices in $[r']$ if necessary,\COMMENT{The lemma gives us the set $V'$, we can assume $V'= \{r+1,\dots, r'\}$.}
we obtain $R_{\rm multi}\subseteq R'_{\rm multi}$ and a collection $\cQ:=\{ Q_{1,1},\dots, Q_{1,\kappa/T}, Q_{2,1}, \dots, Q_{T,\kappa/T}\}$ of edge-disjoint subgraphs of $R_{\rm multi}$ such that the following hold.
\begin{enumerate}[label=\text{\rm (Q\arabic*)}]
\item\label{Q1} $R_{\rm multi} = R'_{\rm multi}[ [r] ]$ with $(1-\epsilon^2)r' \leq r\leq r'$, and $k \mid r$,
\item\label{Q2} $\kappa = \frac{(\delta - \nu/5 \pm \epsilon^2)q r'}{k-1} = \frac{(\delta - \nu/5 \pm \epsilon )q r}{k-1}$ and $T\mid \kappa$,
\item\label{Q3} for each $(t,s)\in [T]\times [\kappa/T]$, $Q_{t,s}$ is a vertex-disjoint union of at least $(1-\epsilon)r/k$ copies of $K_k$,
\item\label{Q4} for each $i\in [r]$, we have $|\{ (t,s) \in [T]\times [\kappa/T]  : i \in V(Q_{t,s}) \}| \geq \kappa - \epsilon r$.
\item\label{Q5} for all $t\in [T]$ and $i,j\in [r]$, we have 
$|\{ s\in [\kappa/T] : j\in N_{Q_{t,s}}(i)\}|\leq 1$.
\end{enumerate}
For each $t\in [T]$, let $\cQ_t:= \{ Q_{t,1},\dots, Q_{t,\kappa/T}\}.$
We define $R:= R'[[r]]$ to be the induced subgraph of $R'$ on $[r]$. Note that each $Q_{t,s}\in \cQ$ can be viewed as a subgraph of $R$.
Moreover, for fixed $t\in [T]$, \ref{Q5} implies that the graphs $Q_{t,1},\dots, Q_{t,\kappa/T}$ are pairwise edge-disjoint when viewed as subgraphs of $R$.
Also, we have \vspace{-0.2cm}
\begin{eqnarray}\label{eq: R min deg}
\delta(R) \geq q^{-1}\delta(R'_{\rm multi}) - (r'-r) \stackrel{\eqref{eq: sum dij},\ref{Q1}}{\geq} (\delta - d^{1/2}) r.
\end{eqnarray}

We need to modify the sets $V'_i$ later to ensure that we obtain appropriate super-regular $Q_{t,s}$-blow-ups. 
For this, we need to move some `bad' vertices in $V'_i$ into $V'_0$. 
For each $i \in [r]$ and each $j\in N_{R}(i)$, we define 
\begin{eqnarray}\label{super1}
U_i(j):= \{ v\in V_i': d_{G',V'_{j}}(v) \neq (d_{i,j}\pm \epsilon^2)|V'_j|\} \enspace \text{and} \enspace
U_i': = \{v \in V_i': |\{j: v \in U_i(j)\}| > \epsilon r\}.
\end{eqnarray}
By Proposition~\ref{prop: super} and \ref{R6}, for any $i\in [r]$ and $j\in N_R(i)$ we have
\begin{eqnarray}\label{eq: Uij property}
|U_i(j)| \le 5\epsilon^2 n/r \qquad \text{ and }\qquad |U_i'| \leq (\epsilon r)^{-1} \sum_{j \in N_R(i)}|U_i(j)| \leq  5\epsilon n/r.
\end{eqnarray}
For each $i\in [r]$, we let 
$ V_i := V_i' \setminus U_i' \enspace \text{ and } \enspace V_0 := V_0' \cup \bigcup_{i=1}^r U_i' \cup \bigcup_{i=r+1}^{r'} V'_i.$

By \ref{R2} and \ref{R3}, for each $i \in [r]$, we have
\begin{align}\label{eq: Vi sizes}
(1 - 6\epsilon)n/r \le |V_i| \le n/r \qquad \text{ and } \qquad |V_0| \le 6\epsilon n.
\end{align}

We apply Lemma~\ref{Reservoir2} with $G', V(G)\backslash V_0, \{V_i\}_{i=1}^{r}$ and $T$ playing the roles of $G, V, \{V_i\}_{i=1}^{r}$ and $t$ to obtain a partition $\{Res_1,\dots, Res_{T}\}$ of $V(G) \backslash V_0$ satisfying the following, where we define $V_i^{t}:= V_i\cap Res_t$.
 \begin{enumerate}[label={\rm(Res\arabic*)}]
\item \label{Res1} For all $t\in [T]$ and $v\in V(G)$, we have
$d_{G', V_i^{t}}(v) = \frac{1}{T} d_{G',V_i}(v) \pm n^{2/3},$
\item \label{Res2} for all $t\in [T]$ and $i\in [r]$, we have
$|V_i^{t}| =  (\frac{1}{T} \pm \epsilon^2)|V_i| \stackrel{\eqref{eq: Vi sizes}}{=} \frac{(1\pm 7\epsilon) n}{Tr}, $
\item \label{Res3} for all $t\in [T]$, we have $|Res_t| \in \{  \lfloor \frac{n-|V_0|}{T}\rfloor ,  \lfloor \frac{n-|V_0|}{T}\rfloor+1  \}$.
\end{enumerate}

Next, we partition the edges in $G'\setminus V_0$ into $L_1,\dots, L_7$ which will be the building blocks for the graphs $G, F$ and $F'$ in Lemma~\ref{embedone}. 
Let $p_1:= 1-6d$ and $p_j:=d$ for $2\leq j\leq 7$. 
Apply Lemma~\ref{star} with $G'\setminus V_0$, $\{V_i^t: i\in [r], t\in [T]\}$, $\{(V_i, V_j) : ij \in E(R) \}$ and $7$ playing the roles of $G$, $\mathcal{U}$, $\mathcal{U}'$ and $s.$
Then we obtain a decomposition $L_1, \dots, L_7$ of $G'\backslash V_0 $ satisfying the following for all $t\in [T]$, $i\in [r]$, $\ell \in [7]$ and $v\in V(G)\setminus V_0$: 
\begin{enumerate}[label=\text{\rm (L\arabic*)}]
\item \label{L1}  $d_{L_\ell, V_i^t}(v)  = p_\ell  d_{G', V_i^t}(v) \pm n^{2/3},$
\item \label{L2}  for each $ij \in E(R)$, we have that $L_\ell[V_i, V_{j}]$ is $(4\epsilon^2, d_{i,j} p_\ell)$-regular.
\end{enumerate}
Let $G'':= L_1$.
For each $t \in [T]$, let $G^*_t$, $F_t$ and $F^*_t$ be the graphs on vertex set $V(G)\setminus V_0$ with 
\begin{align} 
E(G^*_t)&:= \bigcup_{t'=1}^{t-1} E(L_2[Res_{t},Res_{t'}]) \cup  \bigcup_{t'=t+1}^{T} E(L_3[Res_{t},Res_{t'}]) \cup  L_2[Res_{t}],\label{eq: G*t def}\\
E(F_t) &:= \bigcup_{t'=1}^{t-1} E(L_4[Res_{t},Res_{t'}]) \cup  \bigcup_{t'=t+1}^{T} E(L_5[Res_{t},Res_{t'}]), \nonumber \\
E(F^*_t) &:= \bigcup_{t'=1}^{t-1} E(L_6[Res_{t},Res_{t'}]) \cup  \bigcup_{t'=t+1}^{T} E(L_7[Res_{t},Res_{t'}]). \nonumber
\end{align}
For each $t\in [T]$, we let $F_{t,1},\dots, F_{t,\kappa/T}$ be subgraphs of $F_t$ such that for all $s\in [\kappa/T]$
\begin{align}\label{eq: Fts definition}
 F_{t,s} := \bigcup_{i\in V(Q_{t,s})} \bigcup_{j\in N_{Q_{t,s}}(i)} F_t[V_{i}^{t}, V_j\setminus Res_t].
 \end{align}
Note that \ref{Q5} implies that for $s\neq s'\in [\kappa/T]$, the graphs $F_{t,s}$ and $F_{t,s'}$ are edge-disjoint.
Thus $G'', G^*_1,\dots, G^*_{T}, F_{1,1},\dots, F_{T,\kappa/T}, F^*_1,\dots, F^*_T$ form edge-disjoint subgraphs of $G'\setminus V_0$.
The edges in $G^*_t$ will be used to satisfy condition {\rm(A4)}$_{\ref{embedone}}$ when applying Lemma~\ref{embedone}. The graphs $F_{t,s}$ will play the role of $F$ in Lemma~\ref{embedone}. The graphs $F^*_t$ will be
used in the construction of the graph $F'_t$, which will play the role of $F'$ in Lemma~\ref{embedone}.

We will now further partition the edges in $G''=L_1$.
Note that for each $ij\in E(R)$, by \eqref{eq: tij def} we have $q_{i,j}= \lfloor d_{i,j}p_1 q\rfloor.$
To further partition $G''$, we apply Lemma~\ref{star} for each $ij \in E(R)$ with the following objects and parameters.\newline

{\small
\noindent
{ 
\begin{tabular}{c|c|c|c|c|c|c}
object/parameter& $G''[V_i, V_{j}]$ & $\{V_i^t, V_{j}^{t}: t\in [T]\}$ & $ \{(V_i, V_{j}) \}$ & $q_{i,j}+1 $ &  $  1/(d_{i,j}p_1 q) $ &   $1 - q_{i,j}/(d_{i,j}p_1 q) $ \\ \hline
playing the role of &  $G$ & $\mathcal{U}$ &  $\mathcal{U}'$ & $s$  &$p_{i}: i<s$ & $p_s$ 
\end{tabular}
}}\newline \vspace{0.2cm}

Then by \ref{L2}, for each $ij\in E(R)$, we obtain edge-disjoint subgraphs $E^1_{i,j}, \dots, E^{q_{i,j}+1}_{i,j} $
of $G''[V_i,V_{j}]$ satisfying the following for all $t\in [T]$  and $\ell \in [q_{i,j}]$:
\begin{enumerate}[label=\text{\rm (E\arabic*)}]
\item \label{E1} for each $v\in V_i$, we have
 $d_{ E^\ell_{i,j}, V_{j}^t }(v)  = \frac{1}{d_{i,j}p_1 q} d_{G'', V_{j}^t}(v) \pm n^{2/3}$,
\item \label{E2} $E^{\ell}_{i,j}$ is $(8\epsilon^2, 1/q)$-regular.
\end{enumerate}
Recall that we have chosen a collection $\cQ=\{Q_{1,\kappa/T},\dots, Q_{T,\kappa/T}\}$ of edge-disjoint subgraphs of $R_{\rm multi}$ satisfying \ref{Q1}--\ref{Q5}. Let $\psi: E(R_{\rm multi}) \rightarrow \mathbb{N}$ be a function such that
$$\psi( E_{R_{\rm multi}}(i,j)) = [q_{i,j}].$$
For all $ij \in E(R')$, there are exactly $q_{i,j}$ edges between $i$ and $j$ in $R_{\rm multi}$, so such a function $\psi$ exists.
Now, for all $t\in [T]$, $s\in [\kappa/T]$, we let 
\begin{equation} \label{defGts}
G_{t,s}:= \bigcup_{ ij \in E(Q_{t,s}) } E^{\psi(ij)}_{i,j}.
\end{equation}
Since $\cQ$ is a collection of edge-disjoint subgraphs of $R_{\rm multi}$ and $E^1_{i,j}, \dots, E^{q_{i,j}+1}_{i,j} $ are edge-disjoint subgraphs of $G''$, the graphs
$G_{1,1},\dots, G_{T,\kappa/T}$ form edge-disjoint subgraphs of $G''$.

We would like to use $G_{t,s}\setminus Res_t$ and $Res_t$ to play the roles of $G[\bigcup_{i\in [r]} V_i]$ and $U$ in Lemma~\ref{embedone}, respectively. 
However, $E_{i,j}^{\ell} \setminus Res_t$ is not necessarily super-regular and the sizes of $V_i \setminus Res_t$ are not necessarily the same for all $i\in [r]$. 
To ensure this, we will now choose an appropriate subset $V_i^{t,s}$ of $V_i$ which can play the role of $V_i$ in Lemma~\ref{embedone}.

For all $t \in [T]$, $i \in [r]$ and $s \in [\kappa/T]$, let\COMMENT{if $i\in [r]\setminus V(Q_{t,s})$ then we view $N_{Q_{t,s}}(i)$ to be $\emptyset$. So $V_{i}(t,s)$ is still well-defined. (For the application of Lemma~\ref{lem: choice} below, it is more convenient to define $V_{i}(t,s)$ for all $i\in [r]$.}
\begin{eqnarray}\label{Vits}
V_{i}(t,s):= V_i \setminus (Res_t \cup \bigcup_{ j\in N_{ Q_{t,s}}(i)} U_i(j) )\enspace \text{ and } \enspace
m:= \frac{(T-1)n}{Tr} - \frac{10\epsilon n}{r}.\end{eqnarray}
Then by \eqref{eq: Uij property}, \eqref{eq: Vi sizes} and \ref{Res2}, we have
\begin{align}\label{eq: V m sizes}
0\leq |V_i(t,s)| - m \leq 15 \epsilon n/r.
\end{align}
For all $t\in [T]$ and $i\in [r]$, we apply Lemma~\ref{lem: choice} with the following objects and parameters.\newline

{\small
\noindent
{ 
\begin{tabular}{c|c|c|c|c|c|c|c|c}
object/parameter& $\kappa/T$ & $1$  & $s \in [\kappa/T]$ & $V_i\setminus Res_t$ & $ V_i(t,s) $ & $20 \epsilon$  & $ |V_i(t,s)|-m $ & $d$ \\ \hline
playing the role of  & $s$ & $r$ & $i \in [s]$ & $ A$ & $A_{i,1}$ & $\epsilon$  & $m_{i,1}$ & $1/2$ 
\end{tabular}
}}\newline \vspace{0.2cm}

Then we obtain sets $W_i(t,1),\dots, W_i(t,\kappa/T)$ such that 
$W_i(t,s) \subseteq V_i(t,s)$ with $|V_i(t,s)\setminus W_i(t,s)|=m$
and for any $v\in V_i\setminus Res_t$, we have 
\begin{align}\label{eq: v in Wi s ell}
|\{ s \in [\kappa/T] : v\in W_i(t,s)\}| \leq 10\epsilon^{1/2} \kappa/T.
\end{align}
For all $t\in [T]$, $s\in [\kappa/T]$ and $i\in V(Q_{t,s})$, let $V_i^{t,s}:= V_i(t,s)\setminus W_i(t,s).$ Let
\begin{eqnarray*}
V_0^{t,s}:= V_0\cup   \bigcup_{i\in [r]} \bigcup_{ j\in N_{Q_{t,s}}(i)}  \hspace{-0.2cm} ( U_i(j)\setminus Res_t) \cup \bigcup_{i\in [r]} W_i(t,s) \cup \hspace{-0.3cm} \bigcup_{ i\in [r]\setminus V(Q_{t,s})}  \hspace{-0.2cm} (V_i \setminus Res_t).
\end{eqnarray*}
Then the sets $ V_0^{t,s},\{V_i^{t,s}: i\in V(Q_{t,s})\}, Res_t$ form a partition of $V(G)$, and for each $i\in V(Q_{t,s})$
\begin{eqnarray}\label{eq: Vi size}
|V_i^{t,s}|&=& m := \frac{(T-1)n}{Tr} - \frac{10\epsilon n}{r}, \text{ and }  \\ 
|V_0^{t,s}| &\stackrel{\substack{\eqref{eq: Uij property},\\ \eqref{eq: Vi sizes}, \eqref{eq: V m sizes}}}{\leq}& 6\epsilon n + (k-1)r(5\epsilon^2n/r) + 15\epsilon n +(r- |V(Q_{t,s})|)n/r \nonumber\\\label{eq: V0 size}
&\stackrel{ \ref{Q3}}{ \leq}& 25 \epsilon n.   
\end{eqnarray}

We now further modify $V_i^t$ into $U_i^{t,s}$ which can play the role of $U_i$ in Lemma~\ref{embedone}.
For all $(t,s)\in [T]\times [\kappa/T]$ and $i\in V(Q_{t,s})$, we define
$$U_i^{t,s}:= V_i^{t}\setminus \bigcup_{j\in N_{Q_{t,s}}(i)} U_i(j) \enspace \text{and}\enspace
U_0^{t,s}:= \bigcup_{i\in [r]\setminus V(Q_{t,s})} V_i^{t} \cup \bigcup_{i\in V(Q_{t,s})} \bigcup_{j\in N_{Q_{t,s}}(i)} U_i(j).$$
Note that for each $(t,s)\in [T]\times [\kappa/T]$, the sets
$\{U_0^{t,s}\}\cup  \{ U_i^{t,s}: i\in V(Q_{t,s})\}$ form a partition of $Res_t$.
By \eqref{eq: Uij property}, for all $(t,s)\in [T]\times [\kappa/T]$ and $i\in V(Q_{t,s})$, we have 
\begin{eqnarray}\label{eq: Uits sizess}
|U_i^{t,s}| = |V_i^{t}| \pm 5k\epsilon^2 n/r \stackrel{\ref{Res2}}{=} 
\frac{(1\pm 8\epsilon)n}{Tr} \enspace \text{and}\enspace
|U_0^{t,s}| \stackrel{\eqref{eq: Uij property}}{\leq} \sum_{i\in [r]\setminus V(Q_{t,s})} |V_i^{t}| + 5k\epsilon^2 n \stackrel{\ref{Q3}}{\leq} 2\epsilon n.
\end{eqnarray}
 Note that 
for all $(t,s)\in [T]\times[\kappa/T]$ and $i\in V(Q_{t,s})$, we have $U_i^{t,s}, V_i^{t,s} \subseteq V_i$. 
Thus Proposition~\ref{prop: reg smaller} with \eqref{eq: Vi size}, \eqref{eq: Uits sizess}, \ref{L2} and the definition of $p_{\ell}$ implies that for all $(t,s)\in [T]\times [\kappa/T]$, $ij\in E(R[V(Q_{t,s})] )$ and $i'j'\in E(Q_{t,s})$, we have
\begin{equation}\label{eq: F* e-regular}
\begin{minipage}[c]{0.8\textwidth}
$G^*_t[ U_i^{t,s} , U_j^{t,s}], G^*_{t}[V_i^{t,s}, U_j^{t,s}]$ and $F_{t,s}[V_{i'}^{t,s}, U_{j'}^{t,s}]$ are $(\epsilon, (d^2))^+$-regular. 
\end{minipage}
\end{equation}
Moreover, for all $(t,s)\in [T]\times[\kappa/T]$, $ij\in E(Q_{t,s})$ and $u \in U_{i}^{t,s}$, we have
\begin{eqnarray}\label{eq: Ft degree v}
d_{F_{t,s},V_j^{t,s}}(u)  &\stackrel{\eqref{eq: Fts definition},\eqref{eq: Vi size},\ref{Res2}}{\geq}&
d_{F_{t}[V_i^{t}, V_j\backslash Res_t]}(u) - n/(Tr)
\stackrel{\ref{L1},\ref{Res1}}{\geq}  d\cdot d_{G',V_j}(u) -  3n/(Tr) \nonumber \\ 
&\stackrel{\eqref{super1}, \eqref{eq: Uij property}}{\geq }& d\cdot (d_{i,j}- \epsilon^2)|V_j| -   4n/(Tr)
\stackrel{\ref{Res2}}{\geq} (2d^2/3)|V_j\setminus Res_t|.
\end{eqnarray}\COMMENT{For second equality, use $d_{F_{t}[V_i^{t}, V_j\backslash Res_t]}(u) = \sum_{t' \in [T]\backslash\{t\}} d_{F_t, V_j^{t'}}(u) $}
We obtain the third inequality from the definition of $U_{i}^{t,s}$ and the fact that $ij\in E(Q_{t,s})$.

\begin{claim}\label{cl: Hisell super-regular}
For all $t \in [T], s\in [\kappa/T]$ and $ij \in E(Q_{t,s})$,
the graph $G_{t,s}[ V_i^{t,s}, V_{j}^{t,s}]$ is $(\epsilon^{1/2}, 1/q)$-super-regular. 
\end{claim}
\begin{proof}
Let $\ell\in [q_{i,j}]$ be such that $G_{t,s}[V_i, V_j] = E_{i,j}^{\ell}$.
Such an $\ell$ exists by the definition of $G_{t,s}$ and the assumption that $ij\in E(Q_{t,s})$. 
Note that for $i'\in \{i,j\}$ we have $V_{i'}^{t,s} \subseteq V_{i'}$ with $|V_{i'}^{t,s}|= m > \frac{1}{2}|V_{i'}|$ by \eqref{eq: Vi size}. 
Thus Proposition~\ref{prop: reg smaller} with \ref{E2} implies that $G_{t,s}[V_i^{t,s}, V_{j}^{t,s}] = E_{i,j}^{\ell}[V_i^{t,s}, V_{j}^{t,s}]$ is $(16\epsilon^2, 1/q)$-regular.

Consider $v\in V_i^{t,s}$. By the definition of $V_i^{t,s}$, we have $v\notin U_i(j)$. Thus
\begin{eqnarray*}
 \hspace{-0.4cm}d_{G_{t,s}, V_j^{t,s} } (v) \hspace{-0.2cm} &\stackrel{\eqref{eq: Uij property},\eqref{eq: V m sizes}}{=}&  \hspace{-0.2cm}d_{E_{i,j}^{\ell}, V_j\setminus Res_t}(v) \pm \frac{16\epsilon n}{r}
= \sum_{t'\in [T]\setminus \{t\}} d_{E_{i,j}^{\ell}, V_j^{t'}}(v) \pm \frac{16\epsilon n}{r} \\
 \hspace{-0.2cm}&\stackrel{\ref{E1}}{=}&  \hspace{-0.2cm}\sum_{t'\in [T]\setminus \{t\}} \frac{1}{d_{i,j}p_1 q  } d_{G'', V_{j}^{t'}}(v) \pm \frac{17\epsilon n}{r}
\stackrel{\ref{L1}}{=} \sum_{t'\in [T]\setminus \{t\}}\frac{1}{d_{i,j} q} d_{G', V_{j}^{t'}}(v) \pm \frac{18\epsilon n}{r} \\
 \hspace{-0.2cm}&\stackrel{\ref{Res1}}{=}&  \hspace{-0.2cm}\frac{(T-1)}{d_{i,j} q T} d_{G', V_{j}}(v) \pm \frac{19\epsilon n}{r} 
\stackrel{\eqref{super1}}{=}  \frac{(T-1)}{d_{i,j} q T} ( (d_{i,j}\pm \epsilon^2) |V'_j| \pm |U'_j|)  \pm \frac{19\epsilon n}{r}\\
 \hspace{-0.2cm} &\stackrel{\eqref{eq: Uij property}}{=}& \hspace{-0.2cm} \frac{(T-1) n }{q Tr}  \pm \frac{30\epsilon n}{r}  \stackrel{\eqref{eq: Vi size}}{=} (\frac{1}{q}\pm \epsilon^{1/2})|V_j^{t,s}|. 
\end{eqnarray*}
Similarly, for $v\in V_j^{t,s}$, we have 
$d_{G_{t,s}, V_i^{t,s} } (v) = (\frac{1}{q} \pm \epsilon^{1/2})|V_i^{t,s}|.$ 
Thus $G_{t,s}[ V_i^{t,s}, V_{j}^{t,s}]$ is $(\epsilon^{1/2}, 1/q)$-super-regular. 
This proves the claim.
\end{proof}


For all $t\in [T]$, $v\in Res_t$ and $s\in [\kappa/T]$, we know that 
\begin{eqnarray*}
d_{G^*_{t},V_i^{t,s}}(v) & = &d_{G^*_{t},V_i}(v)  \pm |V_i\setminus V_i^{t,s}|  \stackrel{\ref{L1}}{=} \sum_{\ell \in [T]} (d \cdot d_{G',V_i^\ell}(v)\pm n^{2/3}) \pm 
|V_i\setminus V_i^{t,s}| \nonumber \\
&\stackrel{\ref{Res1},\eqref{eq: Vi size}}{=}&  d \cdot d_{G',V_i}(v) \pm 2 n/(Tr).
\end{eqnarray*}
This implies that
\begin{align}\label{eq: v degree good in G*t}
&|\{ i \in V(Q_{t,s}) : d_{G^*_{t}, V_i^{t,s}}(v) \geq d^2 m/2\}|
\geq |\{ i \in V(Q_{t,s}) : d_{G',V_i}(v) \geq d |V_i| \}| \nonumber \\ 
&\geq \frac{d_{G'}(v) -|V_0| -dn }{\max_{i\in [r]} |V_i| }  - |[r]\setminus V(Q_{t,s})|
 \stackrel{\eqref{eq: Vi sizes}, \ref{Q3}}{\geq} (1- 1/k +\sigma/2)r.
\end{align}
We obtain the final inequality since $\delta(G') \geq (\delta -\xi- 2d)n \geq (1- 1/k + 3\sigma/4)n$ by (i) and \ref{R4}.
This together with \eqref{eq: F* e-regular} and Claim~\ref{cl: Hisell super-regular} will ensure that $G_{t,s}\cup G^*_{t}$ can play the role of $G$ in Lemma~\ref{embedone}, and \eqref{eq: Ft degree v} shows that $F_{t,s}$ can play the role of $F$ in Lemma~\ref{embedone}.

The remaining part of this step is to construct a graph which can play the role of $F'$ in Lemma~\ref{embedone}.
$F'$ needs to contain suitable stars centred at $v$ whenever $v \in V_0^{t,s}.$ (For each $t$, the number of
stars we will need for $v$ in order to deal with all $s\in [\kappa/T]$ is bounded from above
by (\ref{boundkappav}).)
 For all $t\in [T]$, $s\in [\kappa/T]$, $v\in V(G)$ and $u\in Res_t$, let 
\begin{eqnarray}\label{Itv}
I_t(v):=\{ s' \in [\kappa/T] : v\in V_0^{t,s'}\} \qquad \text{and} \qquad i_t^{s}(v):= |I_t(v)\cap [s]|, \nonumber \\
J_t(u):= \{ s'\in [\kappa/T] : u\in U_0^{t,s'}\} \qquad \text{and} \qquad j_t^{s}(u):= |J_t(u)\cap [s]|.
\end{eqnarray}
Note that if $v\in V_0$, then $I_t(v) = [\kappa/T]$.
If $v\in V_i\backslash Res_t$ for some $i \in [r]$, then $s \in I_t(v)$ means 
$v\in W_i(t,s) \cup \bigcup_{ j\in N_{Q_{t,s}}(i)} U_i(j) \cup \bigcup_{i' \in [r]\setminus V(Q_{t,s})} V_{i'}.$
Together with the fact that $U'_i \subseteq V_0$ and so $v \notin U'_i$, this implies\COMMENT{In both cases for $I_t(v)$ and $J_t(v)$, we get the term $ |\{j\in [r] : v\in U_i(j)  \}|$ since for each $ij\in E(R)$, there are only one $s\in [\kappa/T]$ with $ij \in E(Q_{t,s})$ by \ref{Q5}.  }
\begin{eqnarray}\label{eq: Lsv size}
 |I_t(v)| \hspace{-0.2cm}&\stackrel{\ref{Q5}}{\leq}&  \hspace{-0.2cm}|\{s \in [\kappa/T]: v\in W_i(t,s)\}| + |\{ j\in [r]: v\in U_i(j)\}| 
+ |\{s \in [\kappa/T] : i \notin V(Q_{t,s})  \}| \nonumber \\
 \hspace{-0.2cm}&\stackrel{\eqref{super1},\eqref{eq: v in Wi s ell},\ref{Q4}}{\leq}& \hspace{-0.2cm} 10\epsilon^{1/2} \kappa/T + \epsilon r + \epsilon r \stackrel{\ref{Q2}}{\leq} 20 \epsilon^{1/2}r.
\end{eqnarray}
Similarly, for $u\in V_i^t$, we have
\begin{eqnarray}\label{eq: Jsv size}
 |J_t(u)| \hspace{-0.2cm}&\leq&  \hspace{-0.2cm} |\{ j\in [r]: u\in U_i(j)\}| 
+ |\{s \in [\kappa/T] : i \notin V(Q_{t,s})  \}|\stackrel{\eqref{super1},\ref{Q4}}{\leq} \hspace{-0.2cm} \epsilon r + \epsilon r \leq 2 \epsilon r.
\end{eqnarray}
For each $v\in V(G)\setminus Res_t$, let
\begin{equation}\label{boundkappav}
\kappa_v:= \left\{\begin{array}{ll} (1+ d)\kappa & \text{ if } v \in V_0, \\ 
\lceil r/(2k) \rceil & \text{ if } v \notin V_0.\end{array}\right.
\end{equation}
$\kappa_v$ is the overall number of stars centred at $v$ that we will construct for given~$t$.
Note that for all $t\in [T]$ and $s\in [\kappa/T]$, no edge of $E(G'[V_0,Res_t])$ belongs to any of the graphs $G_{t,s}, G^*_t, F_t, F^*_t$. 
Now for each $t\in [T]$, we use these edges and edges in $F^*_t$ to construct stars $F'_t(v,s)$ centred at $v$, and subsets $C^t_{v,s}$, $C^{*,t}_{v,s}$ of $[r]$ for all $v\in V(G)\setminus Res_t$ and $s\in [\kappa_v]$, in such a way that the following hold for all $t\in [T]$ and $v\in V(G)\setminus Res_t$.
 \begin{enumerate}[label={\text{\rm(F$'$\arabic*)}}]
\item \label{F'1} For each $s\in [\kappa_v]$, we have $C^{t}_{v,s}\subseteq C^{*,t}_{v,s}$, $|C^{t}_{v,s}|=k-1$, $|C^{*,t}_{v,s}|=k$ and $R[C^{*,t}_{v,s}]\simeq K_{k}$,
\item \label{F'2} for each $i \in [r]$, we have $|\{ s\in [\kappa_v] : i\in C^{*,t}_{v,s}\}| \leq (k+1)q,$
\item \label{F'3} for each $s\in [\kappa_v]$, if $i\in C^{t}_{v,s}$, then $d_{F_t'(v,s), V_i^{t}}(v) \geq \frac{|V_i^{t}|}{q}$.
\end{enumerate}

\begin{claim}\label{cl: F' construction}
For all $t\in [T]$, $v\in V(G)\setminus Res_t$ and $s \in [\kappa_v]$, there exist edge-disjoint stars $F'_t(v,s) \subseteq G'[V_0,Res_t] \cup F_t^*$ centred at $v$, and subsets $C^t_{v,s}$, $C^{*,t}_{v,s}$ of $[r]$ which satisfy \ref{F'1}--\ref{F'3}.
\end{claim}
When applying Lemma \ref{embedone} in Step 3 to pack $\cH_{t,s}$, we will only make use of those stars $F_t'(v,s)$ with $v \in V_0^{t,s}$, but it is slightly more convenient to define them for all $v \in V(G)\backslash Res_t.$
\begin{proof}
First, consider $t\in [T]$ and $v\in V_0$. Then we have
\begin{eqnarray}\label{eq: G' Rest degree}
d_{G',Res_t}(v) = \sum_{i\in [r]} d_{G',V_i^{t}} (v) \stackrel{\ref{Res1}}{=}
 \frac{1}{T} \sum_{i\in [r]} d_{G',V_i}(v) \pm r n^{2/3} 
 \stackrel{ \substack{(i),\ref{R4},\\ \ref{Res3},\eqref{eq: Vi sizes}}}{=} (\delta  \pm 3d) |Res_t|. 
\end{eqnarray}
For all $v\in V_0$, $t \in [T]$ and $i\in [r]$, let  $q_{v,i}^t:=  \lfloor \frac{q\cdot d_{G',V_i^t}(v)  }{|V_i^t|}   \rfloor.$
Consider edge-disjoint subsets $E_{v,i}^{t}(1),\dots, E_{v,i}^{t}(q_{v,i}^t)$ of $E_{G'}(\{v\}, V_i^t)$ such that $|E_{v,i}^{t}(q')| = \frac{1}{q}|V_i^t|$ for each $q'\in [q_{v,i}^t]$.
Let $R^t_v$ be an auxiliary graph such that 
$$V(R^t_v) := \{ (i, q' ) : i\in [r], q' \in [q_{v,i}^t] \} \enspace \text{ and } \enspace E(R^t_v):= \{ (i,q')(j,q'') : ij\in E(R), q'\in [q_{v,i}^t], q''\in [q_{v,j}^t] \}.$$
Note that each $(i,q')$ corresponds to the star $E_{v,i}^{t}(q')$ centred at $v$.
We aim to find a collection of vertex-disjoint cliques of size $k-1$ in $R^t_v$, which will give us edge-disjoint stars in $E_{G'}(\{v\},Res_t)$.
From the definition, we have 
\begin{eqnarray}\label{eq: Rv size}
|V(R^t_v)| = \sum_{i\in [r]} q_{v,i}^{t} \hspace{-0.07cm} \stackrel{\ref{Res2}}{=} \hspace{-0.07cm}  \frac{(1\pm 10\epsilon) d_{G',Res_t}(v) q }{ n/(Tr)} \pm r  \stackrel{\eqref{eq: G' Rest degree}}{=} \frac{(\delta \pm 4d) q |Res_t|}{n/(Tr)}  \stackrel{\ref{Res3}}{=}  (\delta \pm 5d)q r.
\end{eqnarray}
Then, for $(i,q') \in V(R^t_v)$, we have
\begin{eqnarray}
\hspace{-0.2cm} d_{R^t_v}(( i,q'))  \hspace{-0.2cm} &\ge&\hspace{-0.2cm} q \sum_{ j\in N_{R}(i) } d_{G',V_j^t}(v) |V_j^{t}|^{-1}  - d_{R}(i) \stackrel{\ref{Res2}}{\geq }  \frac{Tqr}{(1+7\epsilon)n} \sum_{ j\in N_{R}(i) } d_{G',V_j^t}(v) - r \nonumber \\
\hspace{-0.2cm}&\geq& \hspace{-0.2cm}\frac{Tqr}{(1+7\epsilon) n}\Big(
 \sum_{j\in N_R(i)}|V_j^t| - \sum_{j\in [r]}(|V_j^t| - d_{G',V^t_j}(v)) \Big) - r \nonumber \\
\hspace{-0.2cm}&\stackrel{\substack{\eqref{eq: R min deg}, \eqref{eq: G' Rest degree},\\ \ref{Res2}, \ref{Res3}}}{\ge}&\hspace{-0.2cm} (  2\delta - 2d^{1/2} -1 )qr  -r \stackrel{\eqref{eq: Rv size}}{\geq } ( 1- \frac{1}{k-1} + \sigma ) |V(R_v^t)|.
\end{eqnarray}\COMMENT{To go from line 2 to line 3 above: $\frac{Tqr}{(1\pm 7\epsilon)n}\Big( (\delta - d^{1/2})\frac{(1 \pm 7\epsilon)n}{T} + \frac{(1\pm 6\epsilon)n}{T}\Big((\delta \pm 3d) - 1 \Big)\Big).$}
Here, the final inequality follows from \eqref{eq: deltak 1-1/k}.
By the Hajnal-Szemer\'edi theorem, $R^t_v$ contains at least 
$$|V(R^t_v)|/(k-1) -1 \stackrel{\eqref{eq: Rv size}}{\geq } (\delta - 5d)qr/(k-1) - 1 \stackrel{\ref{Q2}}{\geq} (1+d)\kappa \stackrel{(\ref{boundkappav})}{=} \kappa_v$$
 vertex-disjoint copies of $K_{k-1}$.
Let $C^{t}_{v}(1), \dots, C^{t}_v(\kappa_v)$ be such vertex-disjoint copies of $K_{k-1}$ in $R^t_v$. 
For each 
$s\in [\kappa_v]$, we let 
$$F'_t(v,s):= \bigcup_{(i,q') \in V(C^{t}_{v}(s))} E_{v,i}^{t}(q')\qquad \text{ and } \qquad
C^{t}_{v,s} := \{ i : (i,q') \in V(C^{t}_v(s) ) \text{ for some }q'\in [q^t_{v,i}] \}.
$$
By construction $|C^t_{v,s}| = k-1$ and $R[C^t_{v,s}] \simeq K_{k-1}$. Moreover, the maximum degree of the multi-($k-1$)-graph $\{C^{t}_{v,s} : s\in [\kappa_v]\}$ is at most $q$.
Thus we can apply Lemma~\ref{lem: k-1 clique extend} with $\{C^{t}_{v,s} : s\in [\kappa_v]\}$, $R$, $q$ and $k$ playing the roles of $\cF, R, q$ and $k$. 
 Then we obtain sets $C^{*,t}_{v,s}$ satisfying the following for all $s\in [\kappa_v]$ and $i\in [r]$:
\begin{align}\label{eq: max degree C*}
C^{t}_{v,s}\subseteq C^{*,t}_{v,s}, \enspace R[C^{*,t}_{v,s}]\simeq K_k, \enspace \text{and} \enspace |\{ s\in [\kappa_v] : i\in C^{*,t}_{v,s}\}|\leq (k+1)q.
\end{align}
It is easy to see that for all $s \in [\kappa_v]$ the sets $C_{v,s}^t$, $C_{v,s}^{*,t}$ and the stars $F'_t(v,s)$ satisfy \ref{F'1}--\ref{F'3}.

Now, we consider $t\in [T]$ and $v\in V_i\setminus Res_t$ with $i\in [r]$.
Let $S^t_{v}:=N_R(i) \setminus \{ j : v\in U_i(j)\}$, and 
for each $j\in S^{t}_{v}$, let $E_{v,j}^{t}$ be a subset of $E_{F^*_t}(\{v\},V_j^{t})$ with $|E_{v,j}^{t}| = \frac{1}{q}|V_j^{t}|$.
We can choose such a star as there exists $\ell\in \{6,7\}$ such that
$$d_{F^*_t,V_j^{t}}(v) = d_{L_{\ell},V_j^{t}}(v) \stackrel{\ref{L1}}{=} d\cdot d_{G',V_j^t}(v) \pm n^{2/3} \stackrel{\ref{Res1},\ref{Res2}}{=}  (1\pm 10\epsilon) d \cdot d_{i,j}|V_j^{t}| > \frac{1}{q}|V_j^{t}|.$$
Here, the third equality follows since $v\notin U_i(j)$.
By \eqref{eq: R min deg}, \eqref{super1} and the fact that $v\notin U'_i$, we have $|S^{t}_{v}|\geq (\delta - 2d^{1/2})r$.
Thus 
$$\delta( R[S^t_{v}] )\geq |S^{t}_{v}| - (r- \delta(R)) \stackrel{\eqref{eq: R min deg}}{\geq} ( 1- \frac{1}{k-1})|S^{t}_{v}|.$$
Again, by the Hajnal-Szemer\'edi theorem, $R[S^{t}_{v}]$ contains (at least)  
$\kappa_v =\lceil r/(2k) \rceil$ vertex-disjoint copies of $K_{k-1}$. 
Denote their vertex sets by $C^{t}_{v,1}, \dots, C^{t}_{v,\kappa_v}$. 
We apply Lemma~\ref{lem: k-1 clique extend} with $\{C^{t}_{v,s} : s\in [\kappa_v]\}$, $R$, $1$ and $k$ playing the roles of $\cF, R, q$ and $k$ respectively, to extend each $C^{t}_{v,s}$ into a $C^{*,t}_{v,s}$ with $R[C^{*,t}_{v,s}]\simeq K_k$ and such that $|\{s\in [\kappa_v] : i\in C^{*,t}_{v,s}\}| \leq k+1$ for each $i\in [r]$.
For each $s\in [\kappa_v]$,
let $F'_t(v,s):= \bigcup_{ j\in C^{t}_{v,s}} E_{v,j}^{t}.$
Again, it is easy to see that for all $s \in [\kappa_v]$ the sets $C_{v,s}^t$, $C_{v,s}^{*,t}$ and the stars $F'_t(v,s)$ satisfy \ref{F'1}--\ref{F'3}.
This proves the claim.
\end{proof}

Altogether we will apply Lemma \ref{embedone} $\kappa$ times in Step 3.
In each application, we want the leaves of the stars that we use to be evenly distributed (see condition~{\rm (A8)}$_{\ref{embedone}}$).
This will be ensured by Claim $13$.
More precisely, for each $v\in V(G)\setminus Res_t$, our aim is to  choose a permutation
$\pi^t_v : [\kappa_v] \rightarrow [\kappa_v]$ satisfying the following. 
\begin{enumerate}[label={\text{\rm(F$'$\arabic*)}}]
 \setcounter{enumi}{3}
\item\label{F'4} For all $t\in [T]$, $i\in [r]$ and $s\in [\kappa/T]$, we have $C(t,s,i)\leq \epsilon^{4/5}n/r$, where
$C(t,s,i):= |\{ v \in V_0^{t,s} : i \in C^{*,t}_{v,\pi^t_v(s')} \text{ for some $s'$ with } 
(i^s_t(v)-1)T +1 \leq s' \leq i^{s}_t(v) T\}|,$
\item\label{F'5} for all $t\in [T]$, $s\in [\kappa/T]$ and $t'\in [T]$, we have that \newline
$\bigcup_{v\in V_0^{t,s} }C^{*,t}_{v,\pi^t_v((i^s_t(v)-1)T+t')} \subseteq V(Q_{t,s})$.\COMMENT{If $v\in V_0^{t,\ell}$, then $i^{\ell}_t(v) = i^{\ell-1}_{t}(v)+1$, so $\{(i^{\ell}_t(v)-1)T+t': t'\in [T]\} = [i^{\ell}_{t}(v)T]\setminus [i^{\ell-1}_{t}(V)T]$.
}
\end{enumerate}
Recall from \eqref{Itv} that $i^s_t(v)$ counts the number of $s'\in [s]$ for which $v \in V_0^{t,s'}$. 
The number $C(t,s,i)$ is well-defined because $i^{s}_t(v) \leq \kappa_v/T$ for all $v\in V(G)\setminus Res_t$ by \eqref{eq: Lsv size}.\COMMENT{\eqref{eq: Lsv size} is only about the case when $v\notin V_0$. But if $v\in V_0$, then $i^s_t(v) \leq \kappa/T < (1+d)\kappa/T = \kappa_v/T$. }

\begin{claim}
For each $t\in [T]$ and each $v\in V(G)\setminus Res_t$,
there exists a permutation $\pi^t_{v} : [\kappa_v]\rightarrow [\kappa_v]$ satisfying \ref{F'4}--\ref{F'5}.
\end{claim}
\begin{proof}
We fix $t\in [T]$. We claim that for each $s\in [\kappa/T]\cup \{0\}$ the following hold.
For each $v\in V(G)\setminus Res_t$, there exists an injective map $\pi^t_{v,s}: [i_t^{s}(v) T] \rightarrow [\kappa_v]$ satisfying the following.
\begin{enumerate}
\item[(F$'$4)$_s^t$] 
For all $i\in [r]$ and $\ell \in [s]$, we have 
$$ |\{ v \in V_0^{t,\ell} : i \in C^{*,t}_{v,\pi^t_{v,s} (s')} \text{ for some $s'$ with } 
(i^\ell_t(v)-1)T +1 \leq s' \leq i^{\ell}_t(v) T\}| \leq \epsilon^{4/5}n/r,$$

\item[(F$'$5)$_s^t$] for all $\ell \in [s]$ and $t'\in [T]$, we have that 
$\bigcup_{v\in V_0^{t,\ell} }C^{*,t}_{v,\pi^t_{v,s}((i^\ell_t(v)-1)T+t')} \subseteq V(Q_{t,\ell})$.
\end{enumerate}
Note that both (F$'$4)$_0^t$ and (F$'$5)$_0^t$ hold by letting $\pi^t_{v,0}:\emptyset\rightarrow \emptyset$ be the empty map for all $v\in V(G)\setminus Res_t$.
Assume that for some $s\in [\kappa/T-1]\cup \{0\}$ we have already constructed injective maps $\pi^t_{v,s}$ for all $v\in V(G)\setminus Res_t$ which satisfy (F$'$4)$_s^t$ and (F$'$5)$_s^t$.
For each $v \in V_0^{t,s+1}$, we consider the set 
$$A_v:= \{ s' \in  [\kappa_v]\setminus \pi^t_{v,s}( [i_t^{s}(v)T]) : C^{*,t}_{v,s'}\subseteq V(Q_{t,s+1}) \}.$$
Then we have
\begin{eqnarray}\label{eq: AvAv sizes}
|A_v| &\stackrel{\ref{F'2}}{\geq}& \kappa_v - i_t^s(v)T - (k+1)q (r- |V(Q_{t,s+1})| ) \nonumber\\
&\stackrel{\eqref{eq: Lsv size},\ref{Q3}}{\geq}&  \min\{ d\cdot \kappa , r/(2k) - 20T\epsilon^{1/2} r\} -(k+1)q\epsilon r \geq r/(4k).
\end{eqnarray}
We choose a subset $I_v\subseteq A_v$ of size $T$ uniformly at random.
Then \ref{F'2} implies that for each $i\in V(Q_{t,s+1})$ we have
$$\mathbb{P}[ i\in \bigcup_{s'\in I_v} C^{*,t}_{v,s'} ] \leq (k+1)qT/|A_v| \leq 10qk^2T/r.$$
Thus 
$$\mathbb{E}[ |\{ v\in V_0^{t,s+1} : i\in \bigcup_{s'\in I_v} C^{*,t}_{v,s'} \}| ]
\leq 10qk^2T|V_0^{t,s+1}|/r \stackrel{\eqref{eq: V0 size}}{\leq} \epsilon^{4/5} n/(2r).$$ 
A Chernoff bound (Lemma~\ref{lem: chernoff}) gives us that for each $i\in V(Q_{t,s+1})$
$$\mathbb{P}\Big[ |\{ v\in V_0^{t,s+1} : i\in \bigcup_{s'\in I_v} C^{*,t}_{v,s'}\}|  
\geq \epsilon^{4/5}n/r  \Big] \leq \exp( -\frac{(\epsilon^{4/5} n/(2r))^2}{ 2|V_0^{t,s+1}| } )\stackrel{\eqref{eq: V0 size}}{\leq} e^{-n/r^3}.
$$ 
Since $1- |V(Q_{t,s+1})| e^{-n/r^3}>0$, 
 the union bound implies that there exists a choice of $I_v$ for each $v\in V_0^{t,s+1}$ such that for all $i\in V(Q_{t,s+1})$, we have that
\begin{align}\label{eq: Cts Iv}
 |\{ v\in V_0^{t,s+1} : i\in \bigcup_{s'\in I_v} C^{*,t}_{v,s'}\}|   \leq \epsilon^{4/5}n/r.\end{align}
If $v\in V(G)\setminus (Res_t \cup V_0^{t,s+1})$ (and thus $i^{s+1}_t(v)=i^s_t(v)),$ we let $\pi^t_{v,s+1} := \pi^{t}_{v,s}$.
For each $v\in V_0^{t,s+1}$, we extend $\pi^t_{v,s}$ into $\pi^t_{v,s+1}$ by defining $\pi^t_{v,s+1}: [i_t^{s+1}(v)T]\setminus [i_t^{s}(v)T] \rightarrow I_v$ in an arbitrary injective way.
Then, by the choice of $I_v$, we have that $\pi^t_{v,s+1}$ is an injective map from $[i_t^{s+1}(v)T]$ to $[\kappa_v]$ satisfying (F$'$5)$_{s+1}^t$. Moreover, \eqref{eq: Cts Iv} implies that for any $i\in V(Q_{t,s+1})$, we have 
\begin{align*}
 &|\{ v \in V_0^{t,s+1} : i \in C^{*,t}_{v,\pi^t_{v,s+1} (s')} \text{ for some $s'$ with }
(i^{s+1}_t(v)-1)T +1 \leq s' \leq i^{s+1}_t(v) T\}| \\
 &=  |\{ v\in V_0^{t,s+1} : i\in \bigcup_{s'\in I_v} C^{*,t}_{v,s'}\}|  \stackrel{\eqref{eq: Cts Iv}}{\leq } \epsilon^{4/5}n/r.
\end{align*}
This with (F$'$4)$_{s}^t$ implies (F$'$4)$_{s+1}^t$.
By repeating this, we obtain injective maps $\pi^t_{v,\kappa/T}$ satisfying both (F$'$4)$_{\kappa/T}^t$ and (F$'$5)$_{\kappa/T}^t$.
For each $v\in V(G)\setminus Res_t$, we extend $\pi^t_{v,\kappa/T}$ into a permutation $\pi^t_{v} : [\kappa_v] \rightarrow [\kappa_v]$ by assigning arbitrary values for the remaining values in the domain.
It is easy to see that (F$'$4)$_{\kappa/T}^t$ implies \ref{F'4} and (F$'$5)$_{\kappa/T}^t$ implies \ref{F'5}. We can find such permutations for all $t\in [T]$. Thus such collection satisfies both \ref{F'4} and \ref{F'5}.
\end{proof}

For each $t\in [T]$, let
$$G_t:= G^*_t \cup \bigcup_{s\in [\kappa/T]} G_{t,s}
  \enspace \text{and}\enspace
F'_t:= \bigcup_{v\in V(G)\setminus Res_t} \bigcup_{s\in [\kappa_v] }F'_t(v,s).$$
Then $G_1,\dots, G_{T}, F_{1},\dots, F_{T}, F'_1,\dots, F'_{T}$ form edge-disjoint subgraphs of $G$. 
(Recall that $G_t^*$ was defined in \eqref{eq: G*t def},
$G_{t,s}$ in~\eqref{defGts} and $F'_v(t,s)$ in Claim~\ref{cl: F' construction}.)
\newline

\noindent {\bf Step 2. Partitioning $\cH$.}
Now we will partition $\cH$.
Recall that the graphs in $\cH$ are $\eta^2$-separable. By packing several graphs from $\cH$ with less than $n/4$ edges suitably into a single graph in a way that no edges from distinct graphs intersect each other, we can assume that all but at most one graph in $\cH$ have at least $n/4$ edges, and that all graphs in $\cH$ are $(k,\eta)$-chromatic, $\eta$-separable and have maximum degree at most $\Delta$.%
\COMMENT{We can do this greedily. Choose graphs $H_1,\dots,H_a$ in $\cH$ such that $e(H_i)<n/4$ and $\sum e(H_i)<n/2$, and such that subject to these conditions $a$ is maximal. Let $H'_i$ be obtained from $H_i$ by deleting all its isolated vertices. Thus $\sum |H'_i|\le \sum 2e(H_i) <n$. Let $H$ be the $n$-vertex graph obtained from the vertex disjoint union of all the $H'_i$ by adding $n-\sum |H'_i|$ isolated vertices. Remove $H_1,\dots,H_a$ from $\cH$ and add $H$ instead. We claim that $H$ is $\eta$-separable. To see this, note that at most $1/\eta$ of the $H'_i$ have size at least $\eta n$, and each of these has a separator 
of size at most $\eta^2 n$. So altogether we get a separator of size $\eta n$. Moreover, $H$ is  $(k,\eta)$-chromatic. Indeed, let $H'$ be obtained from $H$ by deleting all its isolated vertices.  Each $H'_i$ has a $(k+1)$-colouring with the smallest colour class of size at most $\eta |H'_i|$, so we get a $(k+1)$-colouring of $H'$ with the smallest colour class of size $\sum \eta |H'_i|=\eta |H'|$. Now by the maximality of $a$ either $e(H)\ge n/4$ and we can continue with the process, or all remaining graphs in 
$\cH$ have at least $n/4$ edges.}
By adding at most $n/4$ edges to at most one graph if necessary, we can then assume that all graphs in $\cH$ have at least $n/4$ edges.
Moreover, if $e(\cH)$ is too small, we can add some copies of $n$-vertex paths 
to $\cH$ to assume that 
\begin{align*}
\epsilon n^2 \leq e(\cH)\stackrel{(iii)}{\leq} (1-\nu) e(G) + n/4.
\end{align*}
We partition $\mathcal{H}$ into $\kappa$ collections $\mathcal{H}_{1,1},\dots, \mathcal{H}_{T,\kappa/T}$ such that for all $t \in [T]$ and $s \in [\kappa/T]$, we have 
\begin{align}\label{size1}
n^{7/4} \stackrel{\ref{Q2}}{\leq} \frac{\epsilon n^2}{\kappa} -\Delta n \leq e(\mathcal{H}_{t,s}) < \frac{1}{\kappa}(1-\nu)e(G) + 2\Delta n \stackrel{(i), \ref{Q2}}{\leq} \frac{(1-2\nu/3) (k-1) n^2}{ 2 qr}.
\end{align}
Indeed, this is possible since $e(H)\leq \Delta n$ for all $H\in \cH$. Now, we are ready to construct the desired packing. \newline

\noindent {\bf Step 3. Construction of packings into the host graphs.} 
As $G_1,\dots, G_{T}, F_1,\dots, F_{T}, F'_1,\dots, F'_{T}$ are edge-disjoint subgraphs of $G$, and  $\mathcal{H}_{1,1},\dots, \mathcal{H}_{T,\kappa/T}$ is a partition of $\mathcal{H}$, 
it suffices to show that for each $t\in [T]$, we can pack 
$\mathcal{H}_t:= \bigcup_{s=1}^{\kappa/T}\mathcal{H}_{t,s}$ into $G_t\cup\bigcup_{s\in [\kappa/T]} F_{t,s}\cup F'_{t}$. (Recall from \eqref{eq: Fts definition} that $F_{t,1},\dots, F_{t,\kappa/T}$ are edge-disjoint subgraphs of $F_t$.)
We fix $t\in [T]$ and will apply Lemma~\ref{embedone} $\kappa/T$ times to show that such a packing exists.

Assume that for some $s$ with $0\leq s\leq \kappa/T-1$, we have already defined a function $\phi_{s}$ packing $\bigcup_{s'=1}^{s} \mathcal{H}_{t,s'}$ into $G_t\cup F_t\cup F'_t$ and satisfying the following,
where $\Phi^{s} :=  \bigcup_{s'=1}^{s} \phi_{s}(\mathcal{H}_{t,s'})$ and $j_t^s(u)$ is defined in \eqref{Itv} and $G^*_t$ is defined in \eqref{eq: G*t def}.
\begin{enumerate}[label={\text{\rm(G\arabic*)$_s$}}]
\item \label{G1} For each $u\in Res_t$, we have 
$d_{\Phi^{s}\cap G^*_t}(u) \leq \frac{4 k \Delta j_t^s(u) n}{qr} + \frac{\epsilon^{1/9} s n}{r}$, 
\item\label{G2} for each $i\in [r]$, we have 
$e_{\Phi^{s}\cap G^*_t}(V_i\backslash V_i^t, Res_t) \leq \frac{\epsilon^{1/3} s n^2}{r^2}$,

\item\label{G3} for $s'\in [\kappa/T]\setminus [s]$, we have $E(\Phi^{s})\cap (E(G_{t,s'})\cup E(F_{t,s'})) =\emptyset$,
\item \label{G4} for $v\in V(G)\setminus Res_t$, $s''\in [\kappa_v]$ with $s''>i_t^s(v)\cdot T$, we have
$E(\Phi^{s}) \cap F'_t(v,\pi_v^t(s'')) =\emptyset$.
\end{enumerate}
Note that (G1)$_{0}$--(G4)$_{0}$ trivially hold with an empty packing $\phi_0: \emptyset\rightarrow \emptyset$.
For each $t' \in [T]$ and $v\in V(G)\setminus Res_t$, let $\ell(v,t'):= \pi^t_v(  (i_t^{s+1}(v)-1)T+ t')$.
(Note that $\ell(v, t')$ is well-defined since $(i_t^{s+1}(v)-1)T + t' \le \kappa_v$ by \eqref{eq: Lsv size}.) 
Let
\begin{align}\label{eq: def V*G*}
V &:= \bigcup_{i\in V(Q_{t,s+1}) } V_{i}^{t,s+1}, \enspace \enspace U:= \bigcup_{i\in V(Q_{t,s+1})} U_{i}^{t,s+1}, \\
\hat{G}&:= (G_{t,s+1}[V]\cup G^*_t[V\cup Res_t]) \setminus E(\Phi^{s}), \enspace \text{and} \enspace
\hat{F'}:=\hspace{-0.3cm} \bigcup_{v\in V_0^{t,s+1}}\bigcup_{t'\in [T]}F'_{t}(v, \ell(v,t'))[\{v\},U].
\label{eq: Def G hat F hat}
\end{align}
Note that \ref{G3} implies that $E(F_{t,s+1}) \cap E(\Phi^{s}) =\emptyset$.
Let $\hat{R}$ be the graph on vertex set $V(Q_{t,s+1})$ with 
$$E(\hat{R}):= \{ ij \in E(R[V(Q_{t,s+1})] ): |E_{G^*_t} (V_i, V_j) \cap E(\Phi^{s})| < \epsilon^{1/10} n^2 /r^2 \}.$$

We wish to apply Lemma~\ref{embedone} with the following objects and parameters.\newline

{\small
\noindent
{ 
\begin{tabular}{c|c|c|c|c|c|c|c|c}
object/parameter& $\hat{G}$ & $F_{t,s+1}[V,U]$ & $\hat{F'}$ & $V^{t,s+1}_0$ & $U^{t,s+1}_0$ & $V^{t,s+1}_i$ & $U_i^{t,s+1}$ & $\hat{R}$ 
\\ \hline
playing the role of & $G$ & $F$ & $F'$ & $V_0$ & $U_0$ & $V_i$ & $U_i$ & $R$  \\ \hline \hline
object/parameter  & $1/q$ & $\cH_{t,s+1}$ & $d$ & $C^{*,t}_{v,\ell(v,t')}$ & $C^t_{v,\ell(v,t')}$ & $F'_{t}(v,\ell(v,t'))[\{v\},U]$  & $k$ & $\Delta$ 
\\ \hline
playing the role of & $\alpha$ & $\cH$  & $d$  & $C^*_{v,t}$ & $C_{v,t}$ & $F'_{v,t}$ & $k$ & $\Delta$ \\ \hline \hline
object/parameter & $Q_{t,s+1}$ & $\eta$ & $25 \epsilon$ & $\sigma/2$ &  $T$ & $\nu/2$ & $m$
\\ \hline
playing the role of & $Q$ & $\eta$ & $\epsilon$ & $\sigma$ & $T$& $\nu$  & $n'$
\end{tabular}
}}\newline \vspace{0.2cm}

Thus $Res_t\setminus U_0^{t,s+1}$ plays the role of $U = \bigcup_{i=1}^r U_i$ in Lemma \ref{embedone}, and $t' \in [T]$ stands for $t \in [T].$
By \eqref{eq: hierarchy}, \eqref{eq: Vi size}, \eqref{eq: V0 size},  \eqref{eq: Uits sizess}, \ref{Q3} and \ref{F'5} 
we have appropriate objects and parameters as well as the hierarchy of constants required in Lemma~\ref{embedone}.    
 Now we show that (A1)$_{\ref{embedone}}$--(A9)$_{\ref{embedone}}$ hold. 
(A1)$_{\ref{embedone}}$ is obvious from Theorem~\ref{thm:main}~(ii) and our assumption in Step 2.
(A2)$_{\ref{embedone}}$ holds by \eqref{size1}.
(A3)$_{\ref{embedone}}$ follows from  Claim~\ref{cl: Hisell super-regular} and \ref{G3}.\COMMENT{$G^*_t[V \cup Res_t]$ plays no role for (A3)$_{\ref{embedone}}$ because every edge of $G^*_t$ has $\ge 1$ endpoint in $Res_t$ and so $e_{G^*_t}(V^{t, s+1}_i, V^{t, s+1}_j) = 0$}
Consider $ij \in E(\hat{R})$, then 
$\hat{G}[U^{t,s+1}_i, U^{t,s+1}_j] = G^*_t[U^{t,s+1}_i,U^{t,s+1}_j] \setminus E(\Phi^s).$
Since $U^{t,s+1}_i \subseteq V_i $ and $U^{t,s+1}_j\subseteq V_j $,
the properties \eqref{eq: Uits sizess}, \eqref{eq: F* e-regular} and the definition of $\hat{R}$ imply that $$e_{G^*_{t}}(U^{t,s+1}_i,U^{t,s+1}_j) - e_{\Phi^s\cap G_t^*}(V_i, V_j)
 \geq  (1- \epsilon^{1/15}) e_{G^*_{t}}(U^{t,s+1}_i,U^{t,s+1}_j).$$
Thus, Proposition~\ref{prop: reg subgraph} with \eqref{eq: F* e-regular}  implies that $\hat{G}[U^{t,s+1}_i, U^{t,s+1}_j]$ is $(\epsilon^{1/50},(d^2))^+$-regular.
The calculation for $\hat{G}[V_i^{t,s+1}, U_j^{t,s+1}]$ is similar.\COMMENT{Similarly, since $V^{t,s+1}_i \subseteq V_i $,
by \eqref{eq: Vi size}, \eqref{eq: F* e-regular},  \ref{Res2}, and \ref{Q2}, we have 
$$e_{G^*_{t}}(V^{t,s+1}_i,U^{t,s+1}_j) - e_{\Phi^s\cap G_t^*}(V^{t,s+1}_i, U^{t,s+1}_j)
\geq e_{G^*_{t}}(V^{t,s+1}_i,U^{t,s+1}_j)- \epsilon^{1/10} n^2/r^2 
 \geq  (1- \epsilon^{1/15}) e_{G^*_{t}}(V^{t, s+1}_i,U^{t,s+1}_j).$$
 Thus by Proposition~\ref{prop: reg subgraph} with \eqref{eq: F* e-regular}, the graph $\hat{G}[V^{t,s+1}_i,U^{t,s+1}_j ] = G^*_{t}[ V^{t,s+1}_i,U^{t,s+1}_j]\setminus E(\Phi^s)$ is $(\epsilon^{1/50},(d^2))^+$-regular.}
Thus (A4)$_{\ref{embedone}}$ holds with the above objects and parameters. 
By \ref{G1}, for each $i \in [r]$ we have \begin{eqnarray}\label{edge calc}
e_{\phi^s \cap G^*_t}(V_i^t, \bigcup_{j \in [r]\backslash \{i\}} V_j) \le \sum_{v \in V_i^t} \Big(\frac{4k\Delta j^s_t(v)n}{qr} + \frac{\epsilon^{1/9}sn}{r}\Big) \stackrel{\ref{Q2}, \eqref{eq: Jsv size}, \ref{Res2}}{\le} \frac{\epsilon^{1/9}n^2}{r}. \end{eqnarray} 
Thus, for $i\in V(Q_{t,s+1}) =V(\hat{R})$, we have
\begin{align*}
d_{R}(i) - d_{\hat{R}}(i) &\leq \frac{ e_{\Phi^s\cap G^*_t}(V_i\backslash V_i^t, Res_t ) +  e_{\Phi^s\cap G^*_t}(V_i^t, \bigcup_{j \in [r]\backslash\{i\}} V_j)}{\epsilon^{1/10} n^2/r^2} + |V(R)\setminus V(\hat{R})|\\
 &\stackrel{\ref{G2},\ref{Q3}, \eqref{edge calc}}{\leq } 
\frac{ \epsilon^{1/3}s n^2/r^2 + \epsilon^{1/9}n^2/r}{\epsilon^{1/10}n^2/r^2} + \epsilon r   \stackrel{\ref{Q2}}{\leq} \epsilon^{1/100} r.
\end{align*}
This with \eqref{eq: R min deg} and \eqref{eq: deltak 1-1/k} implies that (A5)$_{\ref{embedone}}$ holds for $\hat{R}$.
For all $ij\in E(Q_{t,s+1})$ and $u\in U_i^{t,s+1}$, by \eqref{eq: Ft degree v},
we have 
$$d_{F_{t,s+1},V_j^{t,s+1}}(u) \geq 2d^2  |V_j\setminus Res_t|/3 \stackrel{\ref{Res2},\eqref{eq: Vi size}}{\geq} d^3 m .$$
Thus (A6)$_{\ref{embedone}}$ holds.
By \ref{F'1}, \ref{F'4} and the fact that $i_t^{s+1}(v) = i_t^{s}(v)+1$ for all $v \in V_0^{t,s}$, (A8)$_{\ref{embedone}}$ holds (for $C^{*,t}_{v,\ell(v,t')}$, $C^t_{v,\ell(v,t')}$ and all $v\in V_0^{t,s}$).
If $v \in V_0^{t,s}, t' \in [T]$ and $i\in C^t_{v, \ell(v, t')} \subseteq C^{*,t}_{v,\ell(v,t')}$ then \ref{F'5} implies that $i\in V(Q_{t,s+1})$.
Moreover, by \eqref{eq: Uits sizess} we have $|U^{t,s+1}_{i}| \ge |V_i^t| - 5k\epsilon^2n/r$. 
Together with \ref{F'3} this implies that $d_{F'_{t'}(v,\ell(v,t')), U^{t,s+1}_i}(v) \geq (1- \epsilon)|U_i^{t,s+1}|/q$. 
Thus (A7)$_{\ref{embedone}}$ holds.
To check (A9)$_{\ref{embedone}}$, note that
for each $u\in U_0^{t,s+1}$,
we have 
$$ d_{G^*_t\cap \Phi^s}(u)
\stackrel{\ref{G1}}{\leq} 4k\Delta j_t^s(u) n/(qr) +\epsilon^{1/9}sn/r  \stackrel{\ref{Q2},\eqref{eq: Jsv size}}{\leq}   \epsilon^{1/10} n.$$
Thus, 
\begin{align*}
|\{ i\in V(Q_{t,s+1}) : d_{\hat{G},V_i^{t,s+1}}(u) \geq d^2m/3 \}| &\geq
|\{ i\in V(Q_{t,s+1}) : d_{G^*_t,V_i^{t,s+1}}(u) \geq d^2m/2 \}|  - 
\frac{ d_{G^*_t\cap \Phi^s}(u) 
}{ d^2m/6} \\
& \hspace{-0.2cm} \stackrel{\eqref{eq: v degree good in G*t}}{\geq} 
(1- 1/k + \sigma/2)r - \frac{ \epsilon^{1/10} n }{d^2m/6}
\stackrel{\eqref{eq: Vi size}}{\geq} (1-1/k + \sigma/3)r.
\end{align*}
This implies that 
$$|\{ i\in V(Q_{t,s+1}) : d_{\hat{G},V_j^{t,s+1}}(u) \geq d^3 m \text{ for all } j\in N_{Q_{t,s+1}}(i) \}|\geq \sigma^2 r.$$
This shows that (A9)$_{\ref{embedone}}$ holds.
Hence, by Lemma~\ref{embedone}, we obtain a function $\psi_{s+1}$ packing $\mathcal{H}_{t,s+1}$ into $\hat{G}\cup F_{t,s+1}\cup \hat{F'}$ and satisfying the following.
\begin{enumerate}[label={\text{\rm(B\arabic*)}}]
\item \label{B1} $\Delta(\psi_{s+1}(\mathcal{H}_{t,s+1})) \le 4k\Delta n/(qr),$ 
\item \label{B2} for each $u\in Res_t\setminus U_0^{t,s+1}$, we have
$d_{\psi_{s+1}(\mathcal{H}_{t,s+1})\cap \hat{G}}(u) \le 10 \Delta \epsilon^{1/8}n/r,$
\item\label{B3} for each $i \in V(Q_{t,s})$, we have $e_{\psi_{s+1}(\mathcal{H}_{t,s+1})\cap \hat{G}}(V^{t,s+1}_i, Res_t) < 10 \epsilon^{1/2} n^2/r^2.$
\end{enumerate}
Moreover, \ref{G3} with \ref{G4} implies that 
$\psi_{s+1}(\mathcal{H}_{t,s+1})$ is edge-disjoint from $\Phi^{s}$, thus the map
$\phi_{s+1}:= \phi_{s}\cup  \psi_{s+1}$ 
packs $\bigcup_{s'=1}^{s+1}\mathcal{H}_{t,s'}$ into $G_{t}\cup\bigcup_{s'=1}^{\kappa/T} F_{t,s'}\cup F'_{t}$.
Now it remains to show that $\phi_{s+1}$ satisfies (G1)$_{s+1}$--(G4)$_{s+1}$.

Consider any vertex $u\in Res_t$. 
If $u\in U_0^{t,s+1}$, then we know that $j_t^{s+1}(u) = j_t^{s}(u)+1$.
Thus  \ref{G1} together with \ref{B1} implies (G1)$_{s+1}$ for the vertex $u$.
If $u\in Res_t \setminus U_0^{t,s+1}$, then  we have $j_t^{s+1}(u) = j_t^{s}(u)$, thus 
\ref{G1} together with \ref{B2} implies (G1)$_{s+1}$.

For each $i\in [r]$, \eqref{eq: def V*G*} implies that the vertices in $V_i\setminus (V^{t,s+1}_i \cup V_i^t) \subseteq V_0^{t,s+1}$ are not incident to any edges in $\Phi^{s+1}\cap G^*_t$. Thus it is easy to see that  \ref{G2} together with \ref{B3} implies (G2)$_{s+1}$.
As $\psi_{s+1}$ packs $\mathcal{H}_{t,s+1}$ into $\hat{G}\cup F_{t,s+1} \cup \hat{F'}$, \eqref{eq: Def G hat F hat} together with \ref{G3} implies (G3)$_{s+1}$.
 Moreover, we have
$$
i^{s+1}_t(v) = \left\{\begin{array}{ll}
i^s_t(v)+ 1 & \text{ if } v\in V_0^{t,s+1},\\
i^s_t(v) & \text{ otherwise.}
\end{array}\right.
$$
Thus, \eqref{eq: Def G hat F hat} together with \ref{G4} and the definition of $\ell(v,t')$ implies (G4)$_{s+1}$.
 
By repeating this for each $s\in [\kappa/T]$ in order, we obtain a function $\phi_{\kappa/T}$ which packs $\mathcal{H}_{t}$ into $G_t\cup F_t\cup F'_t$. By taking the union of such functions over all $t\in [T]$, we obtain a desired function packing $\mathcal{H}$ into $\displaystyle \bigcup_{t\in [T]} G_t\cup F_t\cup F'_t \subseteq G$. This completes the proof.
\end{proof}

The proof of Theorem~\ref{thm:2}, follows almost exactly the same lines as that of  Theorem~\ref{thm:main}, with one very minor difference.
Indeed, the only place where we need the condition that $G$ is almost regular
is when we apply Lemma~\ref{lem: factor decomposition} in Step 1
to obtain (Q1)--(Q5).
Thus to prove Theorem~\ref{thm:2}, we only need to replace the application of Lemma~\ref{lem: factor decomposition} with an application of the following result.
(Note that  (B1) below implies both (Q3) and (Q4).)

\begin{lemma}\label{lem: factor decomposition2}
Suppose $n, q, T \in \mathbb{N}$ with $0< 1/n \ll \epsilon, 1/T, 1/q, \nu \le 1/2$ and $0<1/n \ll \nu < \sigma/2 < 1$ and $\delta= 1/2+\sigma$ and $q$ divides $T$.
Let $G$ be an $n$-vertex multi-graph with edge-multiplicity at most $q$, such that for all $v\in V(G)$ we have  
$d_G(v) \geq q \delta n.$

Then there exists a subset $V'\subseteq V(G)$ with $|V'|\leq 1$ and $|V(G)\backslash V'|$ being even, and there exist pairwise edge-disjoint matchings $F_{1,1},\dots, F_{1,\kappa}, F_{2,1}, \dots, F_{T,\kappa}$ of $G$ with $\kappa = \frac{(\delta + \sqrt{2\delta-1} -\nu)qn}{2T}\pm 1$ satisfying the following. 
\begin{enumerate}
\item[\rm (B1)] For each $(t',i)\in [T]\times[\kappa]$, we have that $V(F_{t',i}) = V(G) \backslash V'$,
\item[\rm (B2)] for all $t'\in [T]$ and $u,v \in V(G)$, we have 
$|\{ i \in [\kappa] :  u \in N_{F_{t',i}}(v)  \}|\leq 1.$
\end{enumerate}
\end{lemma}

The proof of the above lemma is very similar (but simpler) than that of 
Lemma~\ref{lem: factor decomposition}. We proceed as in the proof of Lemma~\ref{lem: factor decomposition} to obtain simple graphs $G^c$ with $\delta(G^c)>\delta n-\nu^2 n$. We let $V'\subseteq V(G)$ be such that $|V'|\leq 1$ and $|V(G)\backslash V'|$ is even. The difference is that we now apply the following result of~\cite{CKO} to each $G_*^c:=G^c[V(G)\backslash V']$
to obtain the desired matchings $M_i^c$:
for every $\alpha>0$,  any sufficiently large $n$-vertex graph with minimum degree $\delta \ge (1/2 + \alpha)n$ contains at least $ (\delta - \alpha n + \sqrt{n(2\delta - n)})/4$ edge-disjoint Hamilton cycles.

\section*{Acknowledgement}

We are grateful to the referee for helpful comments on an earlier version.

  \bigskip

{\footnotesize \obeylines \parindent=0pt
	
	\begin{tabular}{lllll}
        Padraig Condon,  Daniela K\"uhn and Deryk Osthus	&\ & 	Jaehoon Kim       \\
		School of Mathematics &\ & Mathematics Institute		  		 	 \\
		University of Birmingham &\ & University of Warwick 	  			 	 \\
		Birmingham &\ & Coventry                             			 \\
        B15 2TT      &\ & 		CV4 7AL				  				\\
		UK &\ & UK						      \\
	\end{tabular}
}

\begin{flushleft}
	{\it{E-mail addresses}: \tt{\{pxc644, d.kuhn, d.osthus\}@bham.ac.uk}, Jaehoon.Kim.1@warwick.ac.uk}.
\end{flushleft}

\end{document}